\documentclass[11pt,reqno]{amsart}

\usepackage[leqno]{amsmath}
\usepackage{amsthm}
\usepackage{amsfonts, tikz-cd}
\usepackage{amssymb}
\usepackage{eucal}
\usepackage[all]{xy}
\usepackage{mathtools}
\usepackage{stmaryrd}
\usepackage{csquotes}
\usepackage{enumitem}

\setenumerate{itemsep=2pt,parsep=2pt,before={\parskip=2pt}}

\usepackage[colorlinks=true,hyperindex, linkcolor=magenta, pagebackref=false, citecolor=cyan,pdfpagelabels]{hyperref}

\usepackage{xspace}
\usepackage{epsfig,epic,eepic,latexsym,color}
\usepackage[all]{xy}
\usepackage{comment} 

\usepackage[english]{babel}
\usepackage{latexsym}

\newcommand{\nc}{\newcommand}
\nc{\rnc}{\renewcommand}
\rnc{\P}{\mathbf P}
\nc{\R}{\mathbf R}
\rnc{\rm}{\mathrm}

\nc{\C}{\mathbf C}
\nc{\Q}{\mathbf Q}
\nc{\Z}{\mathbf Z}
\nc{\N}{\mathbf N}
\nc{\A}{\mathbf A}

\nc{\an}{\operatorname{an}}

\nc{\htt}{\operatorname{ht}}
\nc{\Nm}{\operatorname{Nm}}
\nc{\Ker}{\operatorname{Ker}}
\nc{\mmod}{\operatorname{mod}}
\nc{\End}{\operatorname{End}}
\nc{\Aut}{\operatorname{Aut}}
\nc{\cont}{\text{cont}}
\nc{\sep}{\text{sep}}
\nc{\Hom}{\mathrm{Hom}}
\nc{\Gal}{\mathrm{Gal}}
\nc{\Spec}{\operatorname{Spec}}
\nc{\Spv}{\operatorname{Spv}}
\nc{\supp}{\text{supp}}
\nc{\rad}{\operatorname{rad}}
\nc{\cal}{\mathcal}
\rnc{\frak}{\mathfrak}

\nc{\RZ}{\operatorname{RZ}}

\rnc{\t}{\tau}
\nc{\mm}{\pmb{\mu}}
\rnc{\a}{\alpha}
\nc{\n}{\mathfrak n}
\nc{\m}{\mathfrak m}
\nc{\mfs}{\mathfrak s}

\nc{\p}{\mathfrak p}
\nc{\q}{\mathfrak q}

\nc{\Sym}{\operatorname{Sym}}
\nc{\codim}{\operatorname{codim}}
\nc{\rk}{\operatorname{rk}}
\nc{\GL}{\operatorname{GL}}
\nc{\SL}{\operatorname{SL}}
\nc{\Lie}{\operatorname{Lie}}
\nc{\Ind}{\operatorname{Ind}}
\nc{\Div}{\underline{Div}}
\nc{\Pic}{\mathbf{Pic}}
\nc{\uPic}{\underline{ \mathbf{Pic}}}
\nc{\rH}{\mathrm{H}}
\nc{\Spf}{\operatorname{Spf}}
\nc{\Frac}{\operatorname{Frac}}
\nc{\colim}{\operatorname{colim}}
\nc{\Spa}{\operatorname{Spa}}
\nc{\Sp}{\operatorname{Sp}}
\nc{\Tor}{\operatorname{Tor}}

\rnc{\an}{\operatorname{an}}

\nc{\xr}{\xrightarrow}

\nc{\eps}{\epsilon}
\nc{\ov}{\overline}
\nc{\ud}{\underline}
\nc{\wdh}{\widehat}

\nc{\F}{\mathcal F}
\nc{\G}{\mathcal G}
\nc{\E}{\mathcal E}

\nc{\X}{\mathfrak X}
\nc{\sZ}{\mathfrak Z}
\nc{\Y}{\mathfrak Y}
\nc{\T}{\mathfrak T}
\nc{\sU}{\mathfrak U}
\nc{\V}{\mathfrak V}

\nc{\LL}{\mathcal{L}}
\rnc{\S}{\mathfrak S}

\nc{\ra}{\rangle}

\nc{\os}{\overset}

\rnc{\O}{\mathcal O}
\nc{\J}{\mathcal J}
\theoremstyle{definition}

\newtheorem{thm}{Theorem}[subsection]
\newtheorem{lemma}[thm]{Lemma}
\newtheorem{defn}[thm]{Definition}

\newtheorem{Set-up}[thm]{Set-up}
\newtheorem{prop}[thm]{Proposition}

\newtheorem{rmk}[thm]{Remark}

\newtheorem{exmpl}[thm]{Example}
\newtheorem{cor}[thm]{Corollary}

\begin{document}

\title{Quotients Of Admissible Formal Schemes and Adic Spaces by Finite Groups}
\bibliographystyle{halpha-abbrv}
\author{Bogdan Zavyalov}
\maketitle

\begin{abstract}
In this paper we give a self-contained treatment of finite group quotients of admissible (formal) schemes and adic spaces that are locally topologically finite type over a locally strongly noetherian adic space.  
\end{abstract}
\tableofcontents

\section{Introduction}
\subsection{Overview}
This paper studies ``geometric quotients'' in different geometric setups. Namely, we work in three different situations: flat and locally finite type schemes over a typically non-noetherian valuation ring, admissible formal schemes over a complete microbial valuation ring (see Definition~\ref{defn:formal-microbial-valuation}), and locally topologically finite type adic spaces over a locally strongly noetherian analytic adic space. These $3$ different contexts occupy Chapters~\ref{section:alg}, \ref{section:formal}, and \ref{section:adic} respectively.\medskip

The motivation to study these quotients comes from our paper~\cite{Z1}, where we show a refined version of Temkin's local alteration theorem. Our result roughly says that any smooth rigid-space $X$ over an algebraically closed non-archimedean field $C$ locally admits a formal $\O_C$-model $\X$ such that $\X$ is a quotient of a polystable admissible formal $\O_C$-model $\X'$ by a finite group $G$ (acting freely on the generic fiber). This refined uniformization result is an important technical input in the author's proof of $p$-adic Poincar\'e Duality in \cite{Zav-duality}. \smallskip

The actual formulation of this uniformization result is quite technical, and we refer to \cite[Theorem 1.4]{Z1} for the precise formulation. We only mention that, in order to formulate {\it and} prove this theorem, we had to make sure that a quotient of an admissible formal $\O_C$-scheme by an $\O_C$-action of a finite group exists as an admissible formal $\O_C$-scheme. This result seems to be missing in the literature, the main difficulty being that the ring $\O_C$ is never noetherian.

\subsubsection{Scheme Case} Before we deal with quotients of formal schemes and adic spaces, we first discuss quotients of schemes over the base scheme $S=\Spec \O_C$. Even this question is already non-trivial and demonstrates an important source of difficulties in the question of studying quotient spaces in the non-noetherian situation. The same difficulty will arise in every other setup treated in this paper. \smallskip

We fix an $S$-scheme $X$ with an $S$-action of a finite group $G$. Then a standard argument constructs the quotient $X/G$ as an $S$-scheme (under some assumptions); this is carried out in \cite[Exp. V, \textsection 1]{SGA1} (see also Definition~\ref{defn:geometric-quotient} and Theorem~\ref{thm:alg-main}). However, the question of whether, for a finite type $S$-scheme $X$, the quotient $X/G$ is of finite type is quite non-trivial. \smallskip

To explain the main issue, we briefly recall what happens in the classical situation of a finite type $R$-scheme $X$ with an $R$-action of a finite group $G$ for some {\it noetherian} ring $R$. Under some mild assumptions\footnote{In particular, if $X$ is quasi-projective over $R$.} on $X$, one can rather easily reduce to the affine situation $X=\Spec A$, where the main work is to show that $A^G$ is of finite type over $R$. This is done in two steps: one firstly checks that $A$ is a finite $A^G$-module, and then one uses the Artin-Tate Lemma:

\begin{lemma}\label{lemma:intro-Artin-Tate}\cite[Proposition 7.8]{AM} Let $R$ be a noetherian ring, and $B\subset C$ an inclusion of $R$-algebras. Suppose that $C$ is a finite type $R$-algebra, and $C$ is a finite $B$-module. Then $B$ is finitely generated over $R$.
\end{lemma}

One may think that probably the Artin-Tate lemma can hold, more generally, over a non-noetherian base $R$ if $C$ is finitely presented over $R$. However, this is not the case and the Artin-Tate Lemma fails over any non-noetherian base: 

\begin{exmpl}\label{intro:Artin-Tate-fails} Let $R$ be a non-noetherian ring with an ideal $I$ that is not finitely generated. Consider the $R$-algebra $C\coloneqq R[\varepsilon]/(\varepsilon^2)$, and the $R$-subalgebra $B=R\oplus I\varepsilon$. So $C$ is a finitely presented $R$-algebra, and $C$ is finite as a $B$-module since it is already finite over $R$. However, $B$ is not finitely generated $R$-algebra as that would imply that $I$ is a finitely generated ideal.
\end{exmpl}

Example~\ref{intro:Artin-Tate-fails} shows that the strategy should be appropriately modified in the non-noetherian situations like schemes over $\O_C$. We deal with this issue by proving a weaker version of the Artin-Tate Lemma over any valuation ring $k^+$ (see Lemma~\ref{lemma:valuation-Artin-Tate}). That proof crucially exploits features of finitely generated algebras over a valuation ring. We emphasize that our argument does use the $k^+$-flatness assumption in a serious way; we do not know if the quotient of a finitely presented affine $k^+$-scheme by a finite group action is finitely presented (or finitely generated) over $k^+$. 

\subsubsection{Formal Schemes and Adic Spaces} The strategy above can be appropriately modified to work in the world of admissible formal schemes and strongly noetherian adic spaces. In both situations, the main new input is a corresponding version of the Artin-Tate Lemma (see Lemma~\ref{lemma:formal-Artin-Tate} and Lemma~\ref{lemma:adic-Artin-Tate}). However, there are issues that are not seen in the scheme case. We explain a few of the main new technical difficulties that arise while proving the result in the world of adic spaces. \medskip

Compared to the affine (formal) schemes, the underlying topological space of an affinoid space $\Spa(A, A^+)$ is harder to express in terms of the pair $(A, A^+)$. It is a set of all {\it valuations} on $A$ with corresponding continuity and integrality conditions. In particular, even if one works with rigid spaces over a non-archimedean field $K$, one has to take into account points of higher rank that do not have any immediate geometric meaning. Hence, it takes extra care to identify $\Spa(A^G, A^{+, G})$ with $\Spa(A, A^+)/G$ even on the level of {\it underlying topological spaces}. \medskip

Furthermore, the notion of a topologically finite type (resp. finite) morphism of Tate-Huber pairs is more subtle than its counterpart in the algebraic setup for two different reasons.  Firstly, it has a topological aspect that takes some care to work with. Secondly, the notion involves conditions on {\it both} $A$ and $A^+$ (see Definition~\ref{defn:huber-topologically-finite-type} and Definition~\ref{defn:Huber-finite}). Usually, $A^+$ is non-noetherian, so it requires some extra work to check the relevant condition on it.

\subsubsection{Generality} In the case of adic spaces, we consider spaces that are locally topologically finite type over a strongly noetherian analytic adic space in Section~\ref{section:adic}. One reason for this level of generality is to include adic spaces that are topologically finite type over $\Spa(k, k^+)$ for a microbial valuation ring $k^+$ (see Definition~\ref{defn:formal-microbial-valuation}). These spaces naturally arise while studying fibers of morphisms of rigid spaces\footnote{Considered as adic spaces} $X \to Y$ over points of $Y$ of {\it higher rank}. We think that the category of strongly noetherian analytic adic spaces is the natural one to consider\footnote{It may be also reasonable to consider some class of {\it non-analytic} adic spaces. However, it is not clear what should be the correct uniform condition on an adic space that would imply the Artin-Tate lemma in both analytic and non-analytic non-noetherian setups. Since we never need to use non-analytic adic spaces in our intended applications, we prefer to work only with analytic adic spaces in this paper.}. One of its advantages is that it contains both topologically finite type morphisms and morphisms coming from the (not necessary finite) base field extension in rigid geometry. \medskip

In the case of formal schemes, the results of Section~\ref{section:formal} are written in the generality of admissible formal schemes over a complete, microbial valuation ring $k^+$ (see Definition~\ref{defn:formal-microbial-valuation}). We want to point out that Appendix~\ref{appendix-A} contains versions of the main results of  Section~\ref{section:formal} for a topologically universally adhesive base (see Definition~\ref{defn:adhesive-admissible}). These results are more general and include {\it both} the cases of formal schemes topologically finite type (and flat) over some $k^+$ and noetherian formal schemes. However, we prefer to formulate and prove the results in the main body of the paper for admissible formal schemes over $k^+$ since it simplifies the exposition a lot. We only refer to Appendix~\ref{appendix-A} for the necessary changes that have to be made to make the arguments work in the more general adhesive situation. \medskip

Likewise, Appendix~\ref{appendix-A} has versions of the results of Section~\ref{section:alg} over a universally adhesive base (see Definition~\ref{defn:adhesive}). But we want to point out that a valuation ring $k^+$ is universally adhesive only if it is microbial (see Lemma~\ref{lemma:adhesive=microbial}), so the results of Appendix~\ref{appendix-A} do not fully subsume the results of Section~\ref{section:valuation}.




\subsection{Comparison with \cite{Han}} While writing this paper, we found that similar results for adic spaces were already obtained in \cite{Han}. We briefly discuss the main similarities and differences in our approaches. \medskip

David Hansen separately discusses two different situations: rigid spaces~\footnote{Defined as locally topologically finite type adic spaces over $\Spa(K, \O_K)$} over a non-archimedean field $K$, and general analytic adic spaces. In the former case, he shows that (under some assumptions on $X$) $X/G$ exists as a rigid space over $K$ for any finite group $G$. He crucially uses \cite[Proposition 6.3.3/3]{BGR} that states that for a $K$-affinoid affinoid $A$ with a $K$-action of a finite group $G$, the ring of invariants $A^G$ is a $K$-affinoid algebra. The proof of this result uses analytic input: the Weierstrass preparation theorem. In the latter case, he shows that the quotient of $X$ exists as an analytic adic space if the order of $G$ is invertible in $\O_X(X)$. The argument there is based on an averaging trick, so it uses the invertibility assumption in order to be able to divide by $\# G$. We note that if $X$ is a perfectoid space over a perfectoid field, he can drop this invertibility assumption by some other argument. The whole point of the latter case is to be able to work with ``big'' adic spaces such as perfectoid spaces.  \medskip

In contrast with Hansen's approach, our methods neither use any non-trivial input from non-archimedean analysis, nor the averaging trick. What we do is try to imitate the classical algebraic argument based on the Artin-Tate Lemma in the setup of strongly noetherian adic spaces. More precisely, we show that if $X$ is a locally topologically finite type adic space (with some other conditions) over a locally strongly noetherian adic space $S$ with an $S$-action of a finite group $G$ then the quotient $X/G$ exists as a locally topologically finite type adic $S$-space. Our result does not recover Hansen's result as we do not allow ``big'' adic spaces such as perfectoid spaces, but it proves a stronger statement in the case of adic spaces locally of finite type over a locally strongly noetherian $S$ as we do not have any assumptions on the order of $G$. Moreover, even in the case of rigid spaces, it gives a new proof of the existence of $X/G$ as a rigid space that does not use much of analytic theory. 

\subsection{Our results} We firstly study the case of a flat, locally finite type scheme $X$ over a valuation ring $k^+$ and a $k^+$-action of a finite group $G$. We show that $X/G$ exists as a flat, locally finite type $k^+$-scheme under a mild assumption on $X$:

\begin{thm}\label{thm:intro-main-scheme}(Theorem~\ref{thm:alg-main} and Theorem~\ref{thm:valuation-main}) Let $X$ be a flat, locally finite type $k^+$-scheme with a $k^+$-action of a finite group $G$. Suppose that each point $x\in X$ admits an affine neighborhood $V_x$ containing $G.x$. Then $X/G$ exists as a flat, locally finite type $k^+$-scheme. Moreover, it satisfies the following properties:
\begin{enumerate}
\item $\pi\colon X \to X/G$ is universal in the category of $G$-invariant morphisms to locally ringed $S$-spaces. 
\item $\pi: X \to X/G$ is a finite, finitely presented morphism (in particular, it is closed). 
\item Fibers of $\pi$ are exactly the $G$-orbits.
\item\label{thm:intro-main-scheme-4} The formation of the quotient $X/G$ commutes with flat base change (see Theorem~\ref{thm:alg-main}(\ref{thm:alg-main-4}) for the precise statement).
\end{enumerate}
\end{thm}

We then consider quotients of admissible formal schemes $\X$ over a complete microbial valuation ring $k^+$ by a $k^+$-action of a finite group $G$. Under similar conditions, we show that $\X/G$ exists as an admissible formal $k^+$-scheme and satisfies the expected properties:

\begin{thm}\label{thm:intro-main-formal}(Theorem~\ref{thm:formal-main}) Let $\X$ be an admissible formal $k^+$-scheme with a $k^+$-action of a finite group $G$. Suppose that each point $x\in \X$ admits an affine neighborhood $\V_x$ containing $G.x$. Then $\X/G$ exists as an admissible formal $k^+$-scheme. Moreover, it satisfies the following properties:
\begin{enumerate}
\item $\pi\colon \X \to \X/G$ is universal in the category of $G$-invariant morphisms to topologically locally ringed spaces over $\S$. 
\item $\pi: \X \to \X/G$ is a surjective, finite, topologically finitely presented morphism (in particular, it is closed). 
\item Fibers of $\pi$ are exactly the $G$-orbits.
\item The formation of the geometric quotient commutes with flat base change (see Theorem~\ref{thm:formal-main}(\ref{thm:formal-main-4}) for the precise statement).
\end{enumerate}
\end{thm}

Finally, we consider the case of locally topologically finite type adic spaces over a locally strongly noetherian adic space.

\begin{thm}\label{thm:intro-main-adic}(Theorem~\ref{thm:adic-main}) Let $S$ be a locally strongly noetherian analytic adic space (see Definition~\ref{defn:strongly-noetherian-space}), and $X$ a locally topologically finite type adic $S$-space with an $S$-action of a finite group $G$. Suppose that each point $x\in X$ admits an affinoid open neighborhood $V_x$ containing $G.x$. Then $X/G$ exists as a locally topologically finite type adic $S$-space. Moreover, it satisfies the following properties:
\begin{enumerate}
\item $\pi\colon X \to X/G$ is universal in the category of $G$-invariant morphisms to topologically locally $v$-ringed $S$-spaces (see Definition~\ref{defn:valuative-spaces}). 
\item $\pi: X \to X/G$ is a finite, surjective morphism (in particular, it is closed). 
\item Fibers of $\pi$ are exactly the $G$-orbits.
\item The formation of the geometric quotient commutes with flat base change (see Theorem~\ref{thm:adic-main}(\ref{thm:adic-main-4}) for the precise statement).
\end{enumerate}
\end{thm} 

The condition that each point $x\in X$ admits an affinoid open neighborhood $V_x$ containing $G.x$ is a much milder in the adic world than in the scheme world. For example, we show that it is automatic if $X$ is a separated rigid-analytic space. This, in particular, implies that a quotient of a separated rigid-analytic space by a finite group action always exists as a rigid-analytic space. In case of a free finite group action, a similar result has been previously obtained in \cite[Theorem 5.1.1]{conrad-temkin} in the world of Berkovich spaces. 

\begin{lemma}\label{lemma:intro-mild-assumption}(Lemma~\ref{lemma:mild-assumption}) Let $K$ be a non-archimedean field with the residue field $k$, $X$ a separated, locally finite type adic $\Spa(K, \O_K)$-space, and $\{x_1, \dots, x_n\}$ is a finite set of points of $X$. Then there is an open affinoid subset $U\subset X$ containing all $x_i$.
\end{lemma}

\begin{rmk} It is reasonable to expect that the assumption of Theorem~\ref{thm:intro-main-adic} is automatic as long as $X$ is $S$-separated (see Remark~\ref{rmk:referee}). However, the proof of this claim would seem to require a generalization of the main results of \cite{temkin-local} to more general adic spaces. This is beyond the scope of this paper.
\end{rmk}

The natural question is whether these quotients commute with certain functors like formal completion, analytification, and adic generic fiber. We show that this is indeed the case, i.e. the formation of the geometric quotients commutes with the functors mentioned above whenever they are defined. We informally summarize the results below:

\begin{thm}\label{thm:intro-comparison}(Theorem~\ref{thm:comparison-formal-alg}, Theorem~\ref{thm:comparison-formal-adic}, and Theorem~\ref{thm:comparison-alg-adic}) 
\begin{enumerate}
    \item Let $k^+$ be a microbial valuation ring, and $X$ a flat, locally finite type $k^+$-scheme with a $k^+$-action of a finite group $G$. Suppose $X$ satisfies the assumption of Theorem~\ref{thm:intro-main-scheme}. The natural morphism $\wdh{X}/G \to \wdh{X/G}$ is an isomorphism.  
    \item Let $k^+$ be a complete, microbial valuation ring with fraction field $k$, and $\X$ an admissible formal $k^+$-scheme with a $k^+$-action of a finite group $G$. Suppose $\X$ satisfies the assumption of Theorem~\ref{thm:intro-main-formal}. The natural morphism $\X_k/G \to 
    (\X/G)_k$ is an isomorphism. 
    \item Let $K$ be a complete rank-$1$ valued field, and $X$ a locally finite type $K$-scheme with a $K$-action of a finite group $G$. Suppose $X$ satisfies the assumption of Theorem~\ref{thm:alg-main}. The natural morphism $X^{\rm{an}}/G \to (X/G)^{\rm{an}}$ is an isomorphism.  
\end{enumerate}
\end{thm}

\subsection*{Acknowledgements}
We are grateful to B. Bhatt, B. Conrad and S. Petrov for fruitful conversations. We heartfully thank M.\,Temkin for suggesting the argument of Lemma~\ref{lemma:mild-assumption}. We express additional gratitude to B. Conrad for reading the first draft of this paper and making lots of suggestions on how to improve the exposition of this paper. We are also very grateful to the anonymous referee who read the paper very carefully and made lots of useful suggestions.

\section{Quotients of Schemes}\label{section:alg}

\subsection{Review of Classical Theory}
We review the classical theory of quotient of schemes by an action of a finite group. This theory was developed in \cite[Exp. V, \textsection 1]{SGA1}. We review the main results from there, and present some proofs in a way that will be useful for our later purposes. This section is mostly expository. \medskip

For the rest of this section, we fix a base scheme $S$. 

\begin{defn}\label{defn:geometric-quotient} Let $G$ be a finite group, and $X$ a locally ringed space over $S$ with a right $S$-action of $G$. The {\it geometric quotient} $X/G=(|X/G|, \O_{X/G}, h)$ consists of:
\begin{itemize}\itemsep0.5em
\item the topological space $|X/G|\coloneqq |X|/G$ with the quotient topology. We denote by $\pi:|X| \to |X/G|$ the natural projection,
\item the sheaf of rings $\O_{X/G}\coloneqq (\pi_*\O_X)^G$,
\item the morphism $h:X/G \to S$ defined by the pair $(h, h^{\#})$, where $h:|X|/G \to S$ is the unique morphism induced by $f\colon X \to S$ and $h^{\#}$ is the natural morphism 
\[
\O_{S} \to h_*\left(\O_{X/G}\right)=h_*\left(\left(\pi_*\O_{X}\right)^G\right)=\left(h_*\left(\pi_*\O_{X}\right)\right)^G=\left(f_*\O_X\right)^G
\]
that comes from $G$-invariance of $f$.
\end{itemize}
\end{defn}

We note that $X/G$ is, a priori, only a ringed space. In the lemma below, we show that it is actually always a locally ringed space:

\begin{lemma}\label{lemma:geometric-quotient-locally-ringed} Let $X$ be a locally ringed space over $S$ with a right $S$-action of a finite group $G$. Then $X/G$ is a locally ringed space, and $\pi\colon X \to X/G$ is a map of locally ringed spaces (so $X/G \to S$ is too). 
\end{lemma}
\begin{rmk} This lemma must be well-known, but we do not know any particular reference. We decided to include the proof as it will be a convenient technical tool for us. \smallskip

Lemma~\ref{lemma:geometric-quotient-locally-ringed} allows us to construct quotients entirely in the category of locally ringed spaces and not merely in the category of all ringed spaces. The main technical issue with the category of ringed spaces is that locally ringed spaces do not form a full subcategory of it. 
\end{rmk}
\begin{proof}[Proof of Lemma~\ref{lemma:geometric-quotient-locally-ringed}]
We note that the action of $G$ induces a family of ring isomorphisms
    \[
    \O_{X}(g(U)) \xr{a^U_g} \O_X(U)
    \]
for $g\in G$ and open $U\subset X$. Furthermore, for any inclusion of open subsets $V\subset U\subset X$, the diagram
    \begin{equation}\label{eqn:action}
    \begin{tikzcd}
        \O_{X}\left(g\left(U\right)\right) \arrow{d}{r^{g(U)}_{g(V)}}\arrow{r}{a^U_g} & \O_{X}(U) \arrow{d}{r^U_V} \\
        \O_X\left(g\left(V\right)\right) \arrow{r}{a^V_g} & \O_X(V)
    \end{tikzcd}
    \end{equation}
is commutative. In particular, $G$ acts on $\O_X(U)$ for any $G$-stable open $U\subset X$. We describe the stalk $\O_{X/G, \ov{x}}$ for a point $\ov{x}\in X/G$ with a lift $x\in X$ as follows: 
\begin{equation}\label{eqn:local-ring-below}
    \O_{X/G, \ov{x}} \simeq \colim_{\{x\in U\subset X \ | \ g(U)=U \ \forall g\in G\}} \O_X(U)^G.
\end{equation}

That being said, we wish to show that 
\[
    \pi^\sharp_{\ov{x}}\colon \O_{X/G, \ov{x}} \to \O_{X, x}
\]
is a local homomorphism of local rings. This is equivalent to saying that\footnote{In what follows, we slightly abuse the notation and write $T\cap \O_{X/G, \ov{x}}$ as an abbreviation for $\left(\pi_{\ov{x}}^\sharp\right)^{-1}\left(T\right)\subset \O_{X/G, \ov{x}}$ for any subset $T\subset \O_{X, x}$.} $\m_{x} \cap \O_{X/G, \ov{x}}$ is the unique maximal ideal in $\O_{X/G, \ov{x}}$ or, equivalently, that any $f\in \O_{X, x}^\times \cap \O_{X/G, \ov{x}}$ lies in $\O_{X/G, \ov{x}}^\times$. \smallskip

We use $(\ref{eqn:local-ring-below})$ to find a $G$-stable open $x\in U \subset X$ such that $f$ comes from an element $f_U\in \O_{X}(U)^G$. The condition that $f$ becomes invertible in $\O_{X, x}$ means that there is an open $x\in V \subset U$ and a function $k_V\in \O_{X}(V)$ such that 
\[
k_V\cdot f_U|_V=1 \in \O_{X}(V).
\]
We set $k_{g(V)}\coloneqq (a^V_g)^{-1}(k_V) \in \O_{X}(g(V))$ for $g\in G$. Then $G$-invariance of $f_U$ and (\ref{eqn:action}) imply that $k_{g(V)}\cdot f_U|_{g(V)}=1\in \O_{X}(g(V))$. Uniqueness of the inverse element and the sheaf axioms imply that $k_{g(V)}$ glue to a section
\[
k\in \O_X\left(W\right),
\]
where $W=\cup_{g\in G} g(V)$ is a $G$-stable open subset of $X$. Then $k\cdot f_{U}|_W=1\in \O_{X}(W)$ since this can be checked locally. In particular, $k$ is an inverse in $f_U|_W$, so $G$-invariance of $f_U|_{W}$ implies $G$-invariance of $k$. In particular, $f_{U}|_W \in \left(\O_{X}\left(W\right)^G\right)^\times$ implying that $f\in \O_{X/G, \ov{x}}^\times$.
\end{proof}

\begin{rmk}\label{rmk:universal-locally rings} It is straightforward to see that the pair $(X/G, \pi)$ is a universal object in the category of $G$-invariant morphisms to locally ringed spaces over $S$.
\end{rmk}

\begin{rmk}\label{rmk:Hironaka-example} We warn the reader that $X/G$ might not be a scheme even if $S=\Spec \C$ and $X$ is a smooth and proper, connected $\C$-scheme with a $\C$-action of $G=\Z/2\Z$. Namely, Hironaka's example \cite[Example 5.3.2]{olsson-stacks} is a smooth and proper, connected $3$-fold over $\C$ with a $\C$-action of $\Z/2\Z$ such that there is an orbit $G.x$ that is not contained in any open affine subscheme $U \subset X$. Lemma~\ref{lemma:Hiranaka-example} below implies that $X/G$ is not a scheme.
\end{rmk}

\begin{lemma}\label{lemma:alg-open-G-orbit-preserved} Let $X$ be an $S$-scheme with an $S$-action of a finite group $G$. Suppose that each point $x\in X$ admits an open affine subscheme $V_x$ that contains the orbit $G.x$. Then the same holds with $X$ replaced by any $G$-stable open subscheme $U \subset X$.
\end{lemma}
\begin{proof}
Let $x$ be a point in $U$, and $V_x$ an open affine in $X$ that contains $G.x$. Consider $W_x\coloneqq U\cap V_x$ that is an open (possibly non-affine) neighborhood of $x\in U$ containing $G.x$. It suffices to show the stronger claim that {\it any} finite set of points in $W_x$ is contained in an open affine. This follows from \cite[Corollaire 4.5.4]{EGA2}  as $W_x$ is an open subscheme inside the affine scheme $V_x$\footnote{To use this result, we recall that the structure sheaf $\O_Y$ is ample on any affine scheme $Y$.}.
\end{proof}

\begin{lemma}\label{lemma:Hiranaka-example} Let $R$ be a noetherian ring, and $X$ a separated, finite type $R$-scheme with an $R$-action of a finite group $G$. Suppose that there is a point $x\in X$ such that the orbit $G.x$ is not contained in any open affine subscheme $U\subset X$. Then $X/G$ is a not a scheme.
\end{lemma}
\begin{rmk} Lemma~\ref{lemma:Hiranaka-example} must have been known to experts for a long time. However, we are not aware of any reference for this fact. For example, \cite[Exp. V, Proposition 1.8]{SGA1} discusses only a (rather straightforward) statement that it is impossible for $X/G$ to be a scheme {\it and} for $\pi\colon X \to X/G$ to be affine\footnote{Affineness of $X \to X/G$ is part of the definition of an ``admissible'' action of $G$ on $X$ introduced in \cite[Exp. V, Definition 1.7]{SGA1}.}. We strengthen the result and show that $X/G$ is not a scheme without the affineness requirement on $\pi$. 
\end{rmk}
\begin{proof}
Suppose that $X/G$ is an $R$-scheme, and consider the image $\ov{x}\coloneqq \pi(x) \in X/G$. It admits an affine neighborhood $\ov{U}\subset X/G$; this defines an open $G$-stable subscheme $U\coloneqq \pi^{-1}(\ov{U})\subset X$ containing the orbit $G.x$. \smallskip

Now we note that the morphism $\pi|_U\colon U \to \ov{U}$ is quasi-finite and separated. Indeed, it is separated of finite type since $U$ is separated of finite type over $R$ and $\ov{U}$ is separated; its fibers are finite by construction. Therefore, Zariski's main theorem \cite[Proposition 18.12.12]{EGA4_4} implies that $\pi|_{U}$ is quasi-affine, i.e. the natural morphism
\[
U \to \Spec \O_U(U)
\]
is a quasi-compact open immersion. We note that $\Spec \O_U(U)$ naturally admits an action of the group $G$ induced by the action of $G$ on $\O_U$. Trivially, any point $y \in \Spec \O_U(U)$ admits an affine neighborhood containing $G.y$. Thus, Lemma~\ref{lemma:alg-open-G-orbit-preserved} applied to $\Spec \O_U(U)$ and its open subscheme $U$ implies that the same holds for $U$. As a result, the orbit $G.x$ is contained in some open affine subscheme of $X$.  
\end{proof}

Definition~\ref{defn:geometric-quotient} is useless unless we can verify that $X/G$ is a scheme if $X$ is. The main goal of the rest of the section is to review when this is the case under some (mild) assumptions on $X$.\smallskip

We start with the case of an affine scheme $X=\Spec A$ and an affine scheme $S=\Spec R$. Then the natural candidate for the geometric quotient is $Y=\Spec A^G$. There is an evident $G$-invariant $S$-map $p\colon X \to Y$ that induces a commutative triangle
\[
\begin{tikzcd}
X \arrow{d}{\pi} \arrow{dr}{p} & \\
X/G \arrow{r}{\phi} & Y.
\end{tikzcd}
\]
We wish to show that $\phi$ is an isomorphism. Before doing this, we need to recall certain (well-known) properties of $G$-invariants. We include some proofs for the convenience of the reader.

\begin{lemma}\label{lemma:alg-finite-integral} Let $A$ be an $R$-algebra with an $R$-action of a finite group $G$. Then 
\begin{enumerate}
    \item\label{lemma:alg-finite-integral-1} the inclusion $A^G \to A$ is integral. In particular, the morphism $\Spec A \to \Spec A^G$ is closed. 
    \item\label{lemma:alg-finite-integral-2} $\Spec A \to \Spec A^G$ is surjective, the fibers are exactly $G$-orbits.
    \item\label{lemma:alg-finite-integral-3} If $A$ is of finite type over $R$. Then $A^G \to A$ is finite. 
\end{enumerate}
\end{lemma}
\begin{proof}
This is \cite[Expose V, Proposition 1.1(i), (ii) and Corollaire 1.5]{SGA1}. We also point out that the results follow from \cite[Exercise 5.12, 5.13]{AM}, and the observation that an integral, finite type morphism is finite.
\end{proof}

\begin{rmk}\label{rmk:no-finite-type} We warn the reader that Lemma~\ref{lemma:alg-finite-integral} does not imply that, for a finite type $R$-algebra $A$, $A^G$ is of finite type over $R$ (since we allow non-noetherian $R$ as needed later).  
\end{rmk}

\begin{lemma}\label{lemma:invariants-commute-flat-base-change}  Let $R$ be a ring and $A$ an $R$-algebra with an $R$-action of a finite group $G$. Then the formation of invariants $A^G$ commutes with flat base change, i.e. for any flat $R$-algebra morphism $A^G \to B$ the natural homomorphism $B \to (B\otimes_{A^G} A)^G$ is an isomorphism. 
\end{lemma}
\begin{proof}
The proof is outlined just after \cite[Exp. V, Proposition 1.9]{SGA1}. 
\end{proof}

\begin{prop}\label{prop:alg-geometric-quotient-affine} Let $X=\Spec A$ be an affine $R$-scheme with an $R$-action of a finite group $G$. Then the natural map $\phi\colon X/G \to Y=\Spec A^G$ is an $R$-isomorphism of locally ringed spaces. In particular, $X/G$ is an $R$-scheme.
\end{prop}
\begin{proof} This is shown in \cite[Exp. V, Proposition 1.1(iv)]{SGA1}. We review this argument here as this type of reasoning will be adapted to more sophisticated situtions later in the paper. \medskip

{\it Step 1. $\phi$ is a homeomorphism}: We note that Lemma~\ref{lemma:alg-finite-integral} ensures that $p\colon X \to \Spec A^G$ is a closed, surjective map with fibers being exactly $G$-orbits. Thus, $\pi \colon X \to X/G$ and $p\colon X \to \Spec A^G$ are both topological quotient morphisms with the same fibers (namely, $G$-orbits). So the induced map $f$ is clearly a homeomorphism.\medskip 

{\it Step 2. $\phi$ is an isomorphism of locally ringed spaces}: We use Lemma~\ref{lemma:alg-finite-integral} again to check that the morphism of sheaves $\phi^\#\colon \O_{Y} \to \phi_*\O_{X/G}$ is an isomorphism. Using the base of basic affine opens in $Y$, it suffices to show that the map
\[
\left(A^G\right)_f\to \left(A_f\right)^G\simeq \left(A\otimes_{A^G} \left(A^G\right)_f\right)^G
\]
is an isomorphism for any $f\in A^G$. This follows from Lemma~\ref{lemma:invariants-commute-flat-base-change} as $(A^G)_f$ is $A^G$-flat.
\end{proof}

Now we want to discuss when $X/G$ exists as a scheme in the global set-up without a separatedness assumption. Roughly, we want to cover $X$ by $G$-stable affines and then deduce the claim from Proposition~\ref{prop:alg-geometric-quotient-affine}. In order to do this, we need the following lemma:

\begin{lemma}\label{lemma:alg-G-basis} Let $X$ be an $S$-scheme with an $S$-action of a finite group $G$. Suppose that for any point $x\in X$ there is an open affine subscheme $V_x \subset X$ that contains the orbit $G.x$. Then each point $x\in X$ has a $G$-stable open affine neighborhood $U_x \subset X$.
\end{lemma}
\begin{proof}
The proof is outlined just after \cite[Exp V, Proposition 1.8]{SGA1}, we recall the key steps here. Firstly, Lemma~\ref{lemma:alg-open-G-orbit-preserved} ensures that one can reduce to the case of an affine base $S=\Spec R$. Then one shows the claim for a separated $X$, in which case $U_x\coloneqq \bigcap_{g\in G} g(V_x)$ is affine and does the job. In general, Lemma~\ref{lemma:alg-open-G-orbit-preserved} guarantees that one can replace $X$ with the {\it separated} open subscheme $\bigcap_{g\in G} g(V_x)$ and reduce to the separated case. 
\end{proof}

We recall one case where the condition of Lemma~\ref{lemma:alg-G-basis} is satisfied. 

\begin{prop}\label{prop:example-alg} Let $\phi \colon X \to S$ be a locally quasi-projective\footnote{I.e. there exists an open covering $S=\cup V_j$ such that each $\phi^{-1}(V_j)\to V_j$ factors through a quasi-compact immersion $\phi^{-1}(V_j) \to \P^N_{V_j}$ for some $N$.} $S$-scheme with an $S$-action of a finite group $G$. Then every point $x\in X$ admits an affine neighborhood containing the orbit $G.x$.
\end{prop}
\begin{proof}
The statement is local on $S$, so we may and do assume that $S=\Spec R$ is affine and there is a quasi-compact  $R$-immersion $X \subset \P^N_R$. Then it suffices to show a stronger claim that {\it any} finite set of points is contained in an open affine. This is shown in \cite[Corollaire 4.5.4]{EGA2}.
\end{proof}

Now, we are ready to explain the main existence result \cite[Exp V, Proposition 1.8]{SGA1}. For later needs, we give a slightly different proof. 

\begin{thm}\label{thm:alg-main} Let $X$ be an $S$-scheme with an $S$-action of a finite group $G$. Suppose that each point $x\in X$ admits an affine neighborhood $V_x$ containing $G.x$. Then $X/G$ is an $S$-scheme. Moreover, it satisfies the following properties:
\begin{enumerate}
\item $\pi\colon X \to X/G$ is universal in the category of $G$-invariant morphisms to locally ringed $S$-spaces. 
\item $\pi: X \to X/G$ is an integral, surjective morphism (in particular, it is closed). The morphism $\pi$ is finite if $X$ is locally of finite type over $S$.
\item Fibers of $\pi$ are exactly the $G$-orbits.
\item\label{thm:alg-main-4} The formation of the geometric quotient commutes with flat base change, i.e. for any flat morphism $Z\to X/G$, the geometric quotient $(X \times_{X/G} Z)/G$ is a scheme, and the natural morphism $(X \times_{X/G} Z)/G \to Z$ is an isomorphism.
\end{enumerate}
\end{thm} 
\begin{proof}
{\it Step 1. $X/G$ is an $S$-scheme}: We note that the claim is local on $S$, so we can use Lemma~\ref{lemma:alg-open-G-orbit-preserved} to reduce to the case where $S$ is affine. Now Lemma~\ref{lemma:alg-G-basis} allows to cover $X$ by $G$-stable open affine subschemes $U_i$. Then the construction of the geometric quotient implies that  
\[
\pi(U_i) \subset X/G
\] 
is an open subset that is naturally isomorphic to $U_i/G$, and $\pi^{-1}(U_i/G)$ coincides with $U_i$. This implies that it suffices to show that $U_i/G$ is a scheme. This was already shown in Proposition~\ref{prop:alg-geometric-quotient-affine}.\smallskip 

{\it Step 2. $\pi: X \to X/G$ is surjective, integral (resp. finite) and fibers are exactly the $G$-orbits}: Similar to Step~$1$, we can assume that $X$ and $S$ are affine. Then apply Lemma~\ref{lemma:alg-finite-integral}. \smallskip

{\it Step 3. $\pi\colon X \to X/G$ is universal and commutes with flat base change}: The universality is essentially trivial (Remark~\ref{rmk:universal-locally rings}). To show the latter claim, we can again assume that $X=\Spec A$ and $S=\Spec R$ are affine and it suffices to consider affine $Z$. Then the claim follows from Lemma~\ref{lemma:invariants-commute-flat-base-change} and the identification of $X/G$ with $\Spec A^G$. 
\end{proof}

\subsection{Schemes Over a Valuation Ring $k^+$}\label{section:valuation}

The main drawback of Theorem~\ref{thm:alg-main} is that if $R$ is not noetherian we do not know if $X/G$ is finite type over $S=\Spec R$ when $X$ is. This makes this theorem not so useful in practice as we often do not want to leave the realm of finite type morphisms. The main work left is to show that the ring of invariants $A^G$ is finite type over $R$ if $A$ is. If $R$ is noetherian, this problem is resolved using the Artin-Tate Lemma~\ref{lemma:intro-Artin-Tate}. The main goal of this section is to generalize it to certain non-noetherian situations. \smallskip

For the rest of the section, we fix a valuation ring $k^+$ with fraction field $k$ and maximal ideal $\m_k$. \smallskip

\begin{defn}\label{defn:valuation-saturated} Let $N \subset M$ be an inclusion of $k^+$-modules. We say that $N$ is {\it saturated} in $M$ if the quotient $M/N$ is $k^+$-torsion free.
\end{defn}

\begin{lemma}\label{lemma:valuations-general} Let $k^+$ be a valuation ring, $A$ a finite type $k^+$-algebra, and $M$ a finite $A$-module. Then 
\begin{enumerate}
    \item\label{lemma:valuations-general-1} A $k^+$-module $N$ is flat over $k^+$ if and only if it is torsion free. 
    \item\label{lemma:valuations-general-3} If $M$ is $k^+$-flat, it is a finitely presented $A$-module. 
    \item\label{lemma:valuations-general-2} If $A$ is $k^+$-flat, it is a finitely presented $k^+$-algebra.
    \item\label{lemma:valuations-general-4} Let $N\subset M$ be a saturated $A$-submodule of $M$. Then $N$ is a finite $A$-module.

 \end{enumerate}
\end{lemma}
\begin{proof}
By \cite[Theorem 7.7]{M1} a $k^+$-module $N$ is flat if and only if $I\otimes_{k^+} N \to N$ is injective for any {\it finitely generated} ideal $I\subset k^+$. But such $I$ is principal since $k^+$ is a valuation ring, so we are done (see also \cite[\href{https://stacks.math.columbia.edu/tag/0539}{Tag 0539}]{stacks-project} for a different proof).\smallskip

The second and third  claims are proven in \cite[\href{https://stacks.math.columbia.edu/tag/053E}{Tag 053E}]{stacks-project}.   \smallskip

Now we show the last claim. We consider the quotient module $M/N$. The saturatedness assumption says that it is $k^+$-flat, and it is clearly finite as an $A$-module. Thus, (\ref{lemma:valuations-general-3}) ensures that $M/N$ is finitely presented over $A$. So $N$ is a finite $A$-module as it is the kernel of a homomorphism from a finite module to a finitely presented one (see \cite[\href{https://stacks.math.columbia.edu/tag/0519}{Tag 0519}]{stacks-project}). 
\end{proof}




\begin{lemma}[Non-noetherian Artin-Tate]\label{lemma:valuation-Artin-Tate} Let $A \to B$ be a finite injective morphism of $k^+$-algebras. Suppose that $B$ is a finite type $k^+$-algebra and $A$ is a saturated $k^+$-submodule of $B$ (in the sense of Definition \ref{defn:valuation-saturated}). Then $A$ is a $k^+$-algebra of finite type.
\end{lemma}
\begin{proof}

By assumption, $B$ is of finite type over $k^+$, so there is a finite set of elements $x_i\in B$ such that the $k^+$-algebra homomorphism
\[
p:k^+[T_1, \dots, T_n] \to B
\] 
that sends $T_i$ to $x_i$ is surjective. Since $B$ is a finite $A$-module, we can choose some $A$-module generators $y_1, \dots, y_m \in B$. The choice of $x_1, \dots, x_n$ and $y_1, \dots, y_m$ implies that there are some $a_{i,j}, a_{i,j,l} \in A$ with the relations 
\[
x_i = \sum_j a_{i,j} y_i
\]
\[
y_iy_j=\sum_{l} a_{i,j,l}y_{l}.
\]
Now consider the $k^+$-subalgebra $A'$ of $A$ generated by all $a_{i,j}$ and $a_{i,j,l}$. Clearly, $A'$ is of finite type over $k^+$. Moreover, $B$ is finite over $A'$ as $y_1, \dots, y_m$ are $A'$-module generators of $B$.

We use Lemma~\ref{lemma:valuations-general}(\ref{lemma:valuations-general-4}) over $A'$ to ensure that $A$ is finite over $A'$ as it is a saturated $A'$-submodule of the finite $A'$-module $B$. Therefore, $A$ is of finite type over $k^+$.
\end{proof}

\begin{cor}\label{cor:valuation-invariants-finitely-generated} Let $A$ be a flat, finite type $k^+$-algebra with a $k^+$-action of a finite group $G$. Then $A^G$ is a finite type flat $k^+$-algebra, and the natural morphism $A^G \to A$ is finitely presented.
\end{cor}
\begin{proof}
Lemma~\ref{lemma:alg-finite-integral} gives that $A$ is a finite $A^G$-module, and $A^G$ is easily seen to be saturated in $A$ using that $A$ is $k^+$-torsionfree (because it is $k^+$-flat). Therefore, Lemma~\ref{lemma:valuations-general}~(\ref{lemma:valuations-general-1}) implies that $A^G$ is $k^+$-flat and Lemma~\ref{lemma:valuation-Artin-Tate} ensures that $A^G$ is finite type over $k^+$. Now Lemma~\ref{lemma:valuations-general}~(\ref{lemma:valuations-general-3}) guarantees that $A$ is a finitely presented $A^G$-module as it is $A^G$-finite and $k^+$-flat. Thus, it is finitely presented as an $A^G$-algebra by \cite[Proposition 1.4.7]{EGA4_1}. 
\end{proof}

\begin{rmk}\label{rmk:universally-adhesive-artin-tate} Lemma~\ref{lemma:valuation-Artin-Tate} and Corollary~\ref{cor:valuation-invariants-finitely-generated} have versions over a universally adhesive base (see Definition~\ref{defn:adhesive}). We refer to  Lemma~\ref{lemma:universally-adhesive-Artin-Tate} and Corollary~\ref{cor:universally-adhesive-invariants-finitely-generated} for the precise results. 
\end{rmk}

\begin{thm}\label{thm:valuation-main} Let $X$ be a flat, locally finite type $k^+$-scheme with a $k^+$-action of a finite group $G$. Suppose that each point $x\in X$ admits an affine neighborhood $V_x$ containing $G.x$. Then the scheme $X/G$ as in Theorem~\ref{thm:alg-main} is flat and locally finite type over $k^+$, and the integral surjection $\pi\colon X \to X/G$ is finite and finitely presented.
\end{thm} 

\begin{proof}
By construction, $X/G$ is clearly $k^+$-flat. To show that $X/G$ is locally of finite type and that $\pi$ is finitely presented, we reduce to the affine case by passing to a $G$-stable affine open covering of $X$ (see Lemma~\ref{lemma:alg-G-basis}). Now apply Corollary~\ref{cor:valuation-invariants-finitely-generated}. \smallskip
\end{proof}

\begin{rmk}\label{rmk:valuation-main-general} Theorem~\ref{thm:valuation-main} also has a version for the base scheme $S$ that is universally $\mathcal{I}$-adically adhesive for some quasi-coherent ideal of finite type $\mathcal{I}$ (see Definition~\ref{defn:adhesive-global}). We refer to Theorem~\ref{thm:universally-adhesive-main}.
\end{rmk}


\section{Quotients of Admissible Formal Schemes}\label{section:formal}

We discuss the existence of quotients for some class of formal schemes by an action of a finite group $G$. The strategy to construct the quotient spaces is close to the one used in Section~\ref{section:alg}. We first construct the candidate space $\X/G$ that is, a priori, only a topologically locally ringed space. This construction clearly satisfies the universal property, but it is not clear (and generally false) that $\X/G$ is a formal scheme. We resolve this issue by first showing that it is a formal scheme if $\X$ is affine. Then we argue by gluing to prove the claim for a larger class of formal schemes. \smallskip

There are two main complications compared to Section~\ref{section:alg}. The first one is that we cannot anymore firstly show that $\X/G$ is a formal scheme by a very general argument and {\it then} study its properties under further assumptions, e.g. show that it is flat or (topologically) finite type over the base. The problem can be seen even in the case of an affine formal scheme $\X=\Spf A$. The proof of Proposition~\ref{prop:alg-geometric-quotient-affine} crucially uses that the localization $(A^G)_f$ is $A^G$-flat for any $f\in A^G$. The analogue in the world of formal schemes would be that the {\it completed localization} $(A^G)_{\{f\}}=\wdh{(A^G)_f}$ is $A^G$-flat. However, this requires some finiteness assumption on $A^G$ in order to hold. Therefore, we need to verify algebraic properties of $A^G$ at {\it the same time} as constructing the isomorphism $\Spf A/G \simeq \Spf A^G$.  \smallskip

The second, related problem is that one needs to be more careful with certain topological aspects of the theory. For instance, the fiber product of affine formal schemes is given by the {\it completed} tensor product on the level of corresponding algebras. This is a more delicate functor as it is neither left nor right exact. So we pay extra attention to make sure that these complications do not cause any issues under suitable assumptions. 

\subsection{The Setup and the Candidate Space $\X/G$}

\begin{defn}\label{defn:formal-microbial-valuation} A valuation ring $k^+$ is {\it microbial} if it has a finitely generated (hence principal) ideal of definition $I$, i.e. any neighborhood $0 \in U \subset k^+$ open in the valuation topology contains $I^n$ for some $n$.
\end{defn}

\begin{defn} An element $\varpi \in k^+$ is a {\it pseudo-uniformizer} if $(\varpi) \subset k^+$ is an ideal of definition in $k^+$.
\end{defn}

\begin{exmpl} Any valuation ring $k^+$ of finite rank is microbial. This follows from the characterization of microbial valuations in \cite[Definition 1.1.4(e)]{H3} or \cite[Proposition 9.1.3(3)]{Seminar}. 

More generally, a valuation ring $k(x)^+\subset k(x)$ associated with any point $x\in X$ of an {\it analytic} adic space (see Definition \ref{defn:adic-analytic-spaces}) $X$ is microbial. This can be seen from \cite[Definition 1.1.4(c)]{H3} or \cite[Proposition 9.1.3(2)]{Seminar}. 
\end{exmpl}

For the rest of the section, we fix a complete, microbial valuation ring $k^+$ with a pseudo-uniformizer $\varpi$. We denote by $\mathfrak{S}$ the formal spectrum $\Spf k^+$. \smallskip

A {\it formal $k^+$-scheme} will always mean a $\varpi$-adic formal $k^+$-scheme. It is easy to see that this notion does not depend on the choice of an ideal of definition.

\begin{defn}\label{defn:formal-admissible} A $k^+$-algebra $A$ is called {\it admissible} if $A$ is $k^+$-flat and topologically of finite type (i.e. there is a surjection $k^+\langle t_1, \dots t_d\rangle \to A$). \medskip

A formal $k^{+}$-scheme $\X$ is called {\it admissible} if it is $k^+$-flat and locally topologically of finite type.
\end{defn}

\begin{rmk}\label{rmk:formal-flat-issue}
\begin{enumerate}
    \item We note that there are many (non-equivalent) ways to define flatness in formal geometry. They are all equivalent for a morphism $f\colon \X \to \Y$ of locally topologically finite type formal $k^+$-schemes. \smallskip

We prefer to use the following as the definition: $f$ is {\it flat} if $f^\#_{f(x)} \colon \O_{\Y, f(x)} \to \O_{\X, x}$ is flat for all $x\in \X$ (i.e. $f$ is flat as a morphism of locally ringed spaces). We mention that in the case $f\colon \Spf B \to \Spf A$  a morphism of affine, topologically finite type formal $k^+$-schemes, this notion is equivalent to the flatness of $A\to B$. This follows from \cite[Proposition I.4.8.1]{FujKato} and Remark~\ref{rmk:adhesive-microb-tu-adhesive} (see \cite[\textsection I.2.1(a)]{FujKato} to relate adhesiveness to rigid-noetherianness). 
\item Similarly, a morphism $f\colon \Spf B \to \Spf A$ of formal $k^+$-schemes is topologically of finite type if and only if $A \to B$ is topologically of finite type (see \cite[Lemma I.1.7.4]{FujKato}). 
\item\label{rmk:formal-flat-issue-3} In particular, if $\X=\Spf A$ is an admissible formal $k^+$-scheme, the $k^+$-algebra $A$ is admissible. 
\end{enumerate}
\end{rmk}

We summarize the main properties of locally topologically finite type formal $k^+$-schemes in the lemma below:

\begin{lemma}\label{lemma:formal-general} Let $k^+$ be a complete, microbial valuation ring, $A$ a topologically finite type $k^+$-algebra, and $M$ a finite $A$-module. Then
\begin{enumerate}
    \item\label{lemma:formal-general-1}  $M$ is $\varpi$-adically complete. In particular, $A$ is $\varpi$-adically complete. 
    \item\label{lemma:formal-general-2}  If $A$ is $k^+$-flat, it is topologically finitely presented. 
    \item\label{lemma:formal-general-new}  If $M$ is $k^+$-flat, it is finitely presented over $A$. 
    \item\label{lemma:formal-general-3}  Let $N \subset M$ be a saturated (in the sense of Definition~\ref{defn:valuation-saturated}) $A$-submodule of $M$. Then $N$ is a finite $A$-module. 
    \item\label{lemma:formal-general-new-2}  Let $N \subset M$ be an $A$-submodule of $M$. Then the $\varpi$-adic topology on $M$ restricts to the $\varpi$-adic topology on $N$.     
    \item\label{lemma:formal-general-4}  For any element $f\in A$, the completed localization $A_{\{f\}}=\lim_n A_f/\varpi^n A_f$ is $A$-flat. 
\end{enumerate}
\end{lemma}
The first five results of this lemma are essentially due to Raynaud and Gruson \cite{Raynaud-Gruson}. 
\begin{proof}
The first claim is \cite[Proposition 7.3/8]{B}. The second is \cite[Corollary 7.3/5]{B}. The third in \cite[Theorem 7.3/4]{B}. The fourth and fifth are covered by \cite[Lemma 7.3/7]{B}. For the last claim, we note that \cite[Proposition 7.3/11]{B} ensures that it suffices to show that 
\[
A/f^nA \to A_{\{f\}}/f^nA_{\{f\}}
\]
is flat for any integer $n\geq 1$. Now \cite[\href{https://stacks.math.columbia.edu/tag/05GG}{Tag 05GG}]{stacks-project} implies that $A_{\{f\}}/f^nA_{\{f\}} \simeq A_f/f^nA_f$, so the desired statement follows from $A$-flatness of $A_f$. 
\end{proof}

\begin{defn}\label{defn:geometric-quotient-formal} Let $G$ be a finite group, and $\X$ a locally topologically ringed space over $\S$ with a right $\S$-action of $G$. The {\it geometric quotient} $\X/G=(|\X/G|, \O_{\X/G}, h)$ consists of:
\begin{itemize}\itemsep0.5em
\item the topological space $|\X/G|\coloneqq |\X|/G$ with the quotient topology. We denote by $\pi:|\X| \to |\X/G|$ the natural projection.
\item the sheaf of topological rings $\O_{\X/G}\coloneqq (\pi_*\O_\X)^G$ with the subspace topology.
\item the morphism $h:\X/G \to \S$ defined by the pair $(h, h^{\#})$, where $h:|\X|/G \to \S$ is the unique morphism induced by $f\colon \X \to \S$ and $h^{\#}$ is the natural morphism 
\[
\O_{\S} \to h_*\left(\O_{\X/G}\right)=h_*\left(\left(\pi_*\O_{\X}\right)^G\right)=\left(h_*\left(\pi_*\O_{\X}\right)\right)^G=\left(f_*\O_\X\right)^G
\]
that comes from $G$-invariance of $f$.
\end{itemize}
\end{defn}

\begin{rmk}\label{rmk:universal-topologically-locally-ringed} By construction, $\X/G$ is a topologically ringed $\S$-space, and $\pi\colon \X \to \X/G$ is a morphism of topologically ringed $\S$-spaces. Furthermore, Lemma~\ref{lemma:geometric-quotient-locally-ringed} ensures that $\X/G$ is a topologically {\it locally ringed} $\S$-space, and $\pi$ is a morphism of topologically {\it locally ringed} $\S$-spaces (so $\X/G \to \S$ is too). It is trivial to see that the pair $(\X/G, \pi)$ is a universal object in the category of $G$-invariant morphisms to topologically locally ringed $\S$-spaces.
\end{rmk}

Our main goal is to show that under some mild assumptions, $\X/G$ is an admissible formal $\S$-scheme when $\X$ is. We start with the case of affine formal schemes and then move to the general case. 

\subsection{Affine Case}

We show that the quotient $\X/G$ of an admissible affine formal $k^+$-scheme $\X=\Spf A$ is canonically isomorphic to $\Spf A^G$ that is, in turn, an admissible formal $k^+$-scheme. We point out that in contrast to the scheme case, we need firstly to establish that $A^G$ is an admissible $k^+$-algebra, and only then we can show that $\X/G$ is isomorphic to $\Spf A^G$. Therefore, we start the section with studying certain properties of the ring of invariants $A^G$.

\begin{rmk}\label{rmk:automatic-cont} Let $(R, I)$ be a ring with a finitely generated ideal $I$, $M$ an $R$-module with the $I$-adic topology, and $f\colon M \to M$ an $R$-linear homomorphism. Then $f$ is automatically continuous in the $I$-adic topology because $f^{-1}(I^nM) \supset I^nM$.
\end{rmk}

\begin{lemma}\label{lemma:formal-invariants-topology-prelim} Let $(R, I)$ be a ring with a finitely generated ideal $I$, and $M$ an $I$-adically complete $R$-module with a closed $R$-submodule $N\subset M$. Then $N$ is complete in the $I$-adic topology.
\end{lemma}
\begin{proof}
Since $N$ is closed in $M$, it is complete for the subspace topology. All this means by design is that 
\[
N\to \lim_n N/(I^nM\cap N)
\]
is an isomorphism, and we need to justify that this implies that 
\[
N \to \lim_n N/I^nN
\]
is an isomorphism. For this purpose, we will crucially use that $I$ is finitely generated. \smallskip

We start by considering the diagram
\begin{equation}\label{eqn:diagram-triangle}
\begin{tikzcd}
    N \arrow{r}{\alpha} \arrow{rd}{\gamma} & \lim_n N/I^n N \arrow{d}{\beta} \\
    & \lim_n N/(I^nM \cap N).
\end{tikzcd}
\end{equation}
Since $\gamma$ is an isomorphism, we conclude that $\alpha$ is injective. Therefore, it suffices only to justify that $\alpha$ is surjective. By \cite[\href{https://stacks.math.columbia.edu/tag/090S}{Tag 090S}]{stacks-project} (which uses that $I$ is finitely generated), it suffices to justify surjectivity of 
\[
N \to \lim_n N/f^n N
\] 
for each $f\in I$. Furthermore, \cite[\href{https://stacks.math.columbia.edu/tag/090T}{Tag 090T}]{stacks-project} ensures that $M$ is $f$-adically complete. Therefore, we may assume that $I=(f)$ and show that the natural morphism
\[
\alpha \colon N \to \lim_n N/f^nN
\]
is surjective. Now we use Diagram~(\ref{eqn:diagram-triangle}) and the fact that $\gamma$ is an isomorphism to reduce the question to showing that $\beta\colon \lim_n N/f^nN \to \lim_n N/(f^nM\cap N)$ is injective. \smallskip

Injectivity of $\beta$ boils down to showing that, for any sequence of elements $\{a_n\in N\}$ with $a_{n+1}-a_n \in f^nN$ and $a_n\in f^nM \cap N$, we have $a_n \in f^nN$. \smallskip

The assumption on $a_n$ implies that $a_{n}=a_{n+1}+f^nx_n$ for some $x_n\in N$. We note that the sum $x_n + fx_{n+1}+f^2x_{n+2} + \dots$ converges in $N$ because $f^mx_{n+m}\in f^m M \cap N$ for any $m\geq 1$. Let us denote the sum $x_n + fx_{n+1}+f^2x_{n+2} + \dots$ by $b_n\in N$.  Then we claim that 
\begin{equation*}
    a_n = f^n(x_n + fx_{n+1}+f^2x_{n+2} + \dots)=f^nb_n\in f^n N.
\end{equation*}
For this we observe that the partial sums of
\[
f^nb_n = f^n(x_n + fx_{n+1}+f^2x_{n+2} + \dots)
\]
are equal $a_n - a_{n+m}$. Since $a_{n+m}\in f^{n+m}M\cap N$ for any $m\geq 1$ and $N$ is complete in the subspace topology, we conclude 
\[
a_n = f^n(x_n + fx_{n+1}+f^2x_{n+2} + \dots) = f^nb_n
\]
finishing the proof.
\end{proof}

\begin{cor}\label{cor:formal-invariants-topology} Let $(R, I)$ be a ring with a finitely generated ideal $I$, and $A$ an $I$-adically complete $R$-algebra with an $R$-action\footnote{This action is automatically continuous by Remark~\ref{rmk:automatic-cont}.} of a finite group $G$. Then $A^G$ is complete in the $I$-adic topology.
\end{cor}
\begin{proof}
Note that $A^G$ is closed submodule of $A$ since it is the kernel of the continuous morphism $A \xr{\alpha-\rm{Id}} \prod_{g\in G} A$. Therefore, the result follows directly from Lemma~\ref{lemma:formal-invariants-topology-prelim}.
\end{proof}

\begin{lemma}\label{lemma:formal-invariants-saturated} Let $A$ be an admissible $k^+$-algebra with a $k^+$-action of a finite group $G$. Then
\begin{enumerate}
    \item $A^G$ is complete in the $\varpi$-adic topology.
    \item $A^G$ is saturated in $A$.
    \item $A$ is finite as an $A^G$-module.
\end{enumerate}
\end{lemma}
\begin{proof}
The first claim is Corollary~\ref{cor:formal-invariants-topology}. The second claim is clear by $k^+$-flatness of $A$. Thus we only need to show the last claim. \smallskip

Lemma~\ref{lemma:alg-finite-integral}(\ref{lemma:alg-finite-integral-1}) guarantees that $A^G \to A$ is integral. However, the proof of finiteness in Lemma~\ref{lemma:alg-finite-integral}(\ref{lemma:alg-finite-integral-3}) is not applicable here since $A$ is not necessarily finite type over $k^+$: it is only topologically finite type. \smallskip

We now overcome this difficulty. Clearly, the morphism $A^G/\varpi A^G \to A/\varpi A$ is integral. But $A/\varpi A$ is a finite type $k^+/\varpi k^+$-algebra by our assumption, so $A^G/\varpi A^G \to A/\varpi A$ is a finite type morphism. Since an integral map of finite type is finite, we conclude that the morphism $A^G/\varpi A^G \to A/\varpi A$ is finite. Therefore, the successive approximation argument (or \cite[\href{https://stacks.math.columbia.edu/tag/031D}{Tag 031D}]{stacks-project}) implies that $A$ is finite as an $A^G$-module .
\end{proof}

\begin{lemma}[Adic Artin-Tate]\label{lemma:formal-Artin-Tate} Let $A \to B$ be a finite injective morphism of $\varpi$-adically complete $k^+$-algebras. Suppose that $B$ is a topologically finite type $k^+$-algebra and $A$ is a saturated $k^+$-submodule of $B$ (in the sense of Definition \ref{defn:formal-microbial-valuation}). Then $A$ is also a topologically finite type $k^+$-algebra.
\end{lemma}
The proof imitates the proof of Lemma~\ref{lemma:valuation-Artin-Tate}; the  main new difficulty is that we need to keep track of topological aspects of our algebras in order to work with topologically finite type algebras in a meaningful way.
\begin{proof}
Since $B$ is topologically finite type over $k^+$, we can choose a finite set of elements $x_1, \dots, x_n$ such that the natural $k^+$-linear continuous homomorphism
\[
p\colon k^+\langle T_1, \dots, T_n \rangle  \to B
\]
that sends $T_i$ to $x_i$ is surjective. \smallskip

Since $B$ is a finite $A$-module, we can choose some $A$-module generators $y_1, \dots, y_m \in B$. The choice of $x_1, \dots, x_n$ and $y_1, \dots, y_m$ implies that there are some $a_{i,j}, a_{i,j,l} \in A$ and relations 
\[
x_i = \sum_j a_{i,j} y_i
\]
\[
y_iy_j=\sum_l a_{i,j,l}y_{l}.
\]
We consider the $k^+$-algebra $A'\coloneqq k^+\langle T_{i, j}, T_{i,j,l}\rangle$ with a continuous $k^+$-algebra homomorphism $A'\to A$ that sends $T_{i,j}$ to $a_{i,j}$, and $T_{i,j,l}$ to $a_{i,j,l}$. This map is well-defined as $A$ is $\varpi$-adically complete. \smallskip

By definition $A'$ is topologically finite type over $k^+$, and we claim that $B$ is finite over $A'$ since it is generated by $y_1, \dots, y_m$ as an $A'$-module. To see this we note that it suffices to show it mod $\varpi$ by successive approximation (or \cite[\href{https://stacks.math.columbia.edu/tag/031D}{Tag 031D}]{stacks-project}). However, it is easily seen to be finite mod $\varpi$ due to the relations above.\smallskip

We use Lemma~\ref{lemma:formal-general}(\ref{lemma:formal-general-3}) to conclude that $A$ is finite over $A'$ as a {\it saturated} submodule of a finite $A'$-module $B$. This finishes the proof since a finite algebra over a topologically finite type $k^+$-algebra is also topologically finite type.
\end{proof}

\begin{cor}\label{cor:formal-invariants-top-finitely-generated} Let $A$ be an admissible $k^+$-algebra with a $k^+$-action of a finite group $G$. Then $A^G$ is an admissible $k^+$-algebra, the induced topology on $A^G$ coincides with the $\varpi$-adic topology, and $A$ is a finitely presented $A^G$-module. 
\end{cor}
\begin{proof}
We use Lemma~\ref{lemma:formal-invariants-saturated} to see that $A^G$ is $\varpi$-adically complete, and $A^G \to A$ is saturated. Then Lemma~\ref{lemma:formal-Artin-Tate} guarantees that $A^G$ is a topologically finitely generated $k^+$-algebra. Now $A$ is a finite module over a topologically finitely generated $k^+$-algebra $A^G$, so the induced topology on $A^G$ coincides with the $\varpi$-adic topology by Lemma~\ref{lemma:formal-general}(\ref{lemma:formal-general-new-2}). \smallskip

Now Lemma~\ref{lemma:valuations-general}~(\ref{lemma:valuations-general-1}) implies that $A^G$ is $k^+$-flat as it is torsion free. Therefore, Lemma~\ref{lemma:formal-general}~(\ref{lemma:formal-general-new}) guarantees that $A$ is a finitely presented $A^G$-module. 
\end{proof}

\begin{rmk}\label{rmk:tu-adhesive-invariants-finitely-generated} One can show that the $\varpi$-adic topology on $A^G$ coincides with the induced topology from first principles. But we prefer the proof above as it generalizes better to the topologically universally adhesive situation (see Definition~\ref{defn:tu-adhesive}). \smallskip

Namely, Lemma~\ref{lemma:formal-Artin-Tate} and Corollary~\ref{cor:formal-invariants-top-finitely-generated} hold over any $I$-adically complete base ring $R$ that is topologically universally adhesive (see Definition~\ref{defn:tu-adhesive}). We refer to Lemma~\ref{lemma:topologically-universally-adhesive-Artin-Tate} and Corollary~\ref{cor:topologically-universally-adhesive-invariants-top-finitely-generated} for the precise results. 
\end{rmk}

Finally, we are ready to show that $\X/G$ is an affine admissible formal $k^+$-scheme if $\X$ is so. 

\begin{prop}\label{prop:formal-geometric-quotient-affine} Let $\X=\Spf A$ be an affine admissible formal $k^+$-scheme with a $k^+$-action of a finite group $G$. Then the natural map $\phi\colon \X/G \to \Y=\Spf A^G$ is a $k^+$-isomorphism of topologically locally ringed spaces. In particular, $\X/G$ is an admissible formal $k^+$-scheme.
\end{prop}
\begin{proof} 

{\it Step 0. $\Spf A^G$ is an admissible formal $k^+$-scheme}: The $k^+$-algebra $A$ is admissible by Remark~\ref{rmk:formal-flat-issue}(\ref{rmk:formal-flat-issue-3}) (and the analogous fact for topologically finitely generated morphisms). Now the claim immediately follows from Corollary~\ref{cor:formal-invariants-top-finitely-generated}. \medskip

{\it Step 1. $\phi$ is a homeomorphism}: This is completely analogous to Step $1$ of Proposition~\ref{prop:alg-geometric-quotient-affine}. We only need to show that $p\colon \Spf A \to \Spf A^G$ is a surjective, finite morphism with fibers being exactly $G$-orbits. \smallskip

Lemma~\ref{lemma:formal-invariants-saturated} says that $\Spf A \to \Spf A^G$ is finite. We note that surjectivity of $\Spec A \to \Spec A^G$ obtained in Lemma~\ref{lemma:alg-finite-integral}(\ref{lemma:alg-finite-integral-2}) implies that any prime ideal $\mathfrak{p}$ of $A^G$ lifts to a prime ideal $\mathfrak{P}$ in $A$. If $\mathfrak{p}$ is open (i.e. it contains $\varpi^n$ for some $n$), then so is $\mathfrak{P}$. Therefore, the morphism $\Spf A \to \Spf A^G$ is surjective. \smallskip

Now we note that a prime ideal $\mathfrak{P}\subset A$ is open if and only if so is $g(\mathfrak{P})$ for $g\in G$. So Lemma~\ref{lemma:alg-finite-integral}(\ref{lemma:alg-finite-integral-2}) ensures that the fibers of $\Spf A \to \Spf A^G$ are exactly $G$-orbits.
\medskip

{\it Step 2. $\phi$ is an isomorphism of topologically locally ringed spaces:} We already know that $\phi$ is a homeomorphism. So the only thing that we need to show here is that the morphism 
\[
\O_{\Y} \to \phi_* \O_{\X/G}
\] 
is an isomorphism of topological sheaves. Using the basis of basic affine opens in $\Y$, it suffices to show that
\begin{equation}\label{eqn:formal-affine-quotient}
    \left(A^G\right)_{\{f\}} \to \left(A_{\{f\}}\right)^G
\end{equation}
is a topological isomorphism for $f\in A^G$. Corollary~\ref{cor:formal-invariants-top-finitely-generated} ensures that both sides have the $\varpi$-adic topology, so we can ignore the topologies. \smallskip

Now we show that (\ref{eqn:formal-affine-quotient}) is an (algebraic) isomorphism. We note that
\[
A_{\{f\}} \simeq \left(A^G\right)_{\{f\}} \wdh{\otimes}_{A^G} A \simeq \left(A^G\right)_{\{f\}} \otimes_{A^G} A ,
\]
where the second isomorphism follows from  Lemma~\ref{lemma:formal-general}(\ref{lemma:formal-general-1}) and finiteness of $A$ over $A^G$. Therefore, it suffices to show that the natural morphism
\[
\left(A^G\right)_{\{f\}} \to \left(\left(A^G\right)_{\{f\}} \otimes_{A^G} A\right)^G
\]
is an isomorphism of $k^+$-algebras. This follows from Lemma~\ref{lemma:invariants-commute-flat-base-change} and Lemma~\ref{lemma:formal-general}(\ref{lemma:formal-general-4}). \medskip
\end{proof} 

\begin{rmk}\label{rmk:formal-tuadhesive-affine-quotient} Proposition~\ref{prop:formal-geometric-quotient-affine} can be generalized to the case of an affine, universally adhesive base $\S=\Spf R$ (see Definition~\ref{defn:universally-adhesive-formal-schemes}). We refer to Proposition~\ref{prop:topologically-universally-adhesive-geometric-quotient-affine} for the precise statement.
\end{rmk}

\subsection{General Case}

The main goal of this section is to globalize the results of the previous section. This is very close to what we did in the schematic situation in the proof of Theorem~\ref{thm:alg-main}. 

\begin{lemma}\label{lemma:formal-open-G-orbit-preserved} Let $\X$ be a formal $\S$-scheme with an $\S$-action of a finite group $G$. Suppose that each point $x\in \X$ admits an open affine subscheme $\V_x$ that contains the orbit $G.x$. Then the same holds with $\X$ replaced by any $G$-stable open formal subscheme $\sU \subset \X$.
\end{lemma}
 \begin{proof}
 This follows easily from Lemma~\ref{lemma:alg-open-G-orbit-preserved} as 
 \[
 |\X|\simeq |\X\times_{\Spf k^+} \Spec k^+/\varpi| \text{ and } |\S|=|\Spf k^+|\simeq |\Spec k^+/\varpi|.
 \]
 Thus, we can reduce the statement to the case of schemes.
 \end{proof}

\begin{lemma}\label{lemma:formal-G-basis} Let $\X$ be a formal $\S$-scheme with an $\S$-action of a finite group $G$. Suppose that for any point $x\in \X$ there is an open affine subscheme $\V_x$ that contains the orbit $G.x$. Then each point $x\in \X$ has a $G$-stable open affine neighborhood $\sU_x \subset \X$.
\end{lemma}
\begin{proof}
Again, this easily follows from Lemma~\ref{lemma:alg-G-basis} as an open subscheme $\sU \subset \X$ is affine if and only if $\sU_0\coloneqq \sU\times_{\Spf k^+} \Spec k^+/\varpi$ is affine \cite[Proposition I.4.1.12]{FujKato}. 
\end{proof}

\begin{rmk} We note that the condition of Lemma~\ref{lemma:formal-G-basis} is automatically satisfied if the special fiber $\ov{\X}\coloneqq \X \times_{\Spf k^+} \Spec k^+/\m_k$ is quasi-projective over $\Spec k^+/\m_k$. This follows easily from Proposition~\ref{prop:example-alg}.
\end{rmk}

Now we are ready to formulate and prove the main result of this section. 

\begin{thm}\label{thm:formal-main} Let $\X$ be an admissible formal $k^+$-scheme with a $k^+$-action of a finite group $G$. Suppose that each point $x\in \X$ admits an affine neighborhood $\V_x$ containing $G.x$. Then $\X/G$ is an admissible formal $k^+$-scheme. Moreover, it satisfies the following properties:
\begin{enumerate}
\item $\pi\colon \X \to \X/G$ is universal in the category of $G$-invariant morphisms to topologically locally ringed spaces over $\S$. 
\item $\pi: \X \to \X/G$ is a surjective, finite, topologically finitely presented morphism (in particular, it is closed). 
\item Fibers of $\pi$ are exactly the $G$-orbits.
\item\label{thm:formal-main-4} The formation of the geometric quotient commutes with flat base change, i.e. for any flat, topologically finite type $k^+$-morphism $\frak{Z}\to \X/G$, the geometric quotient $(\X \times_{\X/G} \frak{Z})/G$ is an admissible formal $k^+$-schemes, and the natural morphism $(\X \times_{\X/G} \frak{Z})/G \to \frak{Z}$ is an isomorphism.
\end{enumerate}
\end{thm} 
\begin{proof}
{\it Step 1. The geometric quotient $\X/G$ is an admissible formal $k^+$-scheme}: The same proof as used in the proof of Theorem~\ref{thm:alg-main} just goes through. We firstly reduce to the case of an affine $\X=\Spf A$ by choosing a $G$-stable open affine covering, and then use Proposition~\ref{prop:formal-geometric-quotient-affine} to show the claim in the affine case.\smallskip

{\it Step 2. $\pi: \X \to \X/G$ is surjective, finite, topologically finitely presented, and fibers are exactly the $G$-orbits}: The morphism is clearly surjecitve with fibers being exactly the $G$-orbits. \smallskip

To show that it is finite and topologically finitely presented, we can assume that $\X=\Spf A$ is affine. Lemma~\ref{lemma:formal-invariants-saturated} says that $\X \to \X/G$ is finite. Corollary~\ref{cor:formal-invariants-top-finitely-generated} ensures that $A$ is finitely presented as an $A^G$-module. Therefore, it is topologically finitely presented as an $A^G$-algebra because \cite[Proposition 7.3/10]{B} gives that $A^G \to A$ is topologically finitely presented if and only if $A^G/\varpi^n A^G \to A/\varpi^n A$ is finitely presented for any $n\geq 1$. \smallskip

{\it Step 3. $\pi\colon \X \to \X/G$ is universal and commutes with flat base change}: The universality is essentially trivial (see Remark~\ref{rmk:universal-topologically-locally-ringed}). To show the latter claim, we can again assume that $\X=\Spf A$ and $\frak{Z}=\Spf B$ are affine. Then the claim boils down to showing that the natural map
\[
B \to (A \wdh{\otimes}_{A^G} B)^G 
\]
is a topological isomorphism. Now we note that Lemma~\ref{lemma:valuations-general}(\ref{lemma:formal-general-1}) implies that $A \otimes_{A^G} B$ is already $\varpi$-adically complete as it is a finite module over the topologically finite type $k^+$-algebra $B$. Therefore, it suffices to show that the natural map
\begin{equation}\label{eqn:formal-general}
    B \to (A \otimes_{A^G} B)^G
\end{equation}
is a topological isomorphism. Both sides have the $\varpi$-adic topology by Corollary~\ref{cor:formal-invariants-top-finitely-generated}. So we can ignore the topologies. Now (\ref{eqn:formal-general}) is an isomorphism by Lemma~\ref{lemma:invariants-commute-flat-base-change} and flatness of $A^G \to B$ (see Remark~\ref{rmk:formal-flat-issue}).
\end{proof}

\begin{rmk}\label{rmk:tu-adhesive-formal-main}  Theorem~\ref{thm:formal-main} can be generalized to the case of a locally universally adhesive base $\S$ (see Definition~\ref{defn:universally-adhesive-formal-schemes}). We refer to Theorem~\ref{thm:adhesive-formal-main} for the precise statement. 
\end{rmk}

\subsection{Comparison between the schematic and formal quotients}\label{formal-algebraic}

The main goal of this section is to compare the schematic and formal quotients by finite groups actions. \smallskip

Throughout this section, we fix a microbial valuation ring $k^+$ and a pseudo-uniformizer $\varpi\in k^+$. Unlike previous sections, we do not assume that $k^+$ is complete.  \smallskip

If $X$ is a flat, locally finite type $k^+$-scheme, we define $\wdh{X}$ to be the formal $\varpi$-adic completion of $X$. This is easily seen to be an admissible formal $\wdh{k}^+$-scheme with a $\wdh{k}^+$-action of $G$. Using the universal property of geometric quotients, there is a natural morphism $\wdh{X}/G \to \wdh{X/G}$.

\begin{thm}\label{thm:comparison-formal-alg} Let $X$ be a flat, locally finite type $k^+$-scheme with a $k^+$-action of a finite group $G$. Suppose that any orbit $G.x \subset X$ lies in an affine open subset $V_x$. The same holds for its $\varpi$-adic completion $\wdh{X}$ with the induced $\wdh{k}^+$-action of $G$, and the natural morphism:
\[
\wdh{X}/G \to \wdh{X/G}
\]
is an isomorphism.
\end{thm}
\begin{proof}

\noindent{\it Step 1. The condition of Theorem~\ref{thm:formal-main} is satisfied for $\wdh{X}$ with the induced action of $G$}: Firstly, we observe that $\wdh{X}$ is $\wdh{k}^+$-admissible as stated above. Now Lemma~\ref{lemma:alg-G-basis} says that our assumption on $X$ implies that there is a covering of $X=\cup_{i\in I} U_i$ by affine, open $G$-stable subschemes. Then $\wdh{X}=\cup_{i\in I} \wdh{U}_i$ is an open covering of $\wdh{X}$ by affine, $G$-stable open formal subschemes. In particular, every orbit lies in an affine open formal subscheme of $\wdh{X}$. \medskip

\noindent{\it Step 2. We show that $\wdh{X}/G \to \wdh{X/G}$ is an isomorphism}: We have a commutative diagram
\[
\begin{tikzcd}
\wdh{X} \arrow{d}{\pi_{\wdh{X}}}  \arrow{rd}{\wdh{\pi_X}} &\\
\wdh{X}/G \arrow{r}{\phi} &  \wdh{X/G}.
\end{tikzcd}
\]
of admissible formal $\wdh{k}^+$-schemes. We want to show that $\phi$ is an isomorphism. To prove the claim, we can assume that $X=\Spec A$ is affine by passing to an open covering of $X$ by $G$-stable affines. Then $X/G \simeq \Spec A^G$, $\wdh{X}/G\simeq \Spf \wdh{A}^G$, and $\phi$ can be identified with the map 
\[
\Spf (\wdh{A})^G \to \Spf \wdh{\left(A^G\right)}
\]
induced by the continuous homomorphism
\begin{equation}\label{eqn:base-change-completion}
\wdh{A^G} \to (\wdh{A})^G
\end{equation}
whose source has the $\varpi$-adic topology by construction and whose target has the $\varpi$-adic topology by Corollary~\ref{cor:formal-invariants-top-finitely-generated}. 
So it suffices to show that this map is an isomorphism of abstract rings (ignoring topology) for any flat, finitely generated $k^+$-flat algebra $A$. \smallskip
 
We note that Corollary~\ref{cor:valuation-invariants-finitely-generated} shows that $A^G$ is a finite type $k^+$-algebra and Lemma~\ref{lemma:alg-finite-integral}~(\ref{lemma:alg-finite-integral-1}) shows that $A$ is a finite $A^G$-module. Therefore, \cite[Lemma 7.3/14]{B} implies that the natural homomorphism
\[
A\otimes_{A^G} \wdh{A^G} \to \wdh{A}
\]
is an (algebraic) isomorphism. Thus we can identify (\ref{eqn:base-change-completion}) with the natural homomorphism
\[
\wdh{A^G} \to \left(A\otimes_{A^G}\wdh{A^G}\right)^G
\]
that is an (algebraic) isomorphism by Lemma~\ref{lemma:invariants-commute-flat-base-change} and flatness of the map $A^G \to \wdh{A^G}$ (see \cite[Lemma 8.2/2]{B}). 
\end{proof}

\begin{rmk}\label{rmk:adhesive-comparison-formal-alg} Theorem~\ref{thm:comparison-formal-alg} has a version over any topologically universally adhesive base\footnote{We do not assume that $R$ is $I$-adically complete.} $(R, I)$ (see Definition~\ref{defn:tu-adhesive}). We refer to Theorem~\ref{thm:adhesive-comparison-formal-alg} for the precise statement. 
\end{rmk}


\section{Quotients of Strongly Noetherian Adic Spaces}\label{section:adic}

We discuss the existence of quotients of some class of analytic adic spaces by an action of a finite group $G$. We refer the reader to Appendix~\ref{defn-adic} for a review of the main definitions and facts from the theory of Huber rings and corresponding adic spaces. \smallskip

The strategy to construct quotients is close to the one used in Section~\ref{section:alg} and Section~\ref{section:formal}. We firstly construct the candidate space $X/G$ that is, a priori, only a topologically locally $v$-ringed space (see Definition~\ref{defn:valuative-spaces}). This construction clearly satisfies the universal property, but it is not clear whether $X/G$ is an adic space. We resolve this issue by firstly showing that it is an adic space if $X$ is affinoid. Then we argue by gluing to prove the claim for a larger class of adic spaces. \smallskip

We point out the two main complications compared to Section~\ref{section:formal} (and Section~\ref{section:alg}). The first new issue that is not seen in the world of formal schemes is that the notion of a finite (resp. topologically finite type) morphism of Huber pairs $(A, A^+) \to (B, B^+)$ is more involved since there is an extra condition on the morphism $A^+ \to B^+$ that makes the theory more subtle (see Definition~\ref{defn:huber-topologically-finite-type} and Definition~\ref{defn:Huber-finite}).  \smallskip 

The second issue is that the underlying topological space $\Spa(A, A^+)$ of a Huber pair $(A, A^+)$ is more difficult to express in terms of the pair $(A, A^+)$. It is the set of all valuations on $A$ with corresponding continuity and integrality conditions. So one needs some extra work to identify $\Spa(A^G, A^{+,G})$ with $\Spa(A, A^+)/G$ even in the affinoid case.

\subsection{The Candidate Space $X/G$}

For the rest of the section we fix a locally strongly noetherian analytic adic space $S$ (see Definition~\ref{defn:strongly-noetherian-space}). 

\begin{exmpl} An example of a strongly noetherian Tate affinoid adic space $S$ is $\Spa(k, k^+)$ for a microbial valuation ring $k^+$. 
\end{exmpl}

\begin{defn}\label{defn:geometric-quotient-adic} Let $G$ be a finite group and $X$ a valuation locally topologically ringed space over $S$ with a right $S$-action of $G$. The {\it geometric quotient} $X/G=(|X/G|, \O_{X/G}, \{v_{\ov{x}}\}_{\ov{x}\in X/G}, h)$ consists of:
\begin{itemize}\itemsep0.5em
\item the topological space $|X/G|\coloneqq |X|/G$ with the quotient topology. We denote by $\pi:|X| \to |X/G|$ the natural projection,
\item the sheaf of topological rings $\O_{X/G}\coloneqq (\pi_*\O_X)^G$ with the subspace topology,
\item for any $\ov{x}\in X/G$, the  valuation $v_{\ov{x}}$ defined as the composition of the natural morphism\footnote{Lemma~\ref{lemma:geometric-quotient-locally-ringed} ensures that $(|X/G|, \O_{X/G})$ is a locally ringed space, so $k(\ov{x})$ is well-defined.} $k(\ov{x}) \to k(x)$ and the valuation $v_x\colon  k(x) \to \Gamma_{v_x} \cup \{0\}$, where $x\in p^{-1}(\ov{x})$ is any lift\footnote{One can show that $v_{\ov{x}}$ is independent of the choice of $x$ similarly to Lemma~\ref{lemma:geometric-quotient-locally-ringed}.} of $\ov{x}$.
\item the morphism $h:X/G \to S$ defined by the pair $(h, h^{\#})$, where $h:|X|/G \to S$ is the unique morphism induced by $f\colon X \to S$ and $h^{\#}$ is the natural morphism 
\[
\O_{S} \to h_*\left(\O_{X/G}\right)=h_*\left(\left(\pi_*\O_{X}\right)^G\right)=\left(h_*\left(\pi_*\O_{X}\right)\right)^G=\left(f_*\O_X\right)^G
\]
that comes from $G$-invariance of $f$.
\end{itemize}
\end{defn}

\begin{rmk}\label{rmk:universal-valuative-topologically-locally-ringed} We note that Lemma~\ref{lemma:geometric-quotient-locally-ringed} ensures that $X/G$ is a topologically locally $v$-ringed $\S$-space, and $\pi\colon X \to X/G$ is a morphism of topologically locally $v$-ringed $S$-spaces (so $X/G \to S$ is too). It is trivial to see that the pair $(X/G, \pi)$ is a universal object in the category of $G$-invariant morphisms to topologically locally $v$-ringed $S$-spaces.
\end{rmk}

Our main goal is to show that under some assumptions, $X/G$ is a locally topologically finite type adic $S$-space when $X$ is. We start with the case of affinoid adic spaces and then move to the general case.

\subsection{Affinoid Case}\label{affine-analytic} For the rest of this section, we assume that $S=\Spa(R, R^+)$ is a complete Tate affinoid. \smallskip

We show that $X/G$ is a topologically finite type adic $S$-space when $X=\Spa(A, A^+)$ for a topologically finite type complete $\left(R, R^+\right)$-Tate-Huber pair $\left(A, A^+\right)$ with an $(R, R^+)$-action of a finite group $G$. \smallskip

 We start the section by discussing algebraic properties of the Tate-Huber pair $\left(A^G, A^{+, G}\right)$. In particular, we show that it is topologically of finite type over $(R, R^+)$ if $(R, R^+)$ is strongly noetherian. The main new input is the ``analytic'' Artin-Tate Lemma~\ref{lemma:adic-Artin-Tate}. 
 Then we show that the canonical morphism $X/G \to \Spa(A^G, A^{+, G})$ is an isomorphism. In particular, $X/G$ is an adic space, topologically of finite type over $S$. \smallskip

\begin{lemma}\label{lemma:adic-complete} Let $(A, A^+)$ be a complete $(R, R^+)$-Tate-Huber pair with an $(R, R^+)$-action of a finite group $G$. Then 
\begin{enumerate}
    \item\label{lemma:adic-complete-1} $A$ has a $G$-stable pair of definition $(A_0, \varpi)$ such that $A_0\subset A^+$.
    \item The subspace topology on $(A_0^G, \varpi)$ coincides with the $\varpi$-adic topology. 
    \item $(A_0^G, \varpi)$ is a complete pair of definition of $A^G$ with the subspace topology. In particular, $A^G$ is a Huber ring.
    \item $(A^G, A^{+, G})$ with the subspace topology is a Tate-Huber pair.
\end{enumerate}
\end{lemma}
\begin{proof}
We note that $A$ is Tate since $R$ is. We choose a pseudo-uniformizer $\varpi\in R^+$ and a compatible pair of definition\footnote{We abuse the notation and consider $\varpi$ as an element of $A$ via the natural morphism $R \to A$.} $(A'_0, \varpi)$ of $A$. Then \cite[Proposition 1.1]{H0} ensures that a subring $A'\subset A$ is a ring of definition if and only if $A'$ is open and bounded. So we can replace $A'_0$ with $A'_0\cap A^+$ and $\varpi$ with a power to achieve that $A'_0 \subset A^+$. \smallskip

Now {\it loc.cit.} implies that 
\[
\left(A_0, \varpi\right) \coloneqq  \left( \bigcap_{g\in G} g\left(A'_0\right), \varpi \right)
\]
is a pair of definition in $A$ contained in $A^+$, and it is $G$-stable by construction. \smallskip

To show that the subspace topology in $A_0^G$ coincides with the $\varpi$-adic topology, it suffices to show that $\varpi^n A_0 \cap A_0^G =\varpi^n A_0^G$. This can be easily seen from the fact that $\varpi$ is a unit in $A$ (and so a non zero divisor in $A_0$). \smallskip

Now we note that $A_0^G$ is complete in the subspace topology since the action of $G$ on $A_0$ is clearly continuous. Therefore, it is complete in the $\varpi$-adic topology as these topologies were shown to be equivalent. Also, we note that $A_0^G$ with the subspace topology is clearly open and bounded in $A^G$, so it is a ring of definition by \cite[Proposition 1.1]{H0}. \smallskip

Finally, we note that clearly $A^{+, G} \subset A^\circ \cap A^G \subset \left(A^G\right)^\circ$ is an open and integrally closed subring of $\left(A^G\right)^\circ$. So $(A^G, A^{+, G})$ is a Tate-Huber pair. 
\end{proof}

\begin{cor}\label{cor:adic-action-cont} Let $(A, A^+)$ be a complete $(R, R^+)$-Tate-Huber pair with an $(R, R^+)$-action of a finite group $G$. Then the action of $G$ on $A$ is continuous.
\end{cor}
\begin{proof}
We choose a $G$-stable pair of definition $(A_0, \varpi)$ as in Lemma~\ref{lemma:adic-complete}(\ref{lemma:adic-complete-1}). Then it suffices to show that the action of $G$ on $A_0$ is continuous. This is clear because $A_0$ carries the $\varpi$-adic topology.
\end{proof}

\begin{lemma}\label{lemma:adic-finite-invariants} Let $(A, A^+)$ be a topologically finite type (see Definition~\ref{defn:huber-topologically-finite-type}) complete $(R, R^+)$-Tate-Huber pair with an $(R, R^+)$-action of a finite group $G$. Then the morphism $(A^G, A^{+, G}) \to (A, A^+)$ is a finite morphism of complete Huber pairs (see Definition~\ref{defn:Huber-finite}). 
\end{lemma}
\begin{proof}
Firstly, we note that Lemma~\ref{lemma:adic-complete} ensures that $(A^G, A^{+, G})$ is a complete Huber-Tate pair, so it makes to ask whether $(A^G, A^{+, G}) \to (A, A^+)$ is a finite morphism of complete Huber pair. \smallskip

Lemma~\ref{lemma:alg-finite-integral} gives that the morphisms $A^G \to A$ and $A^{+, G} \to A^+$ are integral. So we only need to show that $A$ is finite as an $A^G$-module. Lemma~\ref{lemma:tft-useful} (applied to $(R, R^+) \to (A^G, A^{+, G}) \to (A, A^+)$) ensures that $(A^G, A^{+, G}) \to (A, A^+)$ is a topologically finite type morphism of complete Tate-Huber pairs with $A^{+, G} \to A^+$ being integral. Therefore, Lemma~\ref{lemma:top-finite-type-integral-finite} implies that $(A^G, A^{+, G}) \to (A, A^+)$  is finite.
\end{proof}

\begin{lemma}[Analytic Artin-Tate]\label{lemma:adic-Artin-Tate} Let $(R, R^+)$ be a strongly noetherian complete Tate-Huber pair, and $i\colon \left(A, A^+\right) \to \left(B, B^+\right)$ a finite {\it injective} morphism of complete Tate-Huber $(R, R^+)$-pairs. If $\left(B, B^+\right)$ is a topologically finite type $\left(R, R^+\right)$-Tate-Huber pair, then so is $\left(A, A^+\right)$. 
\end{lemma}
The proof of Lemma~\ref{lemma:adic-Artin-Tate} imitates the proof of the Adic Artin-Tate Lemma (Lemma~\ref{lemma:formal-Artin-Tate}), but it is more difficult due to the issue that we need to control the integral aspect of Definition~\ref{defn:Huber-finite}. We recommend the reader to look at the proof of Lemma~\ref{lemma:formal-Artin-Tate} before reading this proof.
\begin{proof}
{\it Step 0. Preparation for the proof}: We choose a pseudo-uniformizer $\varpi \in R$ and an open, surjective morphism 
\[
f\colon R\left\langle X_1, \dots, X_n \right\rangle \twoheadrightarrow B
\]
such that $B^+$ is integral over $f\left(R^+\left\langle X_1, \dots, X_n \right\rangle \right)$. We denote by $x_i\in B^+$ the image $f(X_i)$. \\

{\it Step 1. We choose ``good'' $A$-module generators $y_1, \dots, y_m$ of $B$}: Remark~\ref{rmk:top-finite-type-integral-finite} implies that there is a compatible choice of rings of definition $A_0 \subset A$, $B'_0 \subset B$ containing all $x_i$ such that $B'_0$ is a finite $A_0$-module. Then we choose $A_0$-module generators $y_1, \dots, y_m$ of $B'_0$. Since $B\simeq B'_0\left[\frac{1}{\varpi}\right]$, $A\simeq A_0\left[\frac{1}{\varpi}\right]$, we conclude that $y_1, \dots, y_m$ are also $A$-module generators of $B$. The crucial property of this choice of $A$-module generators is that there exist $a_{i,j}, a_{i,j,k} \in A_0 \subset A^+$ such that 
\[
x_i=\sum_j a_{i,j}y_j,
\]
\[
y_iy_j=\sum_k a_{i,j,k}y_k. 
\]

{\it Step 2. We define another ring of definition $B_0$}: We consider the unique surjective, continuous $R$-algebra homomorphism
\[
g\colon R\langle X_1,\dots, X_n, Y_1, \dots, Y_m, T_{i,j}, T_{i,j,k}\rangle \to B
\]
defined by $g(X_i)=x_i$, $g(Y_j)=y_j$, $g(T_{i,j})=a_{i,j}$, and $g(T_{i,j,k})=a_{i,j,k}$. This morphism is automatically open by Remark~\ref{rmk:open-mapping}.\smallskip

We define $B_0 \coloneqq g(R_0\langle X_1,\dots, X_n, Y_1, \dots, Y_m, T_{i,j}, T_{i,j,k}\rangle)$, where $R_0$ is a ring of definition in $R$ compatible with $A_0$ (see \cite[Corollary 1.3(ii)]{H0} for its existence). This is clearly an open and bounded subring of $B$, so it is a ring of definition.  \smallskip

By construction, $B_0$ contains $f(R_0\langle X_1, \dots, X_n \rangle)$, and $B_0/\varpi B_0$ is generated as an $R_0/\varpi R_0$-algebra by the classes $\ov{x_i}$, $\ov{y_j}$, $\ov{a_{i,j}}$, and $\ov{a_{i,j,k}}$.\\

{\it Step 3. We show that $B^+$ is integral over $R^+B_0$}: We note that $B^+$ is integral over 
\[
    f(R^+\langle X_1, \dots, X_n \rangle)= f(R^+R_0\langle X_1, \dots, X_n \rangle)=R^+f(R_0\langle X_1, \dots, X_n\rangle). 
\]
Therefore, it is integral over $R^+B_0$ since it contains $R^+f(R_0\langle X_1, \dots, X_n\rangle)$ by the previous Step. \\

{\it Step 4. We show that $(B, B^+)$ is finite over $\left(R\langle T_{i,j}, T_{i,j,k}\rangle, R^+ \langle T_{i,j}, T_{i,j,k}\rangle\right)$}: We recall that $a_{i,j}$, $a_{i,j,k} \in A_0 \subset A^+$ for all $i$, $j$, $k$. So, we can use the universal property of restricted power series to define a continuous morphism of complete Tate-Huber pairs:
\[
r\colon \left(R\langle T_{i,j}, T_{i,j,k}\rangle, R^+\langle T_{i,j}, T_{i,j,k}\rangle\right) \to (A, A^+)
\]
as the unique continuous $R$-algebra morphism such that
\[
r(T_{i,j})=a_{i,j}, r(T_{i,j,k})=a_{i,j,k}.
\]
We also define the morphism
\[
t\colon \left(R\langle T_{i,j}, T_{i,j,k}\rangle, R^+ \langle T_{i,j}, T_{i,j,k}\rangle\right) \to (B, B^+)
\]
as the composition of $r$ and $i$. \smallskip

We now show that $B_0$ is finite over $R_0\langle T_{i,j}, T_{i,j,k}\rangle$. Note that this actually makes sense since the natural morphism
\[
R_0\langle T_{i,j}, T_{i,j,k}\rangle \to B
\]
factors through $B_0$ by the choice of $B_0$. We consider the reduction $B_0/\varpi B_0$ and claim that it is finite over 
\[
R_0\langle T_{i,j}, T_{i,j,k}\rangle/\varpi = \left(R_0/\varpi\right)[T_{i,j}, T_{i,j,k}].
\] 
Indeed, we know that $B_0/\varpi B_0$ is generated as an $R_0/\varpi R_0$-algebra by the elements
\[
\ov{x_1}, \dots, \ov{x_n}, \ov{y_1}, \dots, \ov{y_m}, \ov{a_{i,j}}, \ov{a_{i,j,k}}.
\]
However, we note that $\ov{a_{i,j}}=\ov{t(T_{i,j})}$ and $\ov{a_{i,j,k}}=\ov{t(T_{i,j,k})}$. Thus, we can conclude that $B_0/\varpi B_0$ is generated as an $(R_0/\varpi R_0)[T_{i,j}, T_{i,j,k}]$-algebra by the elements 
\[
\ov{x_1}, \dots, \ov{x_n}, \ov{y_1}, \dots, \ov{y_m}.
\]
Recall that the choice of $x_i$ and $y_j$ implies that each of $\ov{x_i}$ is a linear combination of $\ov{y_j}$ with coefficients in $\ov{a_{i,j}}=\ov{t(T_{i,j})}$. This implies that $B_0/\varpi B_0$ is generated as an $(R_0/\varpi R_0)[T_{i,j}, T_{i,j,k}]$-algebra by $\ov{y_1}, \dots, \ov{y_m}$. But again, the same argument shows that each product $\ov{y_i}\ov{y_j}$ can be expressed as a linear combination of $\ov{y_k}$ with coefficients $\ov{a_{i,j,k}}=\ov{t(T_{i,j,k})}$. This implies that $\ov{y_1}, \dots, \ov{y_m}$ are actually $(R_0/\varpi R_0)[T_{i,j}, T_{i,j,k}]$-module generators for $B_0/\varpi B_0$. Now we use a successive approximation argument (or \cite[\href{https://stacks.math.columbia.edu/tag/031D}{Tag 031D}]{stacks-project}) to conclude that $B_0$ is finite over $R_0\langle T_{i,j}, T_{i,j,k}\rangle$. \smallskip

We conclude that $B$ is a finite module over $R\langle T_{i,j}, T_{i,j,k}\rangle $ since 
\[
B= B_0\left[\frac{1}{\varpi}\right], \text{ and } R\langle T_{i,j}, T_{i,j,k} \rangle = R_0\langle T_{i,j}, T_{i,j,k}\rangle \left[\frac{1}{\varpi}\right].
\]

Thus, we are only left to show that $B^+$ is integral over $R^+\langle T_{i,j}, T_{i,j,k}\rangle$. Step~$3$ implies that $B^+$ is integral over $B_0R^+$, so it suffices to show that $B_0R^+$ is integral over $R^+\langle T_{i,j}, T_{i,j,k}\rangle$. But this easily follows from the fact that $B_0$ is finite over $R_0\langle T_{i,j}, T_{i,j,k}\rangle$. \\

{\it Step 5. We show that $(A, A^+)$ is finite over $(R\langle T_{i,j}, T_{i,j,k}\rangle, R^+ \langle T_{i,j}, T_{i,j,k}\rangle)$:} Note that $R\langle T_{i,j}, T_{i,j,k}\rangle$ is noetherian since $R$ is strongly noetherian by assumption. Therefore, we see that $A$ must be a finite $R\langle T_{i,j}, T_{i,j,k}\rangle$-module as a submodule of a finite $R\langle T_{i,j}, T_{i,j,k}\rangle$-module $B$. Moreover, we see that $A^+$ is equal to the intersection $B^+\cap A$ because $(B, B^+)$ is a finite $(A, A^+)$-Tate-Huber pair. This implies that $A^+$ is integral over the image $r(R^+\langle T_{i,j}, T_{i,j,k}\rangle)$. We conclude that the complete Huber pair $(A, A^+)$ is finite over $(R\langle T_{i,j}, T_{i,j,k}\rangle, R^+ \langle T_{i,j}, T_{i,j,k}\rangle)$. Therefore, it is topologically finite type over $(R, R^+)$ by Lemma~\ref{lemma:huber-finite-finite-type} and  Lemma~\ref{lemma:tft-useful}.
\end{proof}

\begin{cor}\label{cor:adic-invariants-finitely-generated} Let $(R, R^+)$ be a strongly noetherian complete Tate-Huber pair and $(A, A^+)$ a topologically finite type complete $(R, R^+)$-Tate-Huber pair with an $(R, R^+)$-action of a finite group $G$. Then the complete Tate-Huber pair $\left(A^G, A^{+, G}\right)$ is topologically finite type over $(R, R^+)$, and the natural morphism $\left(A^G, A^{+, G}\right) \to \left(A, A^+\right)$ is a finite morphism of complete Tate-Huber pairs.
\end{cor}
\begin{proof}
Lemma~\ref{lemma:adic-finite-invariants} gives that $\left(A^G, A^{+, G}\right) \to \left(A, A^+\right)$ is a finite morphism of complete Tate-Huber pairs. So Lemma~\ref{lemma:adic-Artin-Tate} guarantees that $\left(A^G, A^{+G}\right)$ is a topologically finite type complete $\left(R,R^+\right)$-Tate-Huber pair.
\end{proof}

\begin{thm}\label{thm:adic-geometric-quotient-affine} Let $(R, R^+)$ be a strongly noetherian complete Tate-Huber pair and $X=\Spa (A, A^+)$ a topologically finite type affinoid adic $S=\Spa(R, R^+)$-space with an $S$-action of a finite group $G$. Then the natural morphism $\phi\colon X/G \to Y=\Spa \left(A^G, A^{+G}\right)$ is an isomorphism over $S$. In particular, $X/G$ is a topologically finite type affinoid adic $S$-space. 
\end{thm}
We adapt the proofs of Proposition~\ref{prop:alg-geometric-quotient-affine}~and~\ref{prop:formal-geometric-quotient-affine}. However, there are certain complications due to the presence of higher rank points. Namely, there are usually many different points $v\in \Spa\left(A, A^{+}\right)$ with the same support $\mathfrak p$. Thus in order to study fibers of the map $X \to Y$ we need to work harder than in the algebraic and formal setups. 
\begin{proof}
{\it Step 0. Preparation}:
The $S$-action of $G$ on $\Spa(A, A^+)$ induces an $(R, R^+)$-action of $G$ on $(A, A^+)$. By Corollary~\ref{cor:adic-invariants-finitely-generated},  $(A^G, A^{+, G})$ is topologically finite type over $(R, R^+)$. In particular, $Y=\Spa(A^G, A^{+, G})$ is an adic space\footnote{The structure presheaf $\O_Y$ is a sheaf on $Y$ by \cite[Theorem 2.5]{H1}.}, and it is topologically finite type over $S$. 

Now we recall that there is a natural map of valuative spaces $p'\colon \Spv A \to \Spv A^G$, where $\Spv A$ (resp. $\Spv A^G$) is the set of {\it all} valuations on the ring $A$ (resp. $A^G$). We have the commutative diagram
\[
\begin{tikzcd}
\Spa\left(A, A^+\right) \arrow{r}{p} \arrow{d} & \Spa\left(A^G, A^{+G}\right) \arrow{d}\\
\Spv A \arrow{r}{p'} \arrow{d} & \Spv A^G \arrow{d}\\
\Spec A \arrow{r}{p''}  & \Spec A^G
\end{tikzcd}
\]
with the upper vertical maps being the forgetful maps and the lower vertical maps being the maps that send a valuation to its support.
\smallskip

{\it Step 1. The natural map $p'\colon \Spv A \to \Spv A^G$ is surjective and $G$ acts transitively on fibers:} Recall that data of a valuation $v\in \Spv A$ is the same as data of a prime ideal $\mathfrak p_v \subset A$ (its support) and a valuation ring $R_v \subset k(\mathfrak p)$. \smallskip 

To show surjectivity of $p'$, pick any valuation $v\in \Spv A^G$; we want to lift it to a valuation of $A$. We use Lemma~\ref{lemma:alg-finite-integral} to find a prime ideal $\mathfrak q \subset A$ that lifts the support 
\[
\mathfrak p_v\coloneqq \supp(v) \subset A^G,
\]
so $k(\frak{q})$ is finite over $k(\frak{p_v})$ since $A$ is $A^G$-finite. \smallskip

Now we use \cite[Theorem 10.2]{M1} to dominate the valuation ring $R_v \subset k(\mathfrak p_v)$ by some valuation ring $R_{w} \subset k(\mathfrak q)$. This provides us with a valuation $w\colon A \to \Gamma_{w}\cup \{0\}$ such that $p'(w)=v$. Therefore, the map $p'$ is surjective. 

As for the transitivity of the $G$-action, we note that Lemma~\ref{lemma:alg-finite-integral} implies that $G$ acts transitively on the fiber $(p'')^{-1}(\mathfrak p_v)$. Furthermore,  \cite[Ch.5, \textsection 2, n.2, Theorem 2]{Bou} guarantees that, for any prime ideal $\mathfrak q \in (p'')^{-1}(\mathfrak p_v)$, the stabilizer subgroup 
\[
G_{\mathfrak q}\coloneqq \operatorname{Stab}_G({\mathfrak q})
\]
surjects onto the automorphism group $\Aut(k(\mathfrak q)/ k(\mathfrak p_v))$. We use \cite[Ch. 6, \textsection 8, n.6, Proposition 6]{Bou} to see that there is a bijection between the sets

\begin{equation*}
\left\{\begin{array}{l} \text{Valuations }w  \text{ on } k(\frak{q})\\  \text{ restricting to } v \text{ on }  k(\frak{p_v}) \end{array}\right\} \leftrightarrow \left\{ \text{ Maximal ideals in } \operatorname{Nr}_{k(\mathfrak q)}(R_v)\right\},
\end{equation*}
where $\operatorname{Nr}_{k(\mathfrak q)}(R_v)$ is the integral closure of $R_v$ in the field $k(\mathfrak q)$. Now we use \cite[Theorem 9.3(iii)]{M1} to conclude that $\Aut(k(\mathfrak q)/k(\mathfrak p_v))$ (and therefore $G_{\mathfrak q}$) acts transitively on the set of maximal ideals of $\operatorname{Nr}_{k(\mathfrak q)}(R_v)$. As a consequence, $G_{\mathfrak q}$ acts transitively on the set of valuations $w \in p'^{-1}(v)$ with the support $\mathfrak q$. Therefore, $G$ acts transitively on $p'^{-1}(v)$ for any $v\in \Spv A^G$. \\

{\it Step 2. We show that $p\colon \Spa(A, A^+) \to \Spa(A^G, A^{+, G})$ is surjective, and $G$ acts transitively on fibers}: We recall that $\Spa(A, A^{+})$ (resp. $\Spa(A^G, A^{+, G})$) is naturally a subset of $\Spv(A)$ (resp. $\Spv(A^G)$). Therefore, it suffices (by Step~$1$) to show that, for any $v\in \Spa(A^G, A^{+, G})$, any $w\in p'^{-1}(v)$ is continuous and $w(A^+)\leq 1$. \smallskip

It is clear that $w(A^+)\leq 1$ as $A^+$ is integral over $A^{+, G}$. So we only need to show that the valuation $w\in \Spv(A)$ is continuous.

\begin{lemma}\label{cont} Let $A$ be a Tate ring with a pair of definition $(A_0, \varpi)$, where $\varpi$ is a pseudo-uniformizer. Then a valuation $v\colon A \to \Gamma_v \cup \{0\}$ is continuous if and only if:
\begin{itemize}\itemsep0.5em
\item The value $v(\varpi)$ is cofinal in $\Gamma_v$,
\item $v(a\varpi) < 1$ in $\Gamma_v$ for any $a\in A_0$.
\end{itemize}
\end{lemma}
\begin{proof}
\cite[Corollary 9.3.3]{Seminar}.
\end{proof}

We choose a $G$-stable pair of definition $(A_0, \varpi)$ from Lemma~\ref{lemma:adic-complete}. Then \cite[Ch. 6, \textsection 8, n.1, Proposition 1]{Bou} gives that $\Gamma_w/\Gamma_v$ is a torsion group. Therefore $w(\varpi)=v(\varpi)$ is cofinal in $\Gamma_w$ if it is cofinal in $\Gamma_v$. In particular, $w(\varpi) < 1$. \smallskip

Now we verify the second condition in Lemma~\ref{cont}. Since $w(A^+)\leq 1$  and $v|_{A^{+, G}}=w|_{A^{+, G}}$, 
\[
w(a\varpi)=w(a)w(\varpi)<w(a)\leq 1
\]
for any $a\in A_0 \subset A^+$. \\

{\it Step 3. We show that $\phi \colon X/G \to Y=\Spa(A^G, A^{+, G})$ is a homeomorphism:} Step~$2$ implies that $\phi$ is a bijection. Now note that $p\colon X \to Y$ is a finite, surjective morphism of strongly noetherian adic spaces. Therefore, it is closed by \cite[Lemma 1.4.5(ii)]{H3}. In particular, it is a topological quotient morphism. The map $\pi\colon X \to X/G$ is a topological quotient morphism by its construction. Hence, $\phi$ is a homeomorphism. \\

{\it Step 4. We show that $\phi$ is an isomorphism of topologically locally $v$-ringed spaces}: Firstly, Remark~\ref{rmk:huber-forgetful-conservative} implies that it suffices to show that the natural morphism
\[
\O_Y \to \phi_*\O_{X/G}
\]
is an isomorphism of sheaves of topological rings. Using the basis of rational subdomains in $Y$, it suffices to show that
\begin{equation}\label{eqn:adic-affine-quotient}
A^G\left\langle \frac{f_1}{s}, \dots \frac{f_n}{s} \right\rangle \to \left(A\left\langle \frac{f_1}{s}, \dots, \frac{f_n}{s}\right\rangle\right)^G
\end{equation}
is a topological isomorphism for any $f_1, \dots, f_n, s$ generating the unit ideal in $A^G$. Lemma~\ref{lemma:adic-finite-invariants} gives that (\ref{eqn:adic-affine-quotient}) is a continuous morphism of complete Tate rings. So the Banach Open Mapping Theorem \cite[Lemma 2.4 (i)]{H1} guarantees that it is automatically open (and so a homeomorphism) if it is surjective. Thus, we can ignore the topologies. \smallskip

Now we show that (\ref{eqn:adic-affine-quotient}) is an (algebraic) isomorphism. We note that 
\[
A\otimes_{A^G} A^G\left\langle \frac{f_1}{s}, \dots, \frac{f_n}{s}\right\rangle \simeq A\left\langle \frac{f_1}{s}, \dots, \frac{f_n}{s}\right\rangle,
\]
by Corollary~\ref{cor:complete-localization}.  Therefore, it suffices to show that 

\[
A^G\left\langle \frac{f_1}{s}, \dots \frac{f_n}{s} \right\rangle \to \left(A\otimes_{A^G} A^G\left\langle \frac{f_1}{s}, \dots, \frac{f_n}{s}\right\rangle\right)^G
\]
is an isomorphism. This follows from Lemma~\ref{lemma:invariants-commute-flat-base-change} and flatness of the map $A^G \to A^G\left\langle \frac{f_1}{s}, \dots, \frac{f_n}{s}\right\rangle$ obtained in \cite[Case II.1.(iv) on p. 530]{H1}.
\end{proof}

\subsection{General Case}\label{proof-analytic} 
The main goal of this section is to globalize the results of the previous section. This is very close to what we did in the formal situation in the proof of Theorem~\ref{thm:formal-main}. The main issue is that the adic analog of Lemma~\ref{lemma:formal-G-basis} is more difficult to show. \smallskip

For the rest of the section, we fix a locally strongly noetherian analytic adic space $S$ (see Definition~\ref{defn:strongly-noetherian-space}). \smallskip

\begin{lemma}\label{lemma:adic-quasi-affinoid} Let $X=\Spa(A, A^+)$ be a pre-adic Tate affinoid\footnote{We do not assume that the structure presheaf $\O_{X}$ is a sheaf.}, and $V\subset X$ an open pre-adic subspace. Then any finite set of points $F \subset V$ is contained in an affinoid pre-adic subspace of $V$.
\end{lemma}
Our proof uses an adic analogue of the theory of ``formal'' models of rigid spaces in an essential way. It might be possible to justify this claim directly from first principles, but it seems quite difficult due to the fact that the complement $X\setminus V$ does not have a natural structure of a pre-adic space.  \smallskip

In what follows, for any topological space $Z$ with a map $Z \to \Spec A^+$, we denote by $\ov{Z}$ the fiber product $Z\times_{\Spec A^+} \Spec A^+/\varpi$ in the category of topological spaces. 
\begin{proof}
First of all, we note that rational subdomains form a basis of the topology of $\Spa(A, A^+)$, and they are quasi-compact. Therefore, we can find a quasi-compact open subspace $F\subset V' \subset V$, so we may and do assume that $V$ is quasi-compact. \smallskip

We consider the affine open immersion
\[
U=\Spec A \to S=\Spec A^+.
\]
And define the category of {\it $U$-admissible modifications} $\mathbf{Adm}_{U,S}$ to be the category of projective morphisms\footnote{We emphasize that a projective morphism is not required to be finitely presented} $f\colon Y \to S$ that are isomorphisms over $U$. Then \cite[Theorem 8.1.2 and Remark 8.1.8]{Bhatt-notes} (alternatively, one can adapt the proof of \cite[Theorem A.4.7]{FujKato} to this situation) shows that 
\[
\ov{X} \coloneqq \left(\lim_{f\colon Y\to \Spec A^+ \ | \ f\in \mathbf{Adm}_{U,S}} Y \right) \times_{\Spec A^+} \Spec A^+/\varpi \simeq  \lim_{f\in \mathbf{Adm}_{U,S} } \ov{Y}
\]
admits a canonical morphism $\ov{X} \to \Spa(A, A^+)$ that is a homeomorphism. Since $V \subset X$ is quasi-compact, \cite[\href{https://stacks.math.columbia.edu/tag/0A2P}{Tag 0A2P}]{stacks-project} implies that there is a $U$-admissible modification $Y\to \Spec A^+$ and a quasi-compact open $V'\subset Y$ such that $\pi^{-1}(\ov{V}') = V$ for the projection map $\pi\colon \ov{X} \to \ov{Y}$. \smallskip

Now $\ov{V}'$ is a quasi-projective scheme over $\Spec A^+/\varpi$. Therefore, \cite[Corollaire 4.5.4]{EGA2} implies that there is an open affine subscheme $\ov{W}\subset \ov{V}'$ containing $\pi(F)$. Therefore, $\pi^{-1}(\ov{W}) \subset \Spa(A, A^+)$ contains $F$, and (the proof of) \cite[Corollary 8.1.7]{Bhatt-notes} implies that $\pi^{-1}(\ov{W})$ is affinoid\footnote{The inverse limit giving the preimage in the statement is shown to be affinoid in the proof.}. 
\end{proof}

\begin{lemma}\label{lemma:adic-open-G-orbit-preserved} Let $X$ be a pre-adic space with an action of a finite group $G$. Suppose that each point $x\in X$ admits an open affinoid pre-adic subspace $V_x$ that contains the orbit $G.x$. Then the same holds with $X$ replaced by any $G$-stable open pre-adic subspace $U \subset X$.
\end{lemma}
\begin{proof}
The proof is analogous to that of Lemma~\ref{lemma:alg-open-G-orbit-preserved}. One only needs to use Lemma~\ref{lemma:adic-quasi-affinoid} in place of \cite[Corollaire 4.5.4]{EGA2}.
\end{proof}

\begin{lemma}\label{lemma:adic-G-basis} Let $X$ be a locally topologically finite type adic $S$-space with an $S$-action of a finite group $G$. Suppose that for any point $x\in X$ there is an open affinoid adic subspace $V_x \subset X$ that contains the orbit $G.x$. Then each point $x\in X$ has a $G$-stable strongly noetherian Tate affinoid neighborhood $U_x \subset X$.
\end{lemma}
\begin{proof}
The proof is similar to that of Lemma~\ref{lemma:alg-open-G-orbit-preserved}. Lemma~\ref{lemma:adic-open-G-orbit-preserved} allows to reduce to the case $S$ a strongly noetherian Tate affinoid space. Then for a separated $X$, $U_x\coloneqq \bigcap_{g\in G} g(V_x)$ is a strongly noetherian Tate affinoid (see Corollary~\ref{lemma:intersection-affinoids}) and does the job. In general, Lemma~\ref{lemma:adic-open-G-orbit-preserved} guarantees that one can replace $X$ with the {\it separated} open adic subspace $\bigcap_{g\in G} g(V_x)$ and reduce to the separated case. 
\end{proof}

\begin{thm}\label{thm:adic-main} Let $X$ be a locally topologically finite type adic $S$-space with an $S$-action of a finite group $G$. Suppose that each point $x\in X$ admits an affinoid open neighborhood $V_x$ containing $G.x$. Then $X/G$ is a locally topologically finite type adic $S$-space. Moreover, it satisfies the following properties:
\begin{enumerate}
\item $\pi\colon X \to X/G$ is universal in the category of $G$-invariant morphisms to topologically locally $v$-ringed $S$-spaces. 
\item $\pi: X \to X/G$ is a finite, surjective morphism (in particular, it is closed). 
\item Fibers of $\pi$ are exactly the $G$-orbits.
\item\label{thm:adic-main-4} The formation of the geometric quotient commutes with flat base change, i.e. for any flat morphism $Z\to X/G$ (see Definition~\ref{defn:huber-flat}) of locally strongly noetherian analytic adic spaces, the geometric quotient $(X \times_{X/G} Z)/G$ is an adic space, and the natural morphism $(X \times_{X/G} Z)/G \to Z$ is an isomorphism.
\end{enumerate}
\end{thm} 

\begin{proof}
{\it Step 1. $X/G$ is a topologically locally finite type adic $S$-space}: Similarly to Step~$1$ of Theorem~\ref{thm:alg-main}, we can use Lemma~\ref{lemma:adic-open-G-orbit-preserved} and Lemma~\ref{lemma:adic-G-basis} to reduce to the case of a strongly noetherian Tate affinoid $S=\Spa(R, R^+)$ and affinoid $X=\Spa(A, A^+)$. Then the claim follows from Theorem~\ref{thm:adic-geometric-quotient-affine}. \smallskip

{\it Step 2. $\pi: X \to X/G$ is surjective,  finite, and fibers are exactly the $G$-orbits}: Similarly to Step~$1$, we can assume that $X$ and $S$ are affinoid. Then it follows from Lemma~\ref{lemma:adic-finite-invariants} and Theorem~\ref{thm:adic-geometric-quotient-affine}. \smallskip

{\it Step 3. $\pi\colon X \to X/G$ is universal and commutes with flat base change}: The universality is essentially trivial (Remark~\ref{rmk:universal-locally rings}). To show the latter claim, we can again assume that $X=\Spa (A, A^+)$ and $S=\Spa (R, R^+)$ are strongly noetherian Tate affinoids and it suffices to consider strongly noetherian Tate affinoid $Z=\Spa(B, B^+)$. Then the construction of the quotient implies that it suffices to show that the natural morphism of Tate-Huber pairs
\begin{equation}\label{eqn:flat-base-change-affinoid}
    (B, B^+) \to \left((B, B^+)\wdh{\otimes}_{(A^G, A^{+, G})} (A, A^+)\right)^G=:(C, C^+)
\end{equation}
is a topological isomorphism. \smallskip

We can ignore the topologies to show that $B \to B\wdh{\otimes}_{A^G} A$ is a topological isomorphism since its surjectivity would imply openness by the Banach Open Mapping Theorem \cite[Lemma 2.4 (i)]{H1}. Now Corollary~\ref{cor:adic-invariants-finitely-generated} and Lemma~\ref{complete-base-change} ensure that $B\wdh{\otimes}_{A^G} A \simeq B\otimes_{A^G} A$. Therefore, it suffices to show that the natural map
\[
B \to \left(B\otimes_{A^G} A\right)^G
\]
is an (algebraic) isomorphism. This follows from Lemma~\ref{lemma:invariants-commute-flat-base-change} and flatness of $A \to B$ justified in Lemma~\ref{lemma:huber-flat-affinoids}. 
\end{proof}

\begin{rmk}\label{rmk:referee} As the anonymous referee pointed out, the condition that each point $x\in X$ admits an affinoid open neighborhood $V_x$ containing $G.x$ may be automatic under a very mild assumption on $X$ (it is probably enough to assume that $X$ is separated over $S$). In the lemma below, we give a proof of this claim for separated rigid-analytic spaces. A more general version of this result would seem to require a generalization of the main results of \cite{temkin-local} to more general adic spaces. This is beyond the scope of this paper.
\end{rmk}

The next lemma uses Berkovich spaces in its proof; the reader unfamiliar with Berkovich spaces can safely ignore this lemma as it is never used in the rest of the paper. We heartfully thank Michael Temkin for suggesting the idea of the following argument:

\begin{lemma}\label{lemma:mild-assumption} Let $K$ be a non-archimedean field with the residue field $k$, $X$ a separated, locally finite type adic $\Spa(K, \O_K)$-space, and $\{x_1, \dots, x_n\}$ is a finite set of points of $X$. Then there is an open affinoid subset $U\subset X$ containing all $x_i$.
\end{lemma}
\begin{proof}
    Firstly, we can replace $X$ with a quasi-compact open subset containing all $x_i$ to reduce to the case of a quasi-compact (and separated) $X$. Then the underlying topological space of $X$ is spectral, so \cite[\href{https://stacks.math.columbia.edu/tag/0904}{Tag 0904}]{stacks-project} implies that it suffices to consider the case when all $x_i$ have the same (unique) rank-$1$ generalization $x\in X$. \smallskip
    
    Now since $X$ is separated and quasi-compact, \cite[Lemma 5.1.3]{H3} implies that $X$ is taut in the sense of \cite[Definition 5.1.2]{H3}. Therefore, \cite[Proposition 8.3.1 and Remark 8.3.2]{H3} (or \cite[Construction 7.1]{Henkel}) constructs the associated Hausdorff $K$-strict Berkovich space $X^{\rm{Ber}}$ with the continuous map of underlying topological spaces
    \[
    \omega \colon X \to X^{\rm{Ber}}
    \]
    that sends a point  $x$ to its unique rank-$1$ generalization. In what follows, we will slightly abuse the notation and consider $x$ as both the point of $X$ and $X^{\rm{Ber}}$ (this is essentially harmless because $X^{\rm{Ber}}$ is set-theoretically equal to the set of rank-$1$ points of $X$). \smallskip
    
    We consider the germ $X^{\rm{Ber}}_x$ (see \cite[\textsection 3.4]{Berkovich}) and its reduction $\widetilde{X}^{\rm{Ber}}_x=(V_{\widetilde{x}}, \widetilde{\cal{H}(x)}, \varepsilon)$ (considered as an object of $\rm{bir}_k$, see \cite[p.4]{temkin-local}) that is defined just after \cite[Lemma 2.1]{temkin-local} (see also \cite[Definition 5.7.2.10]{temkin-notes}). Geometrically, the underlying topological space of the reduction $\widetilde{X}_x^{\rm{Ber}}$ is identified with the fiber $\omega^{-1}(\omega(x))$ (see \cite[Remark 2.6]{temkin-local}). In particular, points $x_i$ uniquely correspond to points of $\widetilde{X}^{\rm{Ber}}_x$ that we also denote by $x_i$ (by slight abuse of notation). \smallskip
    
    Now we note that \cite[Fact 5.2.2.4]{temkin-notes} and \cite[Remarks 1.3.18 and 1.3.19]{H3} imply that $X^{\rm{Ber}}$ is a separated Berkovich space since $X$ is a separated adic space. Therefore, \cite[Proposition 2.5]{temkin-local} implies that the reduction $\widetilde{X}^{\rm{Ber}}_x$ is separated (see \cite[Definition/Exercise 5.7.2.8(iv)]{temkin-notes}). In other words, when $\widetilde{X}^{\rm{Ber}}_x$ is considered as an object of $\rm{Bir}_k$ (see \cite[\textsection 2]{temkin-local} for the precise definition and \cite[Proposition 1.4]{temkin-local} for the equivalence between $\rm{Bir}_k$ and $\rm{bir}_k$), the underlying birational equivalence class of $k$-schemes $\widetilde{X}^{\rm{Ber}}_x$ is separated (i.e. any representative is separated over $k$). Therefore, Chow's lemma \cite[\href{https://stacks.math.columbia.edu/tag/0200}{Tag 0200}]{stacks-project} implies that we can find a representative $Y$ of $\widetilde{X}^{\rm{Ber}}_x$ that is quasi-projective over $k$. 
    
    Now we recall that the underlying topological space of $\widetilde{X}^{\rm{Ber}}_x$ is equal to the cofiltered limit of all representatives of the birational equivalence class $\widetilde{X}^{\rm{Ber}}_x$ (see \cite[Corollary 1.3]{temkin-local}). In particular, each point $x_i\in \widetilde{X}^{\rm{Ber}}_x$ (uniquely) defines a point $y_i\in Y$. Then \cite[\href{https://stacks.math.columbia.edu/tag/01ZY}{Tag 01ZY}]{stacks-project} implies that there is an open affine subset $V_Y\subset Y$ containing all $y_i$. Then (by \cite[Definition/Exercise 5.7.2.8(ii)]{temkin-notes}) it gives rise to an affine open subspace $V\subset \widetilde{X}^{\rm{Ber}}_x$. Now \cite[Theorem 2.4 and Theorem 3.1]{temkin-local} show that $V$ defines a {\it good} subdomain $W^{\rm{Ber}}_x\subset X^{\rm{Ber}}_x$ (see \cite[p.7]{temkin-local} for the definition of a good germ and of a subdomain of a germ). Essentially by definition, this gives rise to a subdomain $x\in W^{\rm{Ber}} \subset X^{\rm{Ber}}$ such that $W$ is a good Berkovich space. Now the subdomain $W^{\rm{Ber}}$ defines an open subspace $W\subset X$ (by applying the functor $(-)^{\rm{ad}}$ from \cite[Construction 7.5]{Henkel}) that contains all the points $x_1, \dots, x_n$ by its construction. Therefore, we may replace $X$ with $W$ to assume that $X^{\rm{Ber}}$ is good. \smallskip
    
    Finally, if $X^{\rm{Ber}}$ is good, there is an open affinoid subspace $x\in U^{\rm{Ber}} \subset X^{\rm{Ber}}$ by the very definition of goodness\footnote{The reader unfamiliar with Berkovich spaces may find it strange, but it is {\it not} true that a point of a general Berkovich space contains an open affinoid neighborhood.}. Then $U\coloneqq \omega^{-1}(U^{\rm{Ber}})$ is an open affinoid subspace of $X$ containing $\ov{\{x\}}$. In particular, it contains all point $x_1, \dots, x_n\in \ov{\{x\}}$. 
\end{proof}

\subsection{Comparison of adic quotients and formal quotients}\label{comp-adic-formal}

For this section, we fix a complete, microbial valuation ring $k^+$ (see Definition~\ref{defn:formal-microbial-valuation}) with fraction field $k$, and a choice of a pseudo-uniformizer $\varpi$. \smallskip

We consider the functor of adic generic fiber: 
\begin{equation*}
(-)_k\colon \left\{\begin{array}{l}\text{admissible formal} \\  \ \ \ \ k^+\text{-schemes}\end{array}\right\} \to \left\{\begin{array}{l} \text {Adic Spaces locally of topologically} \\ \text{finite type over}\  \Spa(k, k^+)\end{array}\right\}
\end{equation*}
that is defined in \cite[\textsection 1.9]{H3} (it is denoted by $d$ there). To an affine admissible formal $k^+$-scheme $\Spf(A)$, this functor assigns the affinoid adic space $\Spa(A\left[\frac{1}{\varpi}\right], A^+)$ where $A^+$ is the integral closure of $A$ in $A\left[\frac{1}{\varpi}\right]$. \smallskip

Let $\X$ be an admissible formal $k^+$-scheme with a $k^+$-action of a finite group $G$. Then $\X_k$ is a locally topologically finite type adic $\Spa(k, k^+)$-space with a $\Spa(k, k^+)$-action of $G$. Using the universal property of geometric quotients, there is a natural morphism $\X_k/G \to (\X/G)_k$.

\begin{thm}\label{thm:comparison-formal-adic} Let $\X$ be an admissible formal $k^+$-scheme with a $k^+$-action of a finite group $G$. Suppose that any orbit $G.x \subset \X$ lies in an affine open subset. Then the adic generic fiber $\X_k$ with the induced $\Spa(k, k^+)$-action of $G$ satisfies the assumption of Theorem~\ref{thm:adic-main}, and the natural morphism:
\[
\X_k/G \to (\X/G)_k
\]
is an isomorphism.
\end{thm}

\begin{proof} Similarly to Step~$1$ in the proof of Theorem~\ref{thm:comparison-formal-alg}, the condition of Theorem~\ref{thm:adic-main} is satisfied for $\X_k$ with the induced action of $G$. \smallskip

To show that $\X_k/G \to (\X/G)_k$ is an isomorphism, similarly to Step~$2$ in the proof of Theorem~\ref{thm:comparison-formal-alg} we can assume that $\X=\Spf A$ is affine. Then we have to show that the natural map
\[
\left(A^G\left[\frac{1}{\varpi}\right], (A^{G})^+\right) \to \left(A\left[\frac{1}{\varpi}\right]^G, (A^{+})^G \right) 
\]
is an isomorphism of Tate-Huber pairs. \smallskip

Lemma~\ref{lemma:invariants-commute-flat-base-change} implies that the map $A^G\left[\frac{1}{\varpi}\right] \to A\left[\frac{1}{\varpi}\right]^G$ is an algebraic isomorphism. This is a topological isomorphism since both sides have $A^G$ as a ring of definition (see Corollary~\ref{cor:formal-invariants-top-finitely-generated} and Lemma~\ref{lemma:adic-complete}). Therefore, we are only left to show that the natural map $(A^G)^+ \to (A^+)^G$ is an (algebraic) isomorphism.\smallskip

Clearly, $A^+$ is integral over $A$, and so it is integral over $A^G$ by Lemma~\ref{lemma:alg-finite-integral}(\ref{lemma:alg-finite-integral-1}). Hence, $(A^+)^G$ is integral over $(A^G)^+$. Since $(A^G)^+$ is integrally closed in $A\left[\frac{1}{\varpi}\right]^G=A^G\left[\frac{1}{\varpi}\right]$, we conclude that $(A^G)^+=(A^+)^G$.
\end{proof}

\subsection{Comparison of adic quotients and algebraic quotients}\label{comp-adic-alg}

For this section, we fix a complete, rank-$1$ valuation ring $\O_K$ with fraction field $K$, and a choice of a pseudo-uniformizer $\varpi$. 

A {\it rigid space} over $K$ always means here an adic space locally topologically finite type over $\Spa(K, \O_K)$. When we need to use classical rigid-analytic spaces, we refer to them as Tate rigid-analytic spaces. \smallskip

In what follows, for any topologically finite type $K$-algebra $A$, we define $\Sp A\coloneqq \Spa(A, A^\circ)$. We note that \cite[Lemma 4.4]{H1} implies that for any affinoid space $\Spa(A, A^+)$ that is topologically finite type over $\Spa(K, \O_K)$, we have $A^+=A^\circ$. So this notation does not cause any confusion. \smallskip

We consider the analytification functor: 
\begin{equation*}
(-)^{\operatorname{an}}\colon \left\{\begin{array}{l}\text{locally finite type} \\  \ \ \ \ K\text{-schemes}\end{array}\right\} \to \left\{\begin{array}{l} \text {Rigid Spaces} \\ \ \ \ \ \text{over}\  K\end{array}\right\}
\end{equation*}
that is defined as a composition 
\[
(-)^{\an}=r_K \circ (-)^{\text{rig}}
\]
of the classical analytification functor $(-)^{\text{rig}}$ (as it is defined in \cite[\textsection 5.4]{B} and the functor $r_K$ that sends a Tate rigid space to the associated adic space (see \cite[\textsection 4]{H1}). \smallskip

The main issue with the analytification functor is that it does not send affine schemes to affinoid spaces. More precisely, the analytification of an affine scheme $X=\Spec K\left[T_1, \dots, T_d \right]/I$ is canonically isomorphic to
\begin{equation}\label{eqn:affine-union-affinoid}
\bigcup_{n=0}^{\infty} \Sp\left(\frac{K\left\langle \varpi^nT_1, \dots, \varpi^nT_d\right\rangle}{I\cdot K\left\langle \varpi^nT_1, \dots, \varpi^nT_d\right\rangle}\right)
\end{equation}
by the discussion after the proof of \cite[Lemma 5.4/1]{B}. In particular, $\mathbf{A}^{n, \rm{an}}_K$ is not affinoid as it is not quasi-compact. \smallskip

\begin{lemma}\label{lemma:invariants-local} Let $X$ be a rigid space over $K$ with a $K$-action of a finite group $G$. Suppose that each point $x\in X$ admits an affinoid open neighborhood $V_x$ containing $G.x$. Then, for any classical point $\ov{x}\in X/G$, the natural map
\[
\O_{X/G, \ov{x}} \to \left(\prod_{x\in \pi^{-1}(\ov{x})} \O_{X, x}\right)^G.
\]
is an isomorphism. 
\end{lemma}
\begin{proof}
Theorem~\ref{thm:adic-main} gives that $X/G$ is a rigid space over $K$. Lemma~\ref{lemma:adic-G-basis} implies that we can assume that $X=\Sp A$ is affinoid, so $X/G\simeq \Sp A^G$ by Theorem~\ref{thm:adic-geometric-quotient-affine} (\cite[Lemma 4.4]{H1}  guarantees the equality of $+$-rings). \smallskip

Now, \cite[Corollary 4.1/5]{B} implies that $A^G \to \O_{X/G, \ov{x}}$ is flat. Therefore, Lemma~\ref{lemma:invariants-commute-flat-base-change} ensures that 
\[
\O_{X/G, \ov{x}} \simeq (A\otimes_{A^G} \O_{X/G, \ov{x}})^G.
\]
Finally, we note that the natural map
\[
A\otimes_{A^G} \O_{X/G, \ov{x}} \to \prod_{x\in \pi^{-1}(\ov{x})} \O_{X, x}
\]
is an isomorphism by \cite[Theorem A.1.3]{C-ample}. 
\end{proof}

\begin{cor}\label{cor:invariants-local} Let $X$ be a rigid space over $K$ with a $K$-action of a finite group $G$. Suppose that each point $x\in X$ admits an affinoid open neighborhood $V_x$ containing $G.x$. Then, for any classical point $\ov{x}\in X/G$, the natural map
\[
\wdh{\O}_{X/G, \ov{x}} \to \left(\prod_{x\in \pi^{-1}(\ov{x})} \wdh{\O}_{X, x}\right)^G.
\]
is an isomorphism. 
\end{cor}
\begin{proof}
By \cite[Theorem A.1.3]{C-ample}, each $\O_{X, x}$ is $\O_{X/G,\ov{x}}$-finite. Therefore, $\O_{X, x}/\m_{\ov{x}}\O_{X, x}$ is an artinian $k(\ov{x})$-algebra. Thus, there is $n_x$ such that $\m_{x}^{n_x} \subset \m_{\ov{x}}\O_{X,x}$. This implies that the $\m_{x}$-adic topology on $\O_{X, x}$ coincides with the $\m_{\ov{x}}\O_{X, x}$-adic topology. \smallskip

Therefore, using that $\O_{X/G, \ov{x}}$ is noetherian \cite[Proposition 4.1/6]{B} and $\O_{X, x}$ is finite as an $\O_{X/G, \ov{x}}$-module, we conclude that 
\[
    \prod_{x\in \pi^{-1}(\ov{x})} \wdh{\O}_{X, x} \simeq \left(\prod_{x\in \pi^{-1}(\ov{x})} \O_{X, x}\right) \otimes_{\O_{X/G, \ov{x}}} \wdh{\O}_{X/G, \ov{x}}.
\]
The claim now follows from Lemma~\ref{lemma:invariants-commute-flat-base-change} and Lemma~\ref{lemma:invariants-local}.
\end{proof}

\begin{thm}\label{thm:comparison-alg-adic} Let $X$ be a locally finite type $K$-scheme with a $K$-action of a finite group $G$. Suppose that any orbit $G.x \subset X$ lies in an affine open subset. Then the analytification $X^{\rm{an}}$ with the induced $K$-action of $G$ satisfies the assumption of Theorem~\ref{thm:adic-main}, and the natural morphism:
\[
\phi\colon X^{\rm{an}}/G \to (X/G)^{\rm{an}}
\]
is an isomorphism.
\end{thm}
\begin{proof}
\noindent{\it Step 1. The condition of Theorem~\ref{thm:adic-main} is satisfied for $X^{\rm{an}}$ with the induced action of $G$}: Firstly, Lemma~\ref{lemma:alg-G-basis} says that our assumption on $X$ implies that there is a covering of $X=\cup_{i\in I} U_i$ by affine open $G$-stable subschemes. Then $X^{\rm{an}}=\cup_{i\in I} U^{\rm{an}}_i$ is an open covering by $G$-stable adic subspaces. \smallskip

Now we choose $i$ such that $x\in U^{\rm{an}}_i$. We note that (\ref{eqn:affine-union-affinoid}) implies that each $U^{\rm{an}}_i$ can be written as an ascending union
\[
U^{\rm{an}}_i = \bigcup_{n=0}^{\infty} U^{(n)}_i
\]
of open affinoid subspaces. Since the orbit $G.x$ is finite, it is contained in some $U^{(n)}_i$. \smallskip



{\it Step 2. $\phi$ is a bijection on classical points}: Classical points of $(X/G)^{\rm{an}}$ are identified with $(X/G)_0$ the set of closed points of $X/G$ (see \cite[Lemma 5.1.2(1)]{C-irreducible}). Thus the noetherian analogue of Theorem~\ref{thm:formal-main} (or Theorem~\ref{thm:universally-adhesive-main}) and the fact that finite morphisms reflect closed points imply that classical points of $(X/G)^{\rm{an}}$ are identified with $X_0/G$. \smallskip

Similarly, we can use Theorem~\ref{thm:adic-main} and the fact that $\pi_{X^{\rm{an}}}$ reflects classical points (as being finite) to conclude that the classical points of $X^{\rm{an}}/G$ are identified with the set $X_0/G$. \smallskip

{\it Step 3. We reduce to a claim on completed local rings}: We consider the commutative diagram:
\[
\begin{tikzcd}
X^{\rm{an}} \arrow{d}{\pi_{X^{\rm{an}}}}  \arrow{rd}{\pi_X^{\rm{an}}} &\\
X^{\rm{an}}/G \arrow{r}{\phi} &  (X/G)^{^{\rm{an}}}.
\end{tikzcd}
\]
Since $\pi_X^{\rm{an}}$ is a finite, surjective, $G$-equivariant morphism, we conclude that $\phi$ is a finite, surjective morphism by Proposition~\ref{prop:quot-prop-adic} (that is independent of the present section). Therefore, $\phi_*\left(\O_{X^{\rm{an}/G}}\right)$ is a coherent $\O_{(X/G)^{\rm{an}}}$-module by Corollary~\ref{cor:huber-coherent-sheaves}~(\ref{cor:huber-coherent-sheaves-2}) and it suffices to show \[
\O_{(X/G)^{\rm{an}}} \to \phi_*\left(\O_{X^{\rm{an}}/G}\right)
\]
is an isomorphism. Since {\it classical points} on a rigid-analytic space reflect isomorphisms of {\it coherent} sheaves\footnote{This is standard and can be deduced from \cite[Proposition 9.4.2/6 and Corollary 9.4.2/7]{BGR}}, it suffices to show that 
\[
\O_{(X/G)^{\rm{an}}, y} \to \left(\phi_*\left(\O_{X^{\rm{an}}/G}\right)\right)_{y}
\]
is an isomorphism for any classical point $y\in (X/G)^{\rm{an}}$. Now we use \cite[Theorem A.1.3]{C-ample}, finiteness and surjectivity of $\phi$ (that, in particular, implies that $\phi$ is surjective on classical points), and Step~$2$ to conclude that it suffices to show that the natural map
\begin{equation*}
\phi^\#_{\ov{x}}\colon \O_{(X/G)^{\rm{an}}, \phi(\ov{x})} \to \O_{X^{\rm{an}}/G, \ov{x}} 
\end{equation*}
is an isomorphism at each {\it classical} point of $X^{\rm{an}}/G$. We note that $\phi^\#_{\ov{x}}$ is a local homomorphism of noetherian local rings by \cite[Proposition 4.1/6]{B}. Thus, \cite[Chap III, \textsection 5.4, Prop. 4]{Bou} implies that it suffices to show that the morphism 
\[
\wdh{\phi}^\#_{\ov{x}}\colon \wdh{\O}_{(X/G)^{\rm{an}}, \phi(\ov{x})} \to \wdh{\O}_{X^{\rm{an}}/G, \ov{x}}
\]
is an isomorphism, where the completions on both sides are understood with respect to the corresponding maximal ideals. \smallskip

{\it Step 4. We show that $\wdh{\phi}^\#_{\ov{x}}$ is an isomorphism}: We note Corollary~\ref{cor:invariants-local} gives that 
\[
\wdh{\O}_{X^{\rm{an}}/G, \ov{x}} \cong \left( \prod_{x\in \pi^{-1}_{X^{\rm{an}}}(\ov{x})} \wdh{\O}_{X^{\rm{an}}, x} \right)^G.
\]
Now we consider the natural morphism of locally ringed spaces $\it{i}\colon (X/G)^{\rm{an}} \to X/G$. By \cite[Lemma 5.1.2(1), (2)]{C-irreducible}, $\it{i}$ is a bijection between the sets of classical points of $(X/G)^{\rm{an}}$ and closed points of $X/G$, and the natural morphism 
\[
    \wdh{\O}_{X/G, \it{i}(y)} \to \wdh{\O}_{(X/G)^{\rm{an}}, \it{y}}
\]
is an isomorphism for any closed point $y\in (X/G)^{\rm{an}}$. \smallskip

We set $\ov{z}\coloneqq \it{i}(\phi(\ov{x}))$. Using finiteness of $X \to X/G$ and Lemma~\ref{lemma:invariants-commute-flat-base-change}, we see\footnote{One can repeat the proof of Corollary~\ref{cor:invariants-local} using that Lemma~\ref{lemma:invariants-local} holds on the level of henselian local rings in the scheme world.} that
\[
\wdh{\O}_{(X/G)^{\rm{an}}, \phi(\ov{x})} \simeq \wdh{\O}_{X/G, \ov{z}} \simeq \left(\prod_{z\in \pi^{-1}_{X}(\ov{z})} \wdh{\O}_{X, z} \right)^G,
\]
and $\phi^\#_{\ov{x}}$ identified with the natural map
\begin{equation}\label{eqn:alg-adic}
\left(\prod_{z\in \pi^{-1}_{X}(\ov{z})} \wdh{\O}_{X, z} \right)^G \to \left( \prod_{x\in \pi^{-1}_{X^{\rm{an}}}(\ov{x})} \wdh{\O}_{X^{\rm{an}}, x} \right)^G.
\end{equation}
Finally, we use \cite[Lemma 5.1.2(2)]{C-irreducible} once again to conclude that (\ref{eqn:alg-adic}) is an isomorphism. 
\end{proof}


\section{Properties of the Geometric Quotients}

We discuss some properties of schemes (resp. formal schemes, resp. adic spaces) that are preserved by taking geometric quotients. For instance, one would like to know that $X/G$ is separated (resp. quasi-separated, resp. proper) if $X$ is. This is not entirely obvious as $X/G$ is explicitly constructed only in the affine case, and in general one needs to do some gluing to get $X/G$. This gluing might, a priori, destroy certain global properties of $X$ such as separatedness. In this section we show that this does not happen for many geometric properties in all schematic, formal and adic setups. We mostly stick to the properties we will need in our paper \cite{Z1}.

\subsection{Properties of the Schematic Quotients}\label{section:prop-alg}

In this section, we discuss the schematic case. For the rest of the section, we fix a valuation ring $k^+$. The proofs are written so they will adapt to other settings (formal schemes and adic spaces). \smallskip

Let $f\colon X \to Y$ be a $G$-invariant\footnote{Here and later in the paper, a $G$-invariant morphism means a $G$-equivariant morphism for some $G$-action on the source and a trivial $G$-action on the target.} morphism of flat, locally finite type $k^+$-schemes. We assume that any orbit $G.x \subset X$ lies inside some open affine subscheme $V_x\subset X$. In particular, the conditions of Theorem~\ref{thm:alg-main} are satisfied, so it gives that $X/G$ is a $k^+$-scheme with a finite morphism $\pi \colon X \to X/G$. The universal property of the geometric quotient implies that $f$ factors through $\pi$ and defines the commutative diagram:
\[
\begin{tikzcd}
X \arrow{d}{\pi}  \arrow{rd}{f} &\\
X/G \arrow{r}{f'} & Y.
\end{tikzcd}
\]

\begin{prop}\label{prop:quot-prop-alg} Let $f\colon X\to Y$, a finite group $G$, and $f':X/G \to Y$ be as above. Then $f'$ is quasi-compact (resp. quasi-separated, resp. separated, resp. proper, resp. finite) if $f$ is so.
\end{prop}
\begin{proof}
We note that all these properties are local on $Y$. Since the formation of $X/G$ commutes with open immersions, we can assume that $Y$ is affine. \smallskip

{\it Quasi-compactness}: We suppose that $f$ is quasi-compact. Using the fact that $Y$ is affine, we see that quasi-compactness of $f$ (resp. $f'$) is equivalent to quasi-compactness of the scheme $X$ (resp. $X/G$). Thus, $X$ is quasi-compact. Since $\pi$ is surjective by Theorem~\ref{thm:alg-main}, we conclude that $X/G$ is quasi-compact. Therefore, $f'$ is quasi-compact as well.
\smallskip

{\it Quasi-separatedness}: We suppose that the diagonal morphism $\Delta_{X}:X \to X\times_Y X$ is quasi-compact and consider the following commutative square:
\begin{equation}\label{eqn:diagram-diagonal}
\begin{tikzcd}
X \arrow{r}{\Delta_X} \arrow{d}{\pi} & X\times_Y X \arrow{d}{\pi\times_Y \pi}\\
X/G \arrow{r}{\Delta_{X/G}} & X/G \times_Y X/G
\end{tikzcd}
\end{equation}

We know that $\pi$ is finite, so it is quasi-compact. Therefore, the morphism $\pi\times_Y \pi$ is quasi-compact as well. This implies that the morphism
\[
(\pi\times_Y \pi) \circ \Delta_X = \Delta_{X/G} \circ \pi
\]
is quasi-compact. But we also know that $\pi$ is surjective, so we see that quasi-compactness of $\Delta_{X/G} \circ \pi$ implies quasi-compactness of $\Delta_{X/G}$. Thus $f'$ is quasi-separated. \smallskip

{\it Separatedness:} We consider Diagram~(\ref{eqn:diagram-diagonal}) once again. Since $\pi$ is finite we conclude that $\pi\times_Y \pi$ is finite as well, hence closed. Now we use surjectivity of $\pi$ to get an equality: 
\[
\Delta_{X/G}(X/G)=(\pi\times_Y \pi)(\Delta_{X}(X))
\]
with $(\pi\times_Y \pi)(\Delta_{X}(X))$ being closed as image of the closed subset $\Delta_{X}(X)$ (by separateness of $X$ over affine $Y$). This shows that $\Delta_{X/G}(X/G)$ is a closed subset of $X/G\times_{Y}X/G$, so $X/G$ is separated. \smallskip

{\it Properness:} We already know that properness of $f$ implies that $f'$ is quasi-compact and separated. Also, Theorem~\ref{thm:valuation-main} shows that $f'$ is locally of finite type, so it is of finite type. The only thing that we are left to show is that it is universally closed. But this easily follows from universal closedness of $f$ and surjectivity of $\pi$. 
\smallskip

{\it Finiteness:} A finite morphism is proper, so the case of proper morphisms implies that $f'$ is proper. It is also quasi-finite as $\pi$ is surjective and $f=f' \circ \pi$ has finite fibers. Now Zariski's Main Theorem \cite[Corollaire 18.12.4]{EGA4_4} implies that $f'$ is finite.
\end{proof}

We now slightly generalize Proposition~\ref{prop:quot-prop-alg} to the case of a $G$-equivariant morphism $f$. Namely, we consider a $G$-equivariant morphism $f\colon X \to Y$ of flat, locally finite type $k^+$-schemes. We assume that the actions of $G$ on both $X$ and $Y$ satisfy the condition of Theorem~\ref{thm:valuation-main}. Then the universal property of the geometric quotient implies that $f$ descends to a morphism $f'\colon X/G \to Y/G$ over $k^+$. We show that various properties of $f$ descend to $f'$:

\begin{prop}\label{prop:quot-prop-alg-2} Let $f\colon X\to Y$, a finite group $G$, and $f'\colon X/G \to Y/G$ be as above. Then $f'$ is quasi-compact (resp. quasi-separated, resp. separated, resp. proper, resp. a $k$-modification\footnote{A morphism $f\colon X \to Y$ of flat, locally  finite type $k^+$-schemes is called {\it a $k$-modification}, if it is proper and the base change $f_k\colon X_k \to Y_k$ is an isomorphism.}, resp. finite) if $f$ is so.
\end{prop}
\begin{proof}
We start the proof by considering the commutative diagram
\[
\begin{tikzcd}
X \arrow{r}{f} \arrow{d}{\pi_X}& Y \arrow{d}{\pi_Y}\\
X/G \arrow{r}{f'} & Y/G.
\end{tikzcd}
\]
We denote by $h\colon X \to Y/G$ the composition $ f'\circ \pi_X = \pi_Y\circ f$. Note that, for all relevant properties $\mathbf{P}$ except for $k$-modification, $f$ satisfies $\mathbf{P}$ implies that $h$ satisfies $\mathbf{P}$ due to finiteness of $\pi_Y$. Thus, all but the $k$-modification property follow from Proposition~\ref{prop:quot-prop-alg} applied to $h$.
\smallskip

Now suppose that $f$ is a $k$-modification. We have already proven that $f'$ is a proper map, so we only need to show that its restriction to $k$-fibers is an isomorphism. This follows from the fact that the formation of the geometric quotient commutes with flat base change such as $\Spec k \to \Spec k^+$.
\end{proof}

\begin{lemma}\label{lemma:properties-alg-G-torsor} Let $Y$ be a flat, locally finite type $k^+$-scheme, and $f\colon X \to Y$ a $G$-torsor for a finite group $G$. The natural morphism $f'\colon X/G \to Y$ is an isomorphism.
\end{lemma}
\begin{proof}
Since a $G$-torsor is a finite \'etale morphism, we see that $X$ is a flat, locally finite type $k^+$-scheme. Moreover, we note that the conditions of Theorem~\ref{thm:valuation-main} are satisfied as $f$ is affine and the action on $Y$ is trivial. Thus, $X/G$ is a flat, locally finite type $k^+$-scheme. The universal property of the geometric quotient defines the map
\[
X/G \to Y
\]
that we need to show to be an isomorphism. It suffices to check this \'etale locally on $Y$ as the formation of $X/G$ commutes with flat base change by Theorem~\ref{thm:alg-main}(\ref{thm:alg-main-4}). Therefore, it suffices to show that it is an isomorphism after the base change along $X \to Y$. Now, $X\times_Y X$ is a trivial $G$-torsor over $X$, so it suffices to show the claim for a trivial $G$-torsor. This is essentially obvious and follows either from the construction or from the universal property. 
\end{proof}

\subsection{Properties of the Formal Quotients}\label{section:prop-formal} 
Similarly to Section~\ref{section:prop-alg}, we discuss that certain properties descend to the geometric quotient in the formal setup. Most proofs are similar to those in Section~\ref{section:prop-alg}. \smallskip

For the rest of the section, we fix a complete, microbial valuation ring $k^+$ with a pseudo-uniformizer $\varpi$ and field of fractions $k$. \smallskip

We consider a $G$-equivariant morphism $f\colon \X \to \Y$ of admissible formal $k^+$-schemes. We assume that the actions of $G$ on both $\X$ and $\Y$ satisfy the condition of Theorem~\ref{thm:formal-main}. Then the universal property of the geometric quotient implies that $f$ descends to a morphism $f'\colon \X/G \to \Y/G$ over $k^+$:
\[
\begin{tikzcd}
\X \arrow{d}{\pi_\X}  \arrow{r}{f} & \Y \arrow{d}{\pi_\Y}\\
\X/G \arrow{r}{f'} & \Y/G.
\end{tikzcd}
\]

We show that various properties of $f$ descend to $f'$:

\begin{prop}\label{prop:quot-prop-formal} Let $f\colon \X\to \Y$, a finite group $G$, and $f'\colon \X/G \to \Y/G$ be as above. Then $f'$ is quasi-compact (resp. quasi-separated, resp. separated, resp. proper, resp. a rig-isomorphism\footnote{A morphism $f\colon \X \to \Y$ of admissible formal $k^+$-schemes is called {\it a rig-isomorphism} if the adic generic fiber $f_k\colon \X_k \to \Y_k$ is an isomorphism.}, resp. finite) if $f$ is so.
\end{prop}
\begin{proof}

We note that in the case of a quasi-compact (resp. quasi-separated, resp. separated, resp. proper) $f$, the proof of Proposition~\ref{prop:quot-prop-alg-2} works verbatim. We only need to use Theorem~\ref{thm:formal-main} in place of Theorem~\ref{thm:valuation-main}.\smallskip

The rig-isomorphism case is easy, akin to the $k$-modification case in Proposition~\ref{prop:quot-prop-alg-2}. We only need to use Theorem~\ref{thm:comparison-formal-adic} in place of Theorem~\ref{thm:alg-main}(\ref{thm:alg-main-4}). \smallskip

Now suppose that $f$ is finite. The proper case implies that $f'$ is proper, and it is clearly quasi-finite. Therefore, the mod-$\varpi$ fiber $f'_0\colon (\X/G)_0 \to (\Y/G)_0$ is finite. Now \cite[Proposition I.4.2.3]{FujKato} gives that $f'$ is finite. 
\end{proof}

\begin{lemma}\label{lemma:properties-formal-G-torsor} Let $\Y$ be an admissible formal $k^+$-scheme, and $f:\X \to \Y$ a $G$-torsor for a finite group $G$. The natural morphism $f'\colon \X/G \to \Y$ is an isomorphism.
\end{lemma}
\begin{proof}
The proof of Lemma~\ref{lemma:properties-alg-G-torsor} adapts to this situation. The only non-trivial fact that we used is that one can check that a morphism is an isomorphism after a finite \'etale base change (and we use Theorem~\ref{thm:formal-main}(\ref{thm:formal-main-4}) in place of Theorem~\ref{thm:alg-main}(\ref{thm:alg-main-4})). This follows from descent for adic, faithfully flat morphisms \cite[Proposition I.6.1.5]{FujKato}\footnote{The case of a finite flat morphism is much easier as the completed tensor product along a finite module coincides with the usual tensor product due to Lemma~\ref{lemma:formal-general}(\ref{lemma:formal-general-1}).} 
\end{proof}

\subsection{Properties of the Adic Quotients}\label{section:prop-adic} 
Similarly to Section~\ref{section:prop-alg} and Section~\ref{section:prop-formal}, we discuss that certain properties descend to the geometric quotient in the adic setup. \smallskip

For the rest of the section, we fix a locally strongly noetherian analytic adic space $S$. \smallskip

We consider a $G$-equivariant $S$-morphism $f:X \to Y$ of locally topologically finite type adic $S$-spaces. We assume that the actions of $G$ on both $X$ and $Y$ satisfy the condition of Theorem~\ref{thm:adic-main}. Then the universal property of the geometric quotient implies that $f$ descends to a morphism $f'\colon X/G \to Y/G$ over $S$:
\[
\begin{tikzcd}
X \arrow{d}{\pi_X}  \arrow{r}{f} & Y \arrow{d}{\pi_Y}\\
X/G \arrow{r}{f'} & Y/G.
\end{tikzcd}
\]

We show that various properties of $f$ descend to $f'$:

\begin{prop}\label{prop:quot-prop-adic} Let $f\colon X\to Y$, a finite group $G$, and $f'\colon X/G \to Y/G$ be as above. Then $f'$ is quasi-compact (resp. quasi-separated, resp. separated, resp. proper, resp. finite) if $f$ is so.
\end{prop}
\begin{proof}

The proof is almost identical to that of Proposition~\ref{prop:quot-prop-alg-2}. We use Theorem~\ref{thm:adic-main} in place of Theorem~\ref{thm:valuation-main} (that is used in the proof of Proposition~\ref{prop:quot-prop-alg} that Proposition~\ref{prop:quot-prop-alg-2} relies on). We use \cite[Proposition 1.5.5]{H3} in place of \cite[Corollaire 18.12.4]{EGA4_4} to ensure that a quasi-finite, proper morphism is finite. 
\end{proof}

\begin{lemma}\label{lemma:properties-adic-G-torsor} Let $Y$ be a locally topologically finite type adic $S$-space, and $f:X \to Y$ a $G$-torsor for a finite group $G$. The natural morphism $f'\colon X/G \to Y$ is an isomorphism.
\end{lemma}
\begin{proof}
The proof of Lemma~\ref{lemma:properties-alg-G-torsor} adapts to this situation. The only non-trivial fact that we used is that one can check that $X/G \to Y$ is an isomorphism after a surjective, flat base change (and we use Theorem~\ref{thm:adic-main}(\ref{thm:adic-main-4}) in place of Theorem~\ref{thm:alg-main}(\ref{thm:alg-main-4})). This follows from Lemma~\ref{lemma:huber-flat-descent} as $f'$ is finite due to Proposition~\ref{prop:quot-prop-adic}.
\end{proof}

\newpage

\newpage

\appendix

\section*{Appendix}

\section{Adhesive Rings and Boundedness of Torsion Modules}\label{appendix-A}

Let $A$ be a ring with an ideal $I$. We define the notion of $I$-torsion part of an $A$-module and discuss some of its basic properties. Then we define the notion of (universally) adhesive and topologically (universally) adhesive rings. The main original results in this Appendix are Theorem~\ref{thm:universally-adhesive-main} and Theorem~\ref{thm:adhesive-formal-main}. The rest is mostly a summary of the results from \cite{FujKato} in a form convenient for the reader. 

\subsection{$I$-torsion Submodule}

\begin{defn}\label{defn:I-tors} Let $M$ be an $A$-module, $a\in A$, and $I\subset A$ an ideal. \smallskip

An element $m\in M$ is {\it $a$-torsion} if $a^nm=0$ for some $n\geq 1$. The set of all $a$-torsion elements is denoted by $M_{a-tors}$. \smallskip 

An element $m\in M$ is {\it $I$-torsion} if $m$ is $a$-torsion for every $a\in I$. The set of all $I$-torsion elements is denoted by $M_{I-tors}$. \smallskip

We say that $M$ is {\it I-torsion free} if $M_{I-tors}=0$. \smallskip

An $A$-submodule $N\subset M$ is {\it saturated} if $M/N$ is $I$-torsion free. 
\end{defn}

\begin{rmk}\label{rmk:adhesive-powers} Suppose that $I, J\subset A$ are finitely generated ideals such that $I^n \subset J$ and $J^m \subset I$ for some integers $n$ and $m$. Then $M_{I-tors}=M_{J-tors}$ for any $A$-module $M$.
\end{rmk}

\begin{lemma}\label{lemma:adhesive-torsion-flat-local} Let $A \to B$ be a flat ring homomorphism, and $I\subset A$ a finitely generated ideal, and $M$ an $A$-module. Then $M_{I-tors}\otimes_A B \simeq (M\otimes_A B)_{IB-tors}$.
\end{lemma}
\begin{proof}
We start by choosing some generators $I=(a_1, \dots, a_n)$. Then 
\[
M_{I-tors}=\bigcap M_{a_i-tors}\text{, and } (M\otimes_A B)_{IB-tors}=\bigcap (M\otimes_A B)_{a_i-tors}.
\] 
Therefore, it suffices to show that $M_{a-tors}$ commutes with flat base change. Now we note that $M_{a-tors} = \cup_n M[a^n]$ where $M[a^n]$ is the submodule of elements annihilated by $a^n$. It is clear that \[(M\otimes_A B)[a^n]=M[a^n]\otimes_A B\] 
since $B$ is $A$-flat. This implies that $M_{a-tors}=(M\otimes_A B)_{a-tors}$. 
\end{proof}


Lemma~\ref{lemma:adhesive-torsion-flat-local} implies that the notion of $M_{I, tors}$ can be globalized.

\begin{defn}\label{defn:adhesive-torsion-global} Let $X$ be scheme, $\cal{I}$ a quasi-coherent ideal of finite type, and $\cal{M}$ a quasi-coherent $\O_X$-module. The {\it $\O_X$-submodule of $\cal{I}$-torsion elements} $\cal{M}_{\cal{I}-tors}$ is defined as the sheafification of 
\[
U \mapsto \cal{M}(U)_{\cal{I}(U)-tors}. 
\]
\end{defn}

\begin{rmk}\label{rmk:adhesive-torsion-affine} Lemma~\ref{lemma:adhesive-torsion-flat-local} implies that $\cal{M}_{\cal{I}-tors}$ is a quasi-coherent $\O_X$-module. If $X=\Spec A$, $\cal{I}=\widetilde{I}$, and $\cal{M}=\widetilde{M}$. Then $\cal{M}_{\cal{I}-tors}\simeq \widetilde{M_{I-tors}}$.
\end{rmk}

\begin{defn}\label{defn:adhesive-torsionfree-global}
Let $X$ be a scheme, and $\cal{I}$ a quasi-coherent ideal of finite type. We say that $X$ is {\it $\cal{I}$-torsion free} if $\O_{X, \cal{I}-tors}\simeq 0$. \smallskip

Let $f\colon X \to Y$ be a morphism of schemes, and $\cal{I} \subset \O_Y$ a quasi-coherent ideal of finite type. We say that $X$ is {\it $\cal I$-torsion free} if $X$ is $\cal{I}\O_X$-torsion free. 
\end{defn}

\subsection{Universally Adhesive Schemes}

\begin{defn}\label{defn:adhesive} A pair $(R, I)$ of a ring and a finitely generated ideal is {\it adhesive} (or $R$ is {\it $I$-adically adhesive}) if $\Spec R \setminus \mathrm{V}(I)$ is noetherian and, for any finite $R$-module $M$, the $I$-torsion submodule $M_{I-tors}$ (see Definition~\ref{defn:I-tors}) is $R$-finite (see \cite[Definition 0.8.5.4]{FujKato}). \medskip 

A pair $(R, I)$ is {\it universally adhesive} if $(R[X_1, \dots, X_d], IR[X_1, \dots, X_d])$ is an adhesive pair for all $d\geq 0$.
\end{defn}

\begin{rmk}\label{rmk:microbial-valuation-adhesive} A valuation ring $k^+$ is universally adhesive if it is microbial (in the sense of Definition~\ref{defn:formal-microbial-valuation}). More precisely, $k^+$ is universally $\varpi$-adically adhesive for any choice of a pseudo-uniformizer $\varpi \in k^+$. Indeed, \cite[Proposition 0.8.5.3]{FujKato} implies that it is sufficient to see that any finite $k^+[X_1, \dots, X_d]$-module $M$ that is $\varpi$-torsion free (see Definition~\ref{defn:I-tors}) is finitely presented. This follows from Lemma~\ref{lemma:valuations-general}(\ref{lemma:valuations-general-3}) and the observation that $M$ is torsion free over $k^+$ if and only if it is $\varpi$-torsion free.
\end{rmk}

\begin{lemma}\label{lemma:adhesive=microbial} A valuation ring $k^+$ is $I$-adically adhesive for some finitely generated ideal $I$ if and only if $k^+$ is microbial.
\end{lemma}
\begin{proof}
If $k^+$ is microbial, we take $I=(\varpi)$ for any pseudo-uniformizer $\varpi$. Then $k^+$ is $I$-adically adhesive by Remark~\ref{rmk:microbial-valuation-adhesive}. \smallskip

Now we suppose that $k^+$ is adhesive for some finitely generated ideal $I$. Then $I=(a)$ is principal because $k^+$ is a valuation ring. Hence, $k^+\left[\frac{1}{a}\right]$ is a noetherian valuation ring by the $I$-adic adhesiveness of $k^+$. Therefore, $k^+\left[\frac{1}{a}\right]$ is either a field or discrete valuation ring. \smallskip

We firstly consider the case $k^+\left[\frac{1}{a}\right]$ is a field. We then observe that $\rm{rad}(a)$ is a height-$1$ prime ideal of $k^+$ by \cite[Proposition 0.6.7.2 and Proposition 0.6.7.3]{FujKato}. Therefore, $k^+$ is microbial by \cite[Definition 1.1.4]{H3} or \cite[Proposition 9.1.3]{Seminar}. \smallskip

Now we consider the case $k^+\left[\frac{1}{a}\right]$ a discrete valuation ring. Its maximal ideal $\m$ is clearly of height-$1$, so it defines a height-$1$ prime ideal $\frak{p}$ of $k^+$. Hence, {\it loc. cit.} implies that $k^+$ is microbial. 
\end{proof}

Here we summarize the main results about universally adhesive pairs:

\begin{lemma}\label{lemma:adhesive-general} Let $(R, I)$ be a universally adhesive pair, $A$ a finite type $R$-algebra, and $M$ a finite $A$-module.
\begin{enumerate}
    \item\label{lemma:adhesive-general-4} Let $N\subset M$ be a saturated $A$-submodule of $M$. Then $N$ is a finite $A$-module.
    \item\label{lemma:adhesive-general-3} If $M$ is $I$-torsion free as an $A$-module, then it is a finitely presented $A$-module. 
    \item\label{lemma:adhesive-general-2} If $A$ is $I$-torsion free as an $R$-module, then it is a finitely presented $R$-algebra.
\end{enumerate}
\end{lemma}
\begin{proof}
We choose some surjective morphism $\varphi \colon R[X_1, \dots, X_d] \to A$. Then the definition of universal adhesiveness says that $R[X_1, \dots, X_d]$ is $I$-adically adhesive. This easily implies that so is $A$. Now the first two claims follow from \cite[Proposition 0.8.5.3]{FujKato}. To show the last claim, we note that the kernel $\varphi$ is a saturated submodule of $R[X_1, \dots, X_d]$, so it is a finitely generated ideal by Part~($1$). Therefore, $A$ is finitely presented as an $R$-algebra. 
\end{proof}

\begin{lemma}\label{lemma:universally-adhesive-Artin-Tate} Let $(R, I)$ be a universally adhesive pair (see Definition~\ref{defn:adhesive}), and $A\to B$ be a finite injective morphism of $R$-algebras. Suppose that $B$ is of finite type over $R$, and that $A\subset B$ is saturated (see Definition~\ref{defn:I-tors}). Then $A$ is a finite type $R$-algebra.
\end{lemma}
\begin{proof}
The proof is analogous to that of Lemma~\ref{lemma:valuation-Artin-Tate} and again the only difficulty lies in showing that $A$ is finite over $A'$ (as defined in the proof of Lemma~\ref{lemma:valuation-Artin-Tate}). This follows from Lemma~\ref{lemma:adhesive-general}(\ref{lemma:adhesive-general-4}). 
\end{proof}

\begin{cor}\label{cor:universally-adhesive-invariants-finitely-generated} Let $(R, I)$ be a universally adhesive pair, and $A$ an $I$-torsion free, finite type $R$-algebra with an $R$-action of a finite group $G$. The $R$-flat $A^G$ is a finite type $R$-algebra, and the natural morphism $A^G \to A$ is finitely presented.
\end{cor}
\begin{proof}
The proof of Corollary~\ref{cor:valuation-invariants-finitely-generated} works verbatim. One only has to use Lemma~\ref{lemma:universally-adhesive-Artin-Tate} in place of Lemma~\ref{lemma:valuation-Artin-Tate}. 
\end{proof}

\begin{defn}\label{defn:adhesive-global} A pair $(X, \cal{I})$ of a scheme and a quasi-coherent ideal of finite type is {\it universally adhesive} if there is an open affine covering of $X=\cup_{i\in I}\Spec U_i$ such that $(\O(U_i), \cal{I}(U_i))$ is universally adhesive for all $i\in I$.
\end{defn}

\begin{rmk} The notion of universal adhesiveness is independent of a choice of affine open covering. This is explained in \cite[Proposition 0.8.5.6 and Proposition 0.8.6.7]{FujKato}. It essentially follows from Lemma~\ref{lemma:adhesive-torsion-flat-local} and the fact that noetherianness is local in the fppf topology. 
\end{rmk}

\begin{thm}\label{thm:universally-adhesive-main} Let $(S, \cal{I})$ be a universally adhesive pair (in the sense of Definition~\ref{defn:adhesive-global}), and $X$ be an $\cal{I}$-torsion free, locally finite type $S$-scheme with an $S$-action of a finite group $G$. Suppose that each point $x\in X$ admits an affine neighborhood $V_x$ containing $G.x$. Then the scheme $X/G$ as in Theorem~\ref{thm:alg-main} is $\cal{I}$-torsion free and locally finite type over $S$, and the integral surjection $\pi\colon X \to X/G$ is finite and finitely presented.
\end{thm} 
\begin{proof}
The proof of Theorem~\ref{thm:valuation-main} just goes through if one uses Corollary~\ref{cor:universally-adhesive-invariants-finitely-generated} instead of Corollary~\ref{cor:valuation-invariants-finitely-generated}. 
\end{proof}

\subsection{Universally Adhesive Formal Schemes}

\begin{defn}\label{defn:tu-adhesive} A pair $(R, I)$ of a ring and a finitely generated ideal is {\it topologically universally adhesive} (or $R$ is {\it $I$-adically topologically universally adhesive}) if $(R, I)$ is universally adhesive, and the pair $(\wdh{R}\langle X_1,\dots, X_d\rangle, I\wdh{R}\langle X_1,\dots, X_d\rangle)$ is adhesive (in the sense of Definition~\ref{defn:adhesive}) for any $d\geq 0$. \medskip

An adically topologized ring $R$ endowed with the adic topology defined by a finitely generated ideal of definition $I \subset R$ is {\it topologically universally adhesive} if $R$ is $I$-adically complete, and the pair $(R, I)$ is topologically universally adhesive. 
\end{defn}

\begin{rmk} We note that the definition of topologically universally adhesive topological rings is independent of the choice of a finitely generated ideal of definition. For any two ideals of definition $I$ and $J$, $I^n\subset J$ and $J^m \subset I$ for some integers $n$ and $m$. Therefore, $M_{I-tors}=M_{J-tors}$ by Remark~\ref{rmk:adhesive-powers}. 
\end{rmk}

\begin{rmk}\label{rmk:adhesive-microb-tu-adhesive} We note that a microbial valuation ring $k^+$ is topologically universally adhesive. More precisely, $k^+$ is topologically universally $\varpi$-adically adhesive for any choice of a pseudo-uniformizer $\varpi \in k^+$. This is proven in \cite[Theorem 0.9.2.1]{FujKato}. Alternatively, one can easily show the claim from Lemma~\ref{lemma:formal-general} and the classical fact that $k$ is strongly noetherian, i.e. $k\langle X_1,\dots, X_d\rangle$ is noetherian for any $d\geq 0$.
\end{rmk}

\begin{lemma}\label{lemma:adhesive-formal-general} Let $R$ be a complete, topologically universally $I$-adically adhesive ring, $A$ be a topologically finite type $R$-algebra, and $M$ a finite $A$-module. Then
\begin{enumerate}
    \item\label{lemma:adhesive-formal-general-1}  $M$ is $I$-adically complete. In particular, $A$ is $I$-adically complete.
    \item\label{lemma:adhesive-formal-general-new-2}  Let $N \subset M$ be an $A$-submodule of $M$. Then the $I$-adic topology on $M$ restricts to the $I$-adic topology on $N$. 
    \item\label{lemma:adhesive-formal-general-3}  Let $N \subset M$ be a saturated $A$-submodule of $M$. Then $N$ is a finite $A$-module. 
    \item\label{lemma:adhesive-formal-general-new}  If $M$ is $R$-flat, it is finitely presented over $A$. 
    \item\label{lemma:adhesive-formal-general-2}  If $A$ is $R$-flat, it is topologically finitely presented. 
    \item\label{lemma:adhesive-formal-general-4}  For any element $f\in A$, the completed localization $A_{\{f\}}=\lim_n A_f/I^nA_f$ is $A$-flat. 
\end{enumerate}
\end{lemma}
\begin{proof}
We choose a surjection $R\langle T_1, \dots, T_n\rangle \to A$. Then it suffices to show the first claim for $A=R\langle T_1, \dots, T_n\rangle$. We note that $A$ is $I$-adically adhesive and $I$-adically complete, and so \cite[Proposition 0.8.5.16]{FujKato} (and the discussion after it) implies that any finite $A$-module is $I$-adically complete. The second claim follows from \cite[Proposition 0.8.5.16]{FujKato} and the definition of $\mathbf{(AP)}$ (see \cite[\textsection 7.4(c), p.161]{FujKato}). The proofs of Parts~(3)-(5) are similar to the proof of the analogous statements in Lemma~\ref{lemma:adhesive-general}. The last Part is proven in \cite[Proposition I.2.1.2]{FujKato}.
\end{proof}

\begin{defn}\label{defn:topologically-universally-adhesive-admissible-algebra} An algebra $A$ over a complete, topologically universally $I$-adically adhesive ring $R$ is {\it admissible} if $A$ is topologically finite type over $R$ and $I$-torsion free.
\end{defn}

\begin{lemma}\label{lemma:topologically-universally-adhesive-Artin-Tate} Let $R$ be an $I$-adically complete, $I$-adically topologically universally adhesive ring (see Definition~\ref{defn:tu-adhesive}), and $A\to B$ be a finite injective morphism of $I$-adically complete $R$-algebras. Suppose that $B$ is topologically finite type over $R$, and $A\subset B$ is saturated in $B$ (See Definition~\ref{defn:I-tors}). Then $A$ is a topologically finite type $R$-algebra.
\end{lemma}
\begin{proof}
The proof is analogous to that of Lemma~\ref{lemma:formal-Artin-Tate} and again the only difficulty lies in showing that $A$ is finite over $A'$ (as defined in the proof of Lemma~\ref{lemma:formal-Artin-Tate}). This follows from Lemma~\ref{lemma:adhesive-formal-general}(\ref{lemma:adhesive-formal-general-3}). 
\end{proof}

\begin{cor}\label{cor:topologically-universally-adhesive-invariants-top-finitely-generated} Let $R$ be an $I$-adically complete, $I$-adically topologically universally adhesive ring, and $A$ an admissible $R$-algebra (in the sense of Definition~\ref{defn:topologically-universally-adhesive-admissible-algebra}) with an $R$-action of a finite group $G$. Then $A^G$ is an admissible $R$-algebra, the induced topology on $A^G$ coincides with the $I$-adic topology, and $A$ is a finitely presented $A^G$-module.
\end{cor}
\begin{proof}
The proof of Corollary~\ref{cor:formal-invariants-top-finitely-generated} works verbatim. One only needs to use Lemma~\ref{lemma:topologically-universally-adhesive-Artin-Tate} in place of Lemma~\ref{lemma:formal-Artin-Tate} and Lemma~\ref{lemma:adhesive-formal-general} in place of Lemma~\ref{lemma:formal-general}.
\end{proof}

\begin{prop}\label{prop:topologically-universally-adhesive-geometric-quotient-affine} Let $R$ be an $I$-adically complete, $I$-adically topologically universally adhesive ring, and $\X=\Spf A$ an affine admissible formal $R$-scheme with an $R$-action of a finite group $G$. Then the natural map $\phi\colon \X/G \to \Y=\Spf A^G$ is an $R$-isomorphism of topologically locally ringed spaces. In particular, $\X/G$ is an admissible formal $R$-scheme.
\end{prop}
\begin{proof}
The proof of Proposition~\ref{prop:formal-geometric-quotient-affine} goes through in this more general set-up. The only two differences are that one needs to deduce that $A^G$ is admissible over $R$ (with the induced topology equal to the $I$-adic) from Corollary~\ref{cor:topologically-universally-adhesive-invariants-top-finitely-generated} instead of Corollary~\ref{cor:formal-invariants-top-finitely-generated} and one needs to use Lemma~\ref{lemma:adhesive-formal-general}(\ref{lemma:adhesive-formal-general-4}) instead of Lemma~\ref{lemma:formal-general}(\ref{lemma:adhesive-formal-general-4}) to ensure that $(A^G)_{\{f\}}$ is an $A^G$-flat module. 
\end{proof}

\begin{defn}\label{defn:universally-adhesive-formal-schemes} A formal scheme $\X$ is {\it locally universally adhesive} if there exists an affine open covering $\X =\cup_{i\in I} \sU_i$ such that each $\sU_i$ is isomorphic to $\Spf A$ with $A$ a topologically universally adhesive ring. If $\X$ is, moreover, quasi-compact, we say that $\X$ is universally adhesive.  
\end{defn}

\begin{rmk} Definition~\ref{defn:universally-adhesive-formal-schemes} is independent of the choice of open covering. More precisely, an affine formal scheme $\X=\Spf A$ is universally adhesive if and only if $A$ is topologically universally adhesive. This is shown in \cite[Propostition I.2.1.9]{FujKato}.
\end{rmk}

\begin{rmk}\label{rmk:adhesive-flat-issue} Lemma~\ref{lemma:adhesive-formal-general}(\ref{lemma:adhesive-formal-general-4}) can be strengthened to the statement that an adic morphism of affine universally adhesive formal schemes $\Spf B \to \Spf A$ is flat\footnote{We follow \cite{FujKato} and say that an adic morphism of formal schemes $f\colon \X \to \Y$ is flat if and only if $\O_{\Y, f(x)} \to \O_{\X, x}$ is flat for all $x\in \X$.} if and only if $A \to B$ is flat. This is proven from \cite[Proposition I.4.8.1]{FujKato}.
\end{rmk}

\begin{defn}\label{defn:adhesive-admissible} Let $\S$ be a universally adhesive formal scheme. An adic $\S$-scheme $\X$ is called {\it admissible} if it is locally of topologically finite type, and there is an affine open covering $\X =\cup_{i\in I} \sU_i$ such that each $\sU_i$  is isomorphic to $\Spf A$  with $A$ an $I$-torsion free ring for a(ny) finitely generated ideal of definition $I\subset A$.
\end{defn}

We show that this definition is independent of the choice of a covering. 

\begin{lemma}\label{lemma:adhesive-admissible-local} Let $\X=\Spf A$ be an affine, locally of topologically finite type formal $\S$-scheme. Then $\X$ is admissible if and only if $A$ is $I$-torsion free for a(ny) finitely generated ideal of definition $I$.
\end{lemma}
\begin{proof}
First of all, we note that Remark~\ref{rmk:adhesive-powers} implies that Definition~\ref{defn:adhesive-admissible} is independent of the choice of a finitely generated ideal of definition $I$. Thus, using that $\X$ is quasi-compact, we can assume that $\X=\cup_{i=1}^n \sU_i=\Spf A_i$ with $A_i$ an $I$-torsion free $A$-algebra. Then the morphism $A \to \prod_{i=1}^n A_i$ is faithfully flat by Remark~\ref{rmk:adhesive-flat-issue} and the fact that all maximal ideals are open in an $I$-adically complete ring (see \cite[Lemma 0.7.2.13]{FujKato}). Now Lemma~\ref{lemma:adhesive-torsion-flat-local} implies that 
\[
\left(\prod_{i=1}^n A_i\right)_{I-tors} \simeq A_{I-tors} \otimes_A \left(\prod_{i=1}^n A_i\right). 
\]
Our assumption implies that $(\prod_{i=1}^n A_i)_{I-tors} \simeq 0$. Therefore, $A_{I-tors}\simeq 0$ as $A \to \prod_{i=1}^n A_i$ is  faithfully flat.
\end{proof}

\begin{lemma}\label{lemma:adhesive-admissible-commutes-flat-base-change} Let $\S$ be a universally adhesive formal scheme, and let $\X$ be an  $\S$-finite, admissible formal $\S$-scheme. Suppose that $\S' \to \S$ is a flat adic morphism of universally adhesive formal schemes. Then $\X' \coloneqq \X\times_{\S} \S'$ is an admissible formal $\S'$-scheme.
\end{lemma}
\begin{proof}
Lemma~\ref{lemma:adhesive-admissible-local} ensures that the question is Zariski local on $\X'$. Thus, we may and do assume that $\S$, $\S'$ (and, therefore, $\X$) are affine. Suppose $\S=\Spf A$, $\S'=\Spf A'$, and $\X=\Spf B$ for an $A$-algebra $B$ finite as an $A$-module (see \cite[Proposition I.4.2.1]{FujKato}). Choose an ideal of definition $I\subset A$, our assumptions imply that $IA'$ is an ideal of definition in $A'$. We know that  $\X'$ is given by $\Spf A'\wdh{\otimes}_{A} B$, so it is easily seen to be a topologically finite type formal $\cal{S}'$-scheme. We are only left to check that $A'\wdh{\otimes}_A B$ is $I$-torsionfree. \smallskip

We note that $A'\otimes_{A} B$ is finite over $A'$, so it is already $IA'$-adically complete by Lemma~\ref{lemma:adhesive-formal-general}(\ref{lemma:adhesive-formal-general-1}). Therefore, we conclude that $\X'\simeq \Spf \left(A'\otimes_{A} B\right)$. Now Remark~\ref{rmk:adhesive-flat-issue} ensures that $A'$ is $A$-flat, and so $A'\otimes_A B$ is $B$-flat. Since $B$ had no $I$-torsion, the same holds for $A'\otimes_A B$. 
\end{proof}

\begin{thm}\label{thm:adhesive-formal-main} Let $\S$ be a universally adhesive formal scheme (see Definition~\ref{defn:universally-adhesive-formal-schemes}), and $\X$ an admissible formal $\S$-scheme (see Definition~\ref{defn:adhesive-admissible}). Suppose that $\X$ has an $\S$-action of a finite group $G$ such that each point $x\in \X$ admits an affine neighborhood $\V_x$ containing $G.x$. Then $\X/G$ is an admissible formal $\S$-scheme. Moreover, it satisfies the following properties:
\begin{enumerate}
\item $\pi\colon \X \to \X/G$ is universal in the category of $G$-invariant morphisms to topologically locally ringed $\S$-spaces.
\item $\pi: \X \to \X/G$ is a finite, surjective, topologically finitely presented morphism (in particular, it is closed). 
\item Fibers of $\pi$ are exactly the $G$-orbits.
\item The formation of the geometric quotient commutes with flat base change, i.e. for any universally adhesive formal scheme $\frak{Z}$ and a flat {\it adic} morphism $\frak{Z}\to \X/G$, the geometric quotient $(\X \times_{\X/G} \frak{Z})/G$ is a formal schemes, and the natural morphism $(\X \times_{\X/G} \frak{Z})/G \to \frak{Z}$ is an isomorphism.
\end{enumerate}
\end{thm} 
\begin{proof}
The proofs of parts $(1)$, $(2)$, and $(3)$ are similar to those of Theorem~\ref{thm:formal-main}. The main difference is that one needs to use Proposition~\ref{prop:topologically-universally-adhesive-geometric-quotient-affine} in place of Proposition~\ref{prop:formal-geometric-quotient-affine},  Lemma~\ref{lemma:adhesive-formal-general} in place of Lemma~\ref{lemma:formal-general}, and \cite[Proposition I.2.2.3]{FujKato} in place of \cite[Proposition 7.3/10]{B}. \medskip

We explain part $(4)$ in a bit more detail. We first reduce to the case $\S=\Spf R$, $\X=\Spf A$ with $A$ a finite $R$-module, and $\S'=\Spf R'$. Then $R \to R'$ is flat by Remark~\ref{rmk:adhesive-flat-issue}. Then Lemma~\ref{lemma:adhesive-admissible-commutes-flat-base-change} implies that $\X'$ is $\S'$-admissible, and then one can repeat the proof of Theorem~\ref{thm:formal-main} using Lemma~\ref{lemma:adhesive-formal-general}(\ref{lemma:adhesive-formal-general-1}) in place of Lemma~\ref{lemma:formal-general}(\ref{lemma:formal-general-1}).
\end{proof}

\begin{thm}\label{thm:adhesive-comparison-formal-alg} Let $R$ be an topologically universally $I$-adically adhesive ring, and $X$ an $I$-torsion free, locally finite type $R$-scheme with a $R$-action of a finite group $G$. Suppose that any orbit $G.x \subset X$ lies in an affine open subset $V_x$. The same holds for its $I$-adic completion $\wdh{X}$ with the induced $\wdh{R}^+$-action of $G$, and the natural morphism
\[
\wdh{X}/G \to \wdh{X/G}
\]
is an isomorphism.
\end{thm}
\begin{proof}
The proof of Theorem~\ref{thm:comparison-formal-alg} goes through in this wider generality. The only new non-trivial input is flatness of $A^G \to \wdh{A^G}$. More generally, this flatness holds for any finite type $R$-algebra $B$. Namely, any such algebra is $I$-adically adhesive, so it  satisfies the so called $\mathbf{BT}$ property (see Definition~\cite[Section 0.8.2(a)]{FujKato}) by \cite[Proposition 0.8.5.16]{FujKato}. Therefore, \cite[8.2.18(i)]{FujKato} implies that $B \to \wdh{B}$ is flat.
\end{proof}

\section{Foundations of Adic Spaces}\label{defn-adic}

The theory of adic spaces still seems to lack a ``universal reference'' for proofs of all basic questions one might want to use. For example, all of \cite{H0}, \cite{H1}, \cite{H3} and \cite{KedLiu2} do not really discuss notions of flat and separated morphisms in detail. The two main goals of this Appendix are to provide the reader with the main definitions we use in the paper, and to give proofs of claims that we need in the paper and that seem difficult to find in the standard literature on the subject.

We stick to the case of analytic adic spaces\footnote{We recall that an adic space $X$ is required to be sheafy, i.e. the structure presheaf $\O_X$ must be a sheaf.} since this is the only case that we need in this paper.

\subsection{Basic Definitions}
We start this section by recalling the definition of the category of topologically locally $v$-ringed spaces $\cal{V}$. This category will play the same role as the category of locally ringed spaces plays for the category of schemes. Namely, $\cal{V}$ is going to be a sufficiently flexible category with a fully faithful embedding of the category of adic spaces into it. 

\begin{defn}\label{defn:valuative-spaces} A {\it category of topologically locally $v$-ringed spaces\footnote{Our definition is taken from \cite[Definition 13.1.1]{Seminar}. It is different from \cite[page 521]{H1} since the latter requires $\O_X$ to be a sheaf of complete topological rings.} $\mathcal V$}  is the category  objects of this category are triples $(X, \O_X, \{v_x\}_{x\in X})$ such that
\begin{enumerate}
    \item $X$ is a topological space,
    \item $\O_X$ is a sheaf of topological rings such that the stalk $\O_{X,x}$ is a local ring for all $x\in X$,
    \item $v_x$ is a valuation on the residue field $k(x)$ of $\O_{X,x}$.
\end{enumerate} 

Morphisms $f\colon X \to Y$ of objects in $\mathcal V$ are defined as maps $(f, f^\#)$ of topologically locally ringed spaces such that the induced maps of residue fields $k(f(x)) \to k(x)$ are compatible with valuations (equivalently, induces a local inclusion between the valuation rings). \end{defn}

\begin{rmk}\label{rmk:huber-forgetful-conservative} The category $\cal{V}$ comes with the forgetful functor $F\colon \cal{V} \to \mathbf{TLRS}$ to the category of topologically locally ringed spaces. It is clear that this functor is conservative. 
\end{rmk}

\begin{defn}\label{defn:adic-analytic-spaces} We define the {\it category $\mathbf{AS}$ of analytic adic spaces} as the full subcategory of $\mathcal V$ whose objects are triples $(X, \O_X, \{v_x\}_{x\in X})$ locally isomorphic to $\Spa(A, A^+)$ for a complete Tate-Huber pair $(A, A^+)$. We remind the reader that this requires the pair $(A,A^+)$ to be ``sheafy''. 
\end{defn}


\begin{rmk} Given any analytic adic space $(X, \O_X, \{v_x\}_{x\in X})$, Huber defined a sheaf $\O_{X}^+$ as follows
\[
\O_{X}^+(U)=\left\{ f\in \O_X(U) \ | \ v_x(f) \leq 1 \text{ for any } x\in U \right\}.
\]
We note that \cite[Proposition 1.6]{H1} implies that $\O_X^+(X)=A^+$ for any sheafy complete Tate-Huber pair $(A, A^+)$ and $X=\Spa(A, A^+)$.
\end{rmk}

\begin{rmk}\label{grp-act-adic} Note that \cite[Proposition 2.1(ii)]{H1} guarantees that we have a natural identification
\[
\Hom_{\mathbf{AS}} (X, \Spa(A, A^+))=\Hom_{\operatorname{cont}} ((A, A^+), (\O_X(X), \O_X^+(X)))\footnote{This is the set of all continuous ring homomorphisms $f\colon A \to \O_X(X)$ such that $f(A^+) \subset \O_X^+(X)$. We do not claim that $(\O_X(X), \O_X^+(X))$ is a (Tate-)Huber pair.}.
\]
for any analytic adic space $X$.
\end{rmk}

\begin{defn} A {\it Tate affinoid adic space} is an object of the category $\mathbf{AS}$ that is isomorphic to $\Spa(A, A^+)$ for a complete Tate-Huber pair $(A, A^+)$.
\end{defn}

In the following, if we consider a Tate affinoid adic space $X=\Spa(A, A^+)$, we implicitly assume that $(A, A^+)$ is a {\it complete} Tate-Huber pair.

\begin{rmk} In general, there are analytic affinoid adic spaces that are not isomorphic to $\Spa(A, A^+)$ for any complete Tate-Huber pair $(A,A^+)$. The analytic condition implies the existence of a pseudo-uniformizer only {\it locally} on $\Spa(A, A^+)$, but it does not necessarily exist globally. See \cite[Example 1.5.7]{KedAr} for an explicit example of an analytic affinoid adic space that is not a Tate affinoid.
\end{rmk}

\subsection{Finite and Topologically Finite Type Morphisms of Adic Spaces}
\begin{defn}\label{defn:huber-topologically-finite-type} We say that a morphism of complete Tate-Huber pairs $(A, A^+) \to (B, B^+)$ is {\it topologically of finite type}, if there is a surjective quotient map $f\colon A \langle T_1,\dots,T_n \rangle \to B$ such that $B^+$ is integral over $A^+\langle T_1, \dots, T_n\rangle$. 
\end{defn}

\begin{rmk} This definition coincides with the definition of topologically finite type morphism of Huber pairs from \cite[page 533, before Lemma 3.3]{H1}. This is stated in \cite[Lemma 3.3 (iii)]{H1} and it is proven in \cite[Proposition 15.3.3]{Seminar}. 
\end{rmk}

\begin{rmk}\label{rmk:open-mapping} It turns out that any continuous surjective morphism $f\colon C \to B$ of complete Tate rings is a quotient mapping. Moreover, it is actually an open map; this is the content of the Banach Open Mapping Theorem \cite[Lemma 2.4 (i)]{H1}. 
\end{rmk}

There are crucial properties of topologically finite type morphisms that makes them behave similarly to finite type morphisms: 

\begin{lemma}\label{lemma:tft-useful} Let $f\colon (A, A^+) \to (B, B^+)$ and $g\colon (B, B^+) \to (C, C^+)$ be continuous homomorphisms of complete Tate-Huber pairs. If $f$ and $g$ are topologically finite type morphisms then so is $g\circ f$, and if $g\circ f$ is topologically finite type then so is $g$.
\end{lemma}
\begin{proof}
This is \cite[Lemma 3.3 (iv)]{H1}.
\end{proof}

\begin{defn}\label{adic-tft} A morphism of analytic adic spaces $f\colon  X \to Y$ is called {\it locally of topologically finite type}, if there is an open covering of $Y$ by Tate affinoids $\{V_{i}\}_{i\in I}$ and an open covering of $X$ by Tate affinoids $\{U_i\}_{i\in I}$ such that $f(U_i) \subset V_i$, and $(\O_Y(V_i), \O_Y^+(V_i)) \to (\O_X(U_i), \O_X^+(U_i))$ is topologically of finite type (in the sense of Definition \ref{defn:huber-topologically-finite-type}). If a morphism $f$ is locally of topologically finite type and quasi-compact, it is called {\it topologically finite type}.
\end{defn}

The relation of Definition~\ref{adic-tft} to Definition~\ref{defn:huber-topologically-finite-type} for affinoid $X$ and $Y$ is addressed in Theorem~\ref{global-tft} under some noetherian condition. 

\begin{defn}\label{defn:Huber-finite} We say that a morphism of complete Tate-Huber pairs $(A, A^+) \to (B, B^+)$ is {\it finite} if the ring homomorphism $A \to B$ is finite and the ring homomorphism $A^+ \to B^+$ is integral.
\end{defn}

\begin{rmk} Our definition coincides with the definition in Huber's book \cite[(1.4.2)]{H3} due to the following (easy) Lemma.
\end{rmk}

\begin{lemma}\label{lemma:huber-finite-finite-type} A finite morphism of complete Tate-Huber pairs $f\colon (A, A^+) \to (B, B^+)$ is of topologically finite type.
\end{lemma}
\begin{proof}
We choose a set $(y_1, \dots, y_m)$ of $A$-module generators for $B$. After multiplying by some power of a pseudo-uniformizer $\varpi$ we can assume that $y_i\in B^+$ for all $i$. Then we use the universal property \cite[Lemma 3.5 (i)]{H1} to define the continuous surjective morphism
\[
g\colon A\langle T_1, \dots, T_m \rangle \to B
\]
as the unique continuous $A$-linear homomorphism such that $f(T_i)=y_i$. It is easily seen to be surjective, and it is open by Remark~\ref{rmk:open-mapping}. Moreover, $B$ is integral over $A^+\langle T_1,  \dots, T_m \rangle$ since it is even integral over $A^+$ by the definition of finiteness.
\end{proof}

\begin{lemma}\label{lemma:top-finite-type-integral-finite} Let $f\colon(A, A^+) \to (B, B^+)$ be a topologically finite type morphism of complete Tate-Huber pairs such that $B^+$ is integral over $A^+$. Then there exist rings of definition $A_0 \subset A$ and $B_0\subset B$ such that $f(A_0) \subset B_0$ and $B_0$ is finite over $A_0$. In particular, $(A, A^+) \to (B, B^+)$ is finite.
\end{lemma}
\begin{proof}  
We use Remark~\ref{rmk:open-mapping} to find an open, surjective morphism 
\[
h\colon A\langle T_1, \dots, T_n \rangle \twoheadrightarrow B.
\]
Clearly $B^+$ is integral over $A^+\langle T_1, \dots, T_n \rangle$. The topological generators $b_i\coloneqq h(T_i) \in B^+$ are integral over $A^{+}$. \smallskip

Pick monic polynomials $F_i\in A^+[T]$ such that $F_i(b_i)=0$ for all $i$. We look at the coefficients $\{a_{i,j}\}\in A^+ \subset A^\circ$ of the polynomials $F_i$. There are only finitely many of them. We claim that we can find a pair of definition $\left(A_0, \varpi \right) \subset A^+$ such that $A_0$ contains every $a_{i,j}$. Indeed, we pick any ring of definition $A'_0$ in $A^+$ and consider the subring generated by $A'_0$ and every $a_{i,j}$. It is easy to see that the resulting ring is open and bounded in $A$, so it is a ring of definition by \cite[Proposition 1.1]{H0}.  \smallskip

Now we define the ring of definition $\left(B_0, \varpi\right)$ as the image $h\left(A_0\langle T_1, \dots, T_n \rangle\right)$. It is open because $h$ is open, and it is bounded because any morphism of Tate rings preserves boundedness. \smallskip

We claim that the natural morphism $A_0 \to B_0$ is finite. It suffices to prove that it is finite mod $\varpi$ by successive approximation and completeness. However, it is clearly finite type mod $\varpi$ since it coincides with the composition:
\[
A_0/\varpi A_0 \to \left(A_0/\varpi A_0\right)[T_1, \dots, T_n] \twoheadrightarrow B_0/\varpi B_0,
\]
and it is integral since $B_0/\varpi B_0$ is algebraically generated over $A_0/\varpi A_0$ by the residue classes $\ov{b_1}, \dots, \ov{b_n}$ that are integral over $A_0/\varpi A_0$ by construction. Thus this map is integral and finite type, hence finite.  \smallskip

Finally, $(A, A^+) \to (B, B^+)$ is finite since $A \to B$ is equal to the finite map 
\[
A_0\left[\frac{1}{\varpi}\right] \to B_0\left[\frac{1}{\varpi}\right].
\qedhere
\]
\end{proof}

\begin{rmk}\label{rmk:top-finite-type-integral-finite} The proof of Lemma~\ref{lemma:top-finite-type-integral-finite} actually shows more. We can choose $B_0$ to contain any finite set of elements $x_1, \dots, x_m\in B^+$. Indeed, the proof just goes through if one replaces $h\colon A\langle T_1, \dots, T_n \rangle \to B$ at the beginning of the proof with the continuous $A$-algebra morphism
\[
h'\colon A\langle T_1, \dots, T_n\rangle \langle X_1, \dots, X_m\rangle \to B
\]
satisfying $h'(T_i)=b_i$ and $h'(X_j)=x_j$. Existence of such a morphism follows from the universal property of restricted power series (see \cite[Lemma 3.5(i)]{H1})
\end{rmk}

\begin{lemma}\label{lemma:huber-finite-iso} Let $f\colon (A, A^+) \to (B, B^+)$ be a finite morphism of complete Tate-Huber pairs. If $f$ induces a bijection $A\simeq B$ then $f$ is an isomorphism of Tate-Huber pairs.
\end{lemma}
\begin{proof}
We note that $B^+$ is integral over $A^+$ by the definition of a finite morphism, and $f$ is open by Remark~\ref{rmk:open-mapping}. An open continuous bijection is a homeomorphism. So we are only left to show that $f$ induces a bijection $A^+\simeq B^+$. Now note that $A^+$ is integrally closed in $A=B$ and $B^+$ is integral over $A^+$ because $f$ is assumed to be finite. Thus, $A^+=B^+$ finishing the proof. 
\end{proof}

\begin{defn}\label{defn:adic-finite-global} A morphism of analytic adic spaces $f\colon  X \to Y$ is called {\it finite}, if there is a covering of $Y$ by Tate affinoids $\{V_{i}\}_{i\in I}$ such that each $U_i\coloneqq f^{-1}(V_i)$ is an open Tate affinoid subset of $X$, and the natural morphism $(\O_Y(V_i), \O_Y^+(V_i)) \to (\O_X(U_i), \O_X^+(U_i))$ is finite (in the sense of Definition \ref{defn:Huber-finite}) for all $i$.
\end{defn}

The relation between Definition~\ref{defn:adic-finite-global} and Definition~\ref{defn:Huber-finite} in the case of affinoid $X$ and $Y$ is addressed in Theorem~\ref{thm:global-finite} under some noetherian constraints.

\begin{defn} A Tate-Huber pair $(A, A^+)$ is called {\it strongly noetherian} if $A\langle T_1, \dots, T_n \rangle$ is noetherian for all $n$.
\end{defn}

\begin{lemma}\label{noeth-preserve} Let $(A, A^+)$ be a strongly noetherian complete Tate-Huber pair. A topologically finite type complete $(A, A^+)$-Tate-Huber pair $(B, B^+)$ is strongly noetherian as well.
\end{lemma}
\begin{proof}
This is proven in \cite[Corrolary 3.4]{H1}.
\end{proof}

\begin{defn}\label{defn:strongly-noetherian-space} An analytic adic space $S$ is {\it locally strongly noetherian} if every point $x\in S$ has an affinoid open neighborhood isomorphic to $\Spa(A, A^+)$ for some strongly noetherian complete Tate-Huber pair $(A, A^+)$. 

An analytic adic space $S$ is {\it strongly noetherian} if it is locally strongly noetherian and quasi-compact.
\end{defn}

\begin{lemma}\label{lemma:strongly-noetherian-Tate-affinoid} A Tate affinoid analytic adic space $\Spa(A, A^+)$ is strongly noetherian if and only if $(A, A^+)$ is a strongly noetherian Tate-Huber pair.
\end{lemma}
\begin{proof}
    The ``if'' direction is clear. Now suppose that $X\coloneqq \Spa(A, A^+)$ is strongly noetherian as an analytic adic space. We wish to show that the Tate-Huber pair $(A, A^+)$ is strongly noetherian.\smallskip
    
    By assumption, each point $x\in X$ has an affinoid open neighborhood $U_x=\Spa(A_x, A_x^+)$ with a strongly complete Tate-Huber pair $(A_x, A_x^+)$. Since rational subdomains form a basis of topology of $X$ and strong noetherianess is preserved by passing to rational subdomains (see Lemma~\ref{noeth-preserve}), we can assume that each $U_x\subset X$ is a rational subdomain. Then the claim follows from \cite[Corollary 1.4.19]{KedAr} and sheafiness of $A$.
\end{proof}



\begin{thm}\label{global-tft} Let $f\colon \Spa(B, B^+) \to \Spa(A, A^+)$ be a topologically finite type morphism of strongly noetherian Tate affinoids. Then the corresponding map
\[
f^{\#}\colon (A, A^+) \to (B, B^+)
\]
is topologically finite type.
\end{thm}
\begin{proof}
This is proven in \cite[Satz 3.3.23]{H2}.
\end{proof}

\begin{thm}\label{thm:global-finite} Let $(A, A^+)$ be a strongly noetherian Tate-Huber pair, and $f\colon X \to \Spa(A, A^+)$ a finite morphism. Then $X$ is affinoid and the morphism $(A, A^+) \to (\O_X(X), \O_X^+(X))$ is finite. 
\end{thm}
\begin{proof}
This follows from the combination of \cite[Satz 3.6.20 and Korollar 3.12.12]{H2}. 
\end{proof}

\begin{rmk} We do not know Theorem~\ref{global-tft} or Theorem~\ref{thm:global-finite} hold without the extra strong noetherianness assumption. 
\end{rmk}

\subsection{Completed Tensor Products}

The main goal of this section is to prove that under certain assumptions, completed tensor products of Tate rings coincide with usual tensor products. This should be well-known to the experts, but it seems difficult to extract the proof from the existing literature. \smallskip

For the rest of the section, we fix a complete Tate-Huber pair $\left(A, A^+\right)$ with a choice of a pair of definition $(A_0, \varpi)$. We recall the notion of ``the natural $A$-module topology'' for a finite $A$-module $M$:

\begin{defn} A topological $A$-module structure on $M$ is {\it natural} if any $A$-linear map $M\to N$ to a topological $A$-module $N$ is continuous. 
\end{defn}
By considering the identity morphism $\rm{Id}\colon M\to M$, it is clear that the natural $A$-module topology is unique, if it exists. It is not, a priori, clear if any module admits a natural topology. However, it turns out that the natural topology actually always exists on a finite $A$-module.


\begin{lemma}\label{existence-natural} Let $M$ be a finite $A$-module. There is a topology on $M$ that satisfies the definition of the natural $A$-module topology.
\end{lemma}
\begin{proof}
First of all, we claim that the product topology on a finite free module $A^n$ is the natural $A$-module topology on it. Indeed, it suffices to prove the claim in the case $n=1$ by the universal property of direct products. But any $A$-linear map $A \to N$ to a topological $A$-module is clearly continuous. 

Now we deal the case of an arbitrary finitely generated $M$. We choose a surjective morphism $f\colon A^n \to M$ and provide $M$ with the quotient topology. This is clearly a topological $A$-module structure. We want to show that any $A$-linear morphisms $g\colon M\to N$ to a topological $A$-module $N$ is continuous. We consider the diagram:
\[
\begin{tikzcd}
A^n\arrow{d}{f} \arrow{rd}{h} & \\
M \arrow{r}{g} &N.
\end{tikzcd}
\]
Then for any open $U\subset N$ we see that $f^{-1}(g^{-1}(U))=h^{-1}(U)$ is open since $h$ is continuous by the argument above. The definition of quotient topology implies that $g^{-1}(U)$ is open as well. Thus $g$ is indeed continuous.
\end{proof}

\begin{rmk}\label{rmk:natural-topology-warn} We warn the reader that the natural topology on a finite $A$-module may not be complete as $A$ may have non-closed ideals. 
\end{rmk}

\begin{lemma}\label{finite-natural} Let $M$ be a finite, complete, first countable topological $A$-module. Then the topology on $M$ is the natural $A$-module topology. If $(B, B^+)$ is a finite complete $(A,A^+)$-Tate-Huber pair, then there is a ring of definition $B_0$ and a surjective $A$-linear morphism $p\colon A^n\to B$ with $p(A_0^n)=B_0$.
\end{lemma}
\begin{proof}
In the case of a finite complete first countable topological $A$-module $M$, any surjection $A^n \to M$ must be open by \cite[Lemma 2.4(i)]{H1}. Thus $M$ carries the quotient topology. This topology satisfies the condition of the natural topology by the argument in the last paragraph of Lemma~\ref{existence-natural}. \smallskip

As for the second claim, we use Lemma~\ref{lemma:top-finite-type-integral-finite} to find rings of definition $A_0, B_0$ such that $B_0$ is finite over $A_0$. Choose some generators $b_1, \dots, b_n$ for $B_0$ over $A_0$ and consider the morphism $p\colon A^n \to B$ that sends $\left(a_1, \dots, a_n\right)$ to $a_1b_1+\dots+a_nb_n$. Then clearly $p(A_0^n)=B_0$.
\end{proof}

Finally, we recall that given two morphisms $f\colon A \to B$ and $g\colon A\to C$ of Tate rings there is a canonical way to topologize the tensor product $B\otimes_A C$. Namely, we pick some rings of definition $B_0\subset B$ and $C_0 \subset C$ such that $f(A_0)\subset B_0$ and $g(A_0) \subset C_0$. Then we topologize $A\otimes_B C$ by requiring the image $(B\otimes_A C)_0 \coloneqq \operatorname{Im}(B_0\otimes_{A_0} C_0 \to B\otimes_A C)$ with its $\varpi$-adic topology to be a ring of definition in $B\otimes_A C$. Then it is straightforward\footnote{See \cite[Proposition 2.4.18]{H2} or \cite[Theorem 5.5.4]{Seminar}.} to see that the Tate ring $B\otimes_A C$ ring satisfies the expected universal property in the category of Tate rings. In particular, this shows that this construction does not depend on the choice of rings of definition $A_0, B_0, C_0$. But we warn the reader that $B\otimes_A C$ need not be (separated and) complete even if $A, B$ and $C$ are; its completion is denoted by $B\wdh{\otimes}_A C$.

\begin{lemma}\label{lemma:finite-tensor-product} Let $f\colon \left(A, A^+\right) \to \left(B, B^+\right)$ be a finite morphism of complete Tate-Huber pairs, and let $g\colon A \to C$ be any morphisms of Tate rings. Then the topologized tensor product $B\otimes_A C$ has the natural $C$-module topology.
\end{lemma}
We note that this is not automatic from Lemma~\ref{finite-natural} since $B\otimes_A C$ is not necessarily complete.
\begin{proof}
We use Lemma~\ref{finite-natural} to find a ring of definition $B_0\subset B$ and a surjection $p\colon A^n\to B$ such that $p(A_0^n)=B_0$. Then after tensoring it with $C$ we get a surjective morphism
$C^n \to B\otimes_A C$, and tensoring the surjection $A_0^n\to B_0$ with $C_0$ we get a surjection $C_0^n\to B_0\otimes_{A_0}C_0$. Combining these, we get a commutative diagram:
\[
\begin{tikzcd}
 & C^n \arrow[ld, two heads, "p_C"] &C_0^n  \arrow[l, hook']  \arrow[d, two heads] \\
B\otimes_A C   & \arrow[l, hook'] (B\otimes_A C)_0& B_0\otimes_{A_0} C_0. \arrow[l, two heads]
\end{tikzcd}
\]
By definition, $(B\otimes_A C)_0$ with its $\varpi$-adic topology is open in $B\otimes_A C$, so \[p_C|_{C_0^n}\colon C_0^n \to B\otimes_A C\] is open onto an open image. Hence, $p_C$ is also open, so $B\otimes_A C$ has the quotient topology via $p_C$ as desired.
\end{proof}

\begin{lemma}\label{complete-base-change} Let $f\colon (A, A^+)\to (B, B^+)$ be a finite morphism of complete Tate-Huber pairs with noetherian $A$. Suppose that $A \to C$ is a continuous morphism of noetherian, complete Tate rings. Then the natural morphism $B\otimes_A C \to B\wdh{\otimes}_A C$ is a topological isomorphism.
\end{lemma}
\begin{proof}
Lemma~\ref{lemma:finite-tensor-product} implies that $B\otimes_A C$ carries the natural $C$-module topology. Then we use \cite[Lemma 2.4(ii)]{H1} to conclude that $B\otimes_A C$ is already complete, so the completion map $B\otimes_A C \to B\wdh{\otimes}_A C$ is a topological isomorphism.
\end{proof}

\begin{cor}\label{cor:complete-localization} Let $f\colon (A, A^+) \to (B, B^+)$ be a finite morphism of complete Tate-Huber pairs with a strongly noetherian Tate ring $A$. Then the natural morphism 
\[
B\otimes_A A\left\langle \frac{f_1}{g}, \dots, \frac{f_n}{g}\right\rangle \to B\left\langle\frac{f_1}{g}, \dots, \frac{f_n}{g}\right\rangle
\] is a topological isomorphism for any choice of elements $f_1, \dots, f_n, g \in A$ generating the unit ideal in $A$.
\end{cor}
\begin{proof}
First of all, we note that $A\left\langle \frac{f_1}{g}, \dots, \frac{f_n}{g}\right\rangle$ is a complete Tate ring. Moreover, it is noetherian by Lemma~\ref{noeth-preserve} so we can apply Lemma~\ref{complete-base-change} with $C=A\left\langle \frac{f_1}{g}, \dots, \frac{f_n}{g}\right\rangle$. Thus the question is reduced to showing that the natural morphism
\[
B\wdh{\otimes}_A A\left\langle \frac{f_1}{g}, \dots, \frac{f_n}{g}\right\rangle \to B\left\langle\frac{f_1}{g}, \dots, \frac{f_n}{g}\right\rangle
\]
is a topological isomorphism. But this easily follows from the universal properties of topologized tensor products (see \cite[Proposition 2.4.18]{H2} or \cite[Theorem 5.5.4]{Seminar}), completions (see \cite[Proposition 7.2.2]{Seminar}), and completed rational localizations (see \cite[(1.2) on p. 517]{H1}).
\end{proof}

\subsection{Flat Morphisms of Adic Spaces}
We discuss the notion of a flat morphism of adic spaces. This notion is not discussed much in the existing literature, so we provide the reader with some facts that we are using in the paper. 

\begin{defn}\label{defn:huber-flat} A morphism of analytic adic spaces $f\colon  X \to Y$ is called {\it flat}, if the natural morphism $\O_{Y, f(x)} \to \O_{X,x}$ is flat for any point $x\in X$.
\end{defn}

Similarly to the case of formal schemes, we will soon describe flatness of strongly noetherian Tate affinoids in more concrete terms.

\begin{lemma}\label{flat-adic-preliminary} Let $X=\Spa(A, A^+)$ be a strongly noetherian Tate affinoid adic space, and let $x\in X$ be a point corresponding to a valuation $v$ with support $\mathfrak p$. Then the natural morphism $r_x\colon A_{\mathfrak p} \to \O_{X,x}$ is faithfully flat.
\end{lemma}
\begin{proof}
We note that rational subdomains form a basis of the topology on an affinoid space, so $\O_{X,x}$ is equal to the filtered colimit of $\O_X(U)$ over all rational subdomains in $X$ containing $x$. We use \cite[(II.1), (iv) on page 530]{H1} to note that $A \to \O_X(U)$ is flat for each such $U$. Since flatness is preserved by filtered colimits, we conclude that $A \to \O_{X, x}$ is flat. Note that this implies that $A_{\mathfrak p} \to \O_{X,x}$ is flat as well. Indeed, this easily follows from the fact that for any $A_{\mathfrak p}$-module $M$ we have isomorphisms
\[
M\otimes_{A_{\mathfrak p}} \O_{X,x}\cong(M\otimes_A A_{\mathfrak p})\otimes_{A_{\mathfrak p}} \O_{X,x}\cong M\otimes_A \O_{X,x}.
\]

The discussion above shows that $r_x\colon A_{\mathfrak p} \to \O_{X,x}$ is flat, but we also need to show that it is faithfully flat. In order to prove this claim it suffices to show that $A_{\mathfrak p} \to \O_{X,x}$ is a local ring homomorphism. We recall that the maximal ideal $\mathfrak m_{x} \subset \O_{X,x}$ is given as
\[
\mathfrak m_{x}=\left\{ f\in \O_{X,x} | \ v\left(f\right)=0 \right\}.
\]
We need to show $r_x(\mathfrak pA_{\mathfrak p}) \subset \mathfrak m_x$. We pick any element $h\in \mathfrak pA_{\mathfrak p}$. It can be written as $f/s$ for $f\in \mathfrak p$ and $s\in A\setminus \mathfrak p$, and we need to check that $v\left(\frac{f}{s}\right)=0$. The very definition of $\mathfrak p$ as the support of $v$ implies that $v\left(f\right)=0$ and $v\left(s\right)\neq 0$. Thus 
\[
v\left( \frac{f}{s}\right)=\frac{v(f)}{v(s)}=0.
\qedhere
\]
\end{proof}

\begin{lemma}\label{lemma:huber-flat-affinoids} Let $f\colon  \Spa(B, B^+) \to \Spa(A, A^+)$ be a flat morphism of strongly noetherian Tate affinoid adic spaces. The natural morphism $f^\sharp\colon A \to B$ is flat as well.
\end{lemma}
\begin{proof}
We start the proof by noting that \cite[Lemma 1.4]{H1} implies that for any maximal ideal $\mathfrak m \subset B$ there is a valuation $v\in \Spa(B, B^+)$ such that $\supp(v)=\mathfrak m$. It is easy to see that 
\[
\supp \left(w\right)=(f^\#)^{-1}\left(\mathfrak m\right)=:\mathfrak p,
\] 
where $w=f(x)\in \Spa(A, A^+)$. We use Lemma \ref{flat-adic-preliminary} to conclude that we have a commutative square
\[
\begin{tikzcd}
B_{\mathfrak m} \arrow{r}{r_{\mathfrak m}} & \O_{X,v}\\
A_{\mathfrak p} \arrow{u}{f^\sharp_{\mathfrak p}} \arrow{r}{r_{\mathfrak p}} &\O_{Y, w} \arrow{u}{f^\#_w}
\end{tikzcd}
\]
with $r_{\mathfrak m}$ and $r_{\mathfrak p}$ being faithfully flat. It is easy to see now that flatness of $f^\#_w$ implies flatness of $f^\sharp_{\mathfrak p}$. Finally we note that $\mathfrak m$ was an arbitrary maximal ideal in $B$, so $f^\sharp\colon A \to B$ is flat. 
\end{proof}

\begin{rmk} We warn the reader that it is unknown whether Lemma~\ref{lemma:huber-flat-affinoids} remains true if one drops the strongly noetherian hypothesis. Even the case of rational embeddings is open. However, there are some positive results in this direction in \cite[\textsection 2.4]{KedLiu2}.
\end{rmk}

\begin{rmk}\label{rmk:classical-enough} Let $K$ be a complete rank-$1$ valuation field, and 
\[
    f\colon X=\Spa(B, B^\circ) \to Y=\Spa(A, A^\circ)
\]
a morphism of topologically finite type affinoid adic $\Spa(K, \O_K)$-spaces. Then $f$ sends classical points to classical points (i.e., $\frak{p}$ defined in the proof of Lemma~\ref{lemma:huber-flat-affinoids} is maximal). In particular, the proof of Lemma~\ref{lemma:huber-flat-affinoids} shows that $A\to B$ is flat if 
\[
f_{f(x)}^\sharp \colon \O_{Y, f(x)} \to \O_{X, x}
\]
is flat for any {\it classical} point $x\in X$.
\end{rmk}

At this point, we have not given any non-trivial example of a flat morphism. The next lemma gives the main example of interest:

\begin{lemma}\label{lemma:other-direction} Let $K$ be a complete rank-$1$ valuation ring, and $f^\sharp \colon A \to B$ be a flat morphism of topologically finite type $K$-algebras. Then the corresponding morphism $f\colon X=\Spa(B, B^\circ) \to Y=\Spa(A, A^\circ)$ is flat.
\end{lemma}
\begin{proof}
    {\it Step $1$. The natural morphism $\O_{Y, f(x)} \to \O_{X, x}$ is flat for any {\it classical} point $x\in X$}: Let $x$ correspond to a maximal ideal $\mathfrak{m}\subset B$, and $y=f(x)$ to a maximal ideal $\mathfrak{n} \subset A$. Then \cite[Proposition 4.1/2]{B} implies that 
    \[
    \wdh{A}_{\mathfrak{n}}\simeq \wdh{\O}_{Y, y}, \ \wdh{B}_{\frak{m}} \simeq \wdh{\O}_{X,x}
    \]
    where all completions are taken with respect to the corresponding maximal ideal. Furthermore, the rings $A_{\mathfrak{n}}$, $B_{\frak{m}}$, $\O_{Y, y}$, and $\O_{X, x}$ are noetherian by \cite[Proposition 3.1/3(i) and 4.1/6]{B}. Therefore, \cite[\href{https://stacks.math.columbia.edu/tag/0523}{Tag 0523}]{stacks-project} ensures that flatness of $A_{\frak{n}} \to B_{\frak{m}}$ implies flatness of $\O_{Y, y}\to \O_{X, x}$.
    
    {\it Step $2$. Finish the argument}: We pick a point $x\in X$ with the image $y=f(x)$. Then the morphism
    \[
    f^\sharp_y \colon \O_{Y, y} \to \O_{X, x}
    \]
    can be rewritten as 
    \[
    \colim_{y\in V\subset Y} \O_Y(V) \to \colim_{x\in U \subset X} \O_{X}(U).
    \]
    Therefore, it suffices to show that the natural morphism
    \[
    \O_Y(V) \to \O_X(U)
    \]
    is flat for any open affinoids $V\subset Y$, $U\subset X$ such that $f(U)\subset V$. Now Step~$1$ ensures that the morphism
    \[
    \O_{V, f(u)} \to \O_{U, u}
    \]
    is flat for any classical $u\in U$. Then Remark~\ref{rmk:classical-enough} ensures that $\O_V(V) \to \O_U(U)$ is flat finishing the proof. 
\end{proof}

\begin{rmk} One can also show that any smooth morphism $f\colon X \to Y$ of locally strongly noetherian adic spaces is flat. The proof is similar to Step~$2$ of the proof of Lemma~\ref{lemma:other-direction} using \cite[Lemma 1.1.19(a)]{KedAr} and \cite[Lemma 1.7.6]{H3} for a cofinal system of opens $V\subset Y$, $U\subset X$ such that $f(U)\subset V$.
\end{rmk}

\begin{rmk} One can also show that $X_L \to X$ is flat for any extension of non-archimedean fields $K\subset L$ and a rigid-analytic $K$-space $X$. Again, the main input is to show that the moprhism
\[
A \to A\wdh{\otimes}_K L
\]
is flat for any $K$-affinoid algebra $A$. This is worked out in \cite[Lemma 1.1.5]{C-irreducible}.
\end{rmk}

\begin{rmk} We do not know if flatness of $A \to B$ implies flatness of $\Spa(B, B^+) \to \Spa(A, A^+)$  in general (even under the strongly noetherian assumption).
\end{rmk}

\begin{lemma}\label{lemma:huber-flat-descent} Let $f\colon X=\Spa(B, B^+) \to Y=\Spa(A, A^+)$ be a finite morphism of strongly noetherian Tate affinoids, and $g\colon Z=\Spa(C, C^+) \to \Spa(A, A^+)$ be a surjective flat morphism of strongly noetherian Tate affinoids. Then $f$ is an isomorphism if and only if $f'\colon X\times_Y Z \to Z$ is. 
\end{lemma}
\begin{proof}
We note that $f'$ is finite by \cite[Lemma 1.4.5(i)]{H3}, so Lemma~\ref{lemma:huber-finite-iso} ensures that it suffices to show that $A \to B$ is a (topological) isomorphism if and only if $C \to C\wdh{\otimes}_A B$ is. We can ignore topologies by Remark~\ref{rmk:open-mapping}. \smallskip

Now we note that Lemma~\ref{complete-base-change} gives that $C\otimes_A B \simeq C\wdh{\otimes}_A B$. Thus, it suffices to show that $A\to B$ is an isomorphism if and only if $C \to C\otimes_A B$ is. This follows from the usual faithfully flat descent as $A \to C$ is flat by Lemma~\ref{lemma:huber-flat-affinoids}, and therefore faithfully flat by \cite[Lemma 1.4]{H1}.
\end{proof}

\subsection{Coherent Sheaves}

We review the basic theory of coherent sheaves on locally strongly noetherian adic spaces. \smallskip

We first recall the construction of the $\O_X$-module $\widetilde{M}$ on a strongly noetherian Tate affinoid $X=\Spa(A, A^+)$ associated to a finite $A$-module $M$. For each rational subset $U \subset X$, we have 
\[
\widetilde{M}(U)=\O_X(U)\otimes_{A} M;
\]
\cite[Theorem 2.5]{H1} guarantees that this assignment is indeed a sheaf. 

\begin{defn} An $\O_X$-module $\F$ on a locally strongly noetherian analytic adic space $X$ is {\it coherent} if there is an open covering $X=\cup_{i\in I} U_i$ by strongly noetherian Tate affinoids such that $\F|_{U_i} \cong \widetilde{M}_i$ for a finite $\O_X(U_i)$-module $M_i$.
\end{defn}

\begin{thm}\label{thm:huber-coherent-sheaves} Let $X=\Spa(A, A^+)$ be a strongly noetherian Tate affinoid, and $\F$ a coherent $\O_X$-module. Then
\begin{enumerate}
    \item\label{thm:huber-coherent-sheaves-1} there is a unique finite $A$-module $M$ such that $\F \cong \widetilde{M}$.
    \item\label{thm:huber-coherent-sheaves-2} $\rm{H}^i(X, \F)=0$ for $i\geq 1$.
\end{enumerate}
\end{thm}
\begin{proof}
(\ref{thm:huber-coherent-sheaves-1}) is shown in \cite[Theorem 1.4.18]{KedAr} (see also \cite[Definition 1.4.5]{KedAr} for a definition of $\mathbf{PCoh}_A$), and (\ref{thm:huber-coherent-sheaves-2}) is shown in \cite[Theorem 2.5]{H1}.
\end{proof}

\begin{cor}\label{cor:huber-coherent-sheaves} Let $f\colon X \to Y$ be a finite morphism of locally strongly noetherian adic spaces. Then
\begin{enumerate}
    \item\label{cor:huber-coherent-sheaves-1} coherent $\O_Y$-modules are closed under kernels, cokernels, and extensions in $\mathbf{Mod}_{\O_Y}$;
    \item\label{cor:huber-coherent-sheaves-2} for any coherent $\O_X$-module $\F$, $f_*\F$ is a coherent $\O_Y$-module.   
\end{enumerate}
\end{cor}
\begin{proof}
It suffices to prove the claim under the additional assumption that $Y$ is a strongly noetherian Tate affinoid. Now both parts easily follow from Theorem~\ref{thm:global-finite}, Theorem~\ref{thm:huber-coherent-sheaves} and flatness of $\O_Y(Y) \to \O_Y(U)$ for a rational subdomain $U\subset Y$ \cite[(II.1), (iv) on page 530]{H1}.  
\end{proof}

\subsection{Closed Immersions}
In this section we discuss the notion of closed immersion in the context of locally strongly noetherian adic spaces. 


\begin{defn} We say that a morphism $f\colon X\to Y$ of analytic adic spaces is an {\it open immersion} if $f$ is a homeomorphism of $X$ onto an open subset of $Y$, and the map $f^{-1}\O_Y \to \O_X$ is an isomorphism.
\end{defn}

\begin{rmk} Remark~\ref{rmk:huber-forgetful-conservative} ensures that $f\colon X \to Y$ is an open immersion if and only if $f$ is an isomorphism onto an open adic subspace of $Y$. 
\end{rmk}

\begin{defn} We say that a morphism $f\colon X \to Y$ of locally strongly noetherian analytic adic spaces is a {\it closed immersion} if $f$ is a homeomorphism of $X$ onto a closed subset of $Y$, the map $\O_Y \to f_*\O_X$ is surjective, and the kernel $\mathcal I\coloneqq \ker (\O_Y \to f_*\O_X)$ is coherent.
\end{defn}

\begin{rmk}\label{rmk:support-closed-immersion} If $i\colon X \to Y$ is a closed immersion of (locally strongly noetherian) adic spaces with $\mathcal {I}=\ker (\O_Y \to i_*\O_X)$, then there is a set-theoretic identification:
\[
|X|=\{y\in Y \ | \ (i_*\O_{X})_y \not\simeq 0\} = \{y\in Y \ | \ \mathcal{I}_y\not\simeq \O_{Y,y}\}.
\]
\end{rmk}

\begin{lemma}\label{lemma:huber-closed-tate-topology} Let $Y=\Spa(A, A^+)$ be a strongly noetherian Tate affinoid\footnote{Recall that we always implicitly assume that $(A, A^+)$ is complete.}, and $i\colon X \to Y$ a closed immersion. Then $B\coloneqq \O_X(X)$ is a complete Tate ring, and the natural morphism $i^*\colon A \to B$ is a topological quotient morphism.
\end{lemma}
\begin{proof}
The natural morphism $A \to B$ is clearly continuous as $\O_Y \to i_*\O_X$ is a morphism of sheaves of {\it topological} rings. Therefore, any topologically nilpotent unit $\varpi_A \in A$ defines a topologically nilpotent unit $\varpi\coloneqq i^*(\varpi_A)\in B$. \smallskip

We now show that $B$ is a Tate ring. Since $X$ is closed in an affinoid, we conclude that $X$ is quasi-compact and quasi-separated. So we choose a finite covering $X=\cup_{i=1}^n U_i$ by open affinoid $U_i=\Spa(B_i, B_i^+)$. Then 
\[
B \subset \prod_{i=1}^n B_i
\]
and the topology on $B$ coincides with the subspace topology. Each $B_i$ admits a ring of definition $B_{i, 0}$, and we can assume that the topology on every $B_{i, 0}$ is the $\varpi$-adic topology (possibly after replacing $\varpi_A$ with a power). We claim that 
\[
B_0 \coloneqq \left(\prod_{i=1}^n B_{i, 0}\right) \cap B =\prod_{i=1}^n \left(B_{i, 0} \cap B \right)
\]
is a ring of definition in $B$. It suffices to show the topology on $B_0$ induced from $\prod_{i=1}^n B_{i, 0}$ coincides with the $\varpi$-adic topology. This follows from the equalities
\[
\varpi^n \left((\prod_{i=1}^n B_{i, 0}) \cap B\right) = \left(\varpi^n\prod_{i=1}^n B_{i, 0}\right) \cap B
\]
that, in turn, follow from the fact that $\varpi$ is invertible in $B$. \smallskip

Now we address completeness of $B$. By a similar reason as above, we see that there is a short exact sequence
\[
0 \to B \xr{d} \prod_{i=1}^n B_i \xr{a} \prod_{i, j} \O_X(U_i\cap U_j)
\]
such that $d$ is a topological embedding and $a$ is continuous. Using that $X$ is quasi-separated, we can cover each $U_i \cap U_j$ by a finite number of affinoids $V_{i, j, k}$. Thus, we get a short exact sequence
\[
0 \to B \xr{d} \prod_{i=1}^n B_i \xr{b} \prod_{i, j, k} \O_X(V_{i,j,k})
\]
such that $d$ is a topological embedding and $b$ is continuous. Every $B_i=\O_X(U_i)$ and $\O_X(V_{i, j, k})$ is complete by construction. Therefore, we conclude that $B$ is closed inside the complete Tate ring $\prod_{i=1}^n B_i$. Thus, it is also complete. \smallskip

Finally, we show that $A \to B$ is a topological quotient morphism. We consider the short exact sequence
\[
0 \to \mathcal{I} \to \O_Y \to i_*\O_X \to 0.
\]
Theorem~\ref{thm:huber-coherent-sheaves}(\ref{thm:huber-coherent-sheaves-2}) ensures that $\rm{H}^1(Y, \mathcal{I})=0$, so $A\to B$ is surjective. Now $A \to B$ is a surjective continuous morphism of complete Tate rings, so it is open by Remark~\ref{rmk:open-mapping}. In particular, it is a topological quotient morphism. 
\end{proof}

\begin{lemma}\label{example-closed} Let $(A, A^+)$ be a strongly noetherian complete Tate-Huber pair, and let $I$ be an ideal in $A$. We define $(A^+/I\cap A^+)^c$ to be the integral closure of $A^+/(I\cap A^+)$ in $A/I$. Then $(A/I, (A^+/I\cap A^+)^c)$ is a complete strongly noetherian Tate-Huber pair, and the morphism 
\[
\Spa(A/I, (A^+/I\cap A^+)^c) \to \Spa(A, A^+)
\]
is a closed immersion
\end{lemma}
\begin{proof}
First of all, we note that $A/I$ is complete by \cite[Proposition 2.4(ii)]{H1} and the natural morphism $p\colon A \to A/I$ is open. Now we show that $(A/I, (A^+/I\cap A^+)^c)$ is also a Tate-Huber pair. We choose a pair of definition $(A_0, \varpi)$ with $\varpi$ being a pseudo-uniformizer in $A$. Then openness of $p$ implies that $p(A_0)$ is open in $A/I$. Moreover, its quotient topology coincides with the $p(\varpi)$-adic topology, so it is a ring of definition in $A/I$. Also, $p(\varpi)$ is a topologically nilpotent unit in $A/I$, so $A/I$ is a Tate ring. A similar argument shows that $(A^+/I\cap A^+)^c$ is an open subring of $A/I$ that is contained in $(A/I)^{\circ}$.  \smallskip

We claim that $A/I$ is strongly noetherian. It suffices to show that
\[
A\left\langle T_1, \dots, T_n \right\rangle \to \left(A/I\right)\left\langle T_1, \dots, T_n \right\rangle
\]
is surjective for each $n\geq 1$. By induction, it suffices to prove the claim for $n=1$. We pick an element $f\in (A/I)\langle T \rangle$; it can be written as $f=\sum_i \ov{a_i}T^i$ for some $a_i\in A$ such that $\{\ov{a_i}\}$ is a null-system in $A/I$. This means that for any $m$ there is $N_m$ such that: 
\[
\ov{a_i}\in p(\varpi)^mp(A_0)
\]
for any $i\geq N_m$. Thus we can we can find a sequence $(b_i)$ of elements of $A$ such that $\ov{b_i}=\ov{a_i}$ for any $i\geq 0$ and $b_i\in \varpi^mA_0$ for any $i\geq N_m$. This means that $\sum_i b_iT^i$ lies in $A\langle T\rangle$ and its image in $\left(A/I\right)\langle T\rangle$ coincides with $f$. \smallskip

Now we check that the natural morphism $i\colon X\coloneqq\Spa(A/I, (A^+/I\cap A^+)^c) \to Y\coloneqq \Spa(A, A^+)$ is a closed immersion. Firstly, we note that topologically we have an equality
\[
i(X) = V(I)\coloneqq \{ x\in \Spa(A, A^+) \ | \ v_x(I)=0\}
\]
with $v_x$ being the valuation corresponding to a point $x$. We show that this set is closed. It suffices to show that the set
\[
V(f)\coloneqq \{ x\in \Spa(A, A^+) \ | \ v_x(f)=0\}
\]
is closed for any $f\in I$ since $V(I)=\cap_{f\in I} V(f)$. And $V(f)$ is closed as its complement is equal to the union of the rational subdomains:
\[
    Y\setminus V(f) = \bigcup_{n\in \N} Y\left(\frac{\varpi^n}{f}\right).
\]

 We also need to check that the map $\O_Y \to i_*\O_X$ is surjective with coherent kernel. Clearly, $i\colon X=\Spa(A/I, A^+/(I\cap A^+)^c) \to Y=\Spa(A, A^+)$ is finite, so $i_*\O_X$ is a coherent $\O_Y$-module by Corollary~\ref{cor:huber-coherent-sheaves}(\ref{cor:huber-coherent-sheaves-2}). Thus, Corollary~\ref{cor:huber-coherent-sheaves}(\ref{cor:huber-coherent-sheaves-1}) ensures that $\ker(\O_Y \to i_*\O_X)$ is coherent. \smallskip
 
Now we show that $\O_Y \to i_*\O_X$ is surjective. It suffices to show that for any rational subdomain $U=Y\left(\frac{f_1}{g}, \dots, \frac{f_n}{g}\right)$ the morphism $\O_Y(U) \to (i_*\O_X)(U)$ is surjective. This boils down to showing that the map
\[
A\left\langle \frac{f_1}{g}, \dots, \frac{f_n}{g} \right\rangle \to (A/I)\left\langle \frac{f_1}{g}, \dots, \frac{f_n}{g}\right\rangle
\]
is surjective. Consider the commutative diagram
\[
\begin{tikzcd}
A\left\langle T_1, \dots, T_n \right\rangle \arrow[r, two heads] \arrow[d, two heads] & \left(A/I\right)\left\langle T_1, \dots, T_n \right\rangle \arrow[d, two heads] \\
A\left\langle \frac{f_1}{g}, \dots, \frac{f_n}{g} \right\rangle \arrow{r} & (A/I)\left\langle \frac{f_1}{g}, \dots, \frac{f_n}{g}\right\rangle
\end{tikzcd}
\]
where the upper horizontal arrow is surjective by the discussion above. This implies that the lower horizontal arrow is surjective as well. 
\end{proof}

\begin{lemma}\label{lemma:huber-closed-commutes-affinoid} Let $f\colon \Spa(B, B^+) \to \Spa(A, A^+)$ be a morphism of strongly noetherian Tate affinoids, and $I\subset A$ an ideal. Then the natural morphism 
\[
\Spa \left(B/IB, (B^+/B^+\cap IB)^c\right) \to \Spa(B, B^+)\times_{\Spa(A, A^+)} \Spa\left(A/I, (A^+/I\cap A^+)^c\right).
\]
is an isomorphism.
\end{lemma}
\begin{proof}
Lemma~\ref{complete-base-change} applied to the finite morphism $A \to A/I$ ensures that\footnote{In the formula below, $(A/I\otimes_A B)^+$ stands for the integral closure of $\rm{Im}\left(A^+/I\cap A^+)^c \otimes_{A^+} B^+ \to A/I\otimes_A B\right)$ inside $A/I\otimes_A B$.} 
\[
 \Spa\left(B, B^+\right)\times_{\Spa(A, A^+)} \Spa\left( A/I, (A^+/I\cap A^+)^c\right) \simeq \Spa\left(A/I\otimes_A B, (A/I\otimes_A B)^+\right).
\]
    
Lemma~\ref{example-closed} implies that $(A/I, (A^+/I\cap A^+)^c)$ is a complete Tate-Huber pair. Thus Lemma~\ref{lemma:finite-tensor-product} ensures that the tensor product topology on $(A/I)\otimes_A B$ coincides with the natural topology. Lemma~\ref{example-closed} also implies that $B/IB$ is a complete Tate ring, in particular, its topology is first countable. Thus, Lemma~\ref{finite-natural} ensures that its topology coincides with the natural topology. Therefore, the universal property of the natural topology guarantees that the canonical algebraic isomorphism $\left(A/I\right)\otimes_A B \simeq B/IB$ preserves topologies on both sides. \smallskip

Now we recall that 
\begin{align*}
 \left(\left(A/I\right)\otimes_A B\right)^+ & =  \rm{Im}\left( \left(A^+/I\cap A^+\right)^c\otimes_{A^+} B^+ \to \left(A/I\right)\otimes_A B\right)^c \\
 & = \rm{Im}\left( B^+/(I\cap A^+)B^+ \to B/IB\right)^c.
\end{align*}
This admits a natural morphism 
\[
\rm{Im}\left( B^+/\left(I\cap A^+\right)B^+ \to B/IB\right)^c \to \left(B^+/\left(B^+\cap IB\right)\right)^c
\]
that is both injective and surjective. This implies that
\begin{align*}
    \Spa\left(B, B^+\right)\times_{\Spa(A, A^+)} \Spa\left(\left(A/I\right), A^+/(I\cap A^+)^c\right) & \simeq \Spa\left(\left(A/I\right)\otimes_A B, (A/I\otimes_A B)^+\right) \\
    & \simeq \Spa(B/IB, (B^+/(B^+\cap IB))^c).
    \qedhere
\end{align*}
\end{proof}

\begin{cor}\label{cor:huber-top} Let $Y=\Spa(A, A^+)$ be a strongly noetherian Tate affinoid, $I\subset A$ an ideal, and $X=\Spa(A/I, (A^+/I\cap A^+)^c)$. Then the natural map
\[
\widetilde{I} \to \ker (\O_Y \to i_*\O_X)
\]
is an isomorphism.
\end{cor}
\begin{proof}
This follows from the fact that the formation of $\Spa(A/I, (A^+/I\cap A^+)^c)$ commutes with base change by Lemma~\ref{lemma:huber-closed-commutes-affinoid}, and the fact that $I\O_Y(U)=I\otimes_A \O_Y(U)$ by $A$-flatness of $\O_Y(U)$ for a rational subdomain $U\subset Y$. 
\end{proof}

\begin{cor}\label{cor:huber-closed-classification} Let $Y=\Spa(A, A^+)$ be a strongly noetherian Tate affinoid, and let $i\colon X \to Y$ be a closed immersion. Then it is isomorphic to the closed immersion from $\Spa(A/I, (A^+/I\cap A^+)^c)$ for a unique ideal $I\subset A$.
\end{cor}
\begin{proof}
Uniqueness of $I$ is easy: Corollary~\ref{cor:huber-top} implies that, for a closed immersion 
\[
X=\Spa(A/I, (A^+/I\cap A^+)^c) \to Y=\Spa(A, A^+),
\]
we can recover $I$ as $\Gamma(Y, \mathcal{I})$ for $\mathcal{I}=\ker(\O_Y \to i_*\O_X)$. \smallskip

Now we show existence of $I$. We consider the short exact sequence
\[
0 \to \mathcal{I} \to \O_Y \to i_*\O_X \to 0.
\]
Theorem~\ref{thm:huber-coherent-sheaves}(\ref{thm:huber-coherent-sheaves-1}) implies that $\mathcal{I}\cong \widetilde{I}$ for an ideal $I\subset A$. Moreover, Lemma~\ref{lemma:huber-closed-tate-topology} ensures that $\O_X(X)$ is a complete Tate ring and $\O_X(X)\simeq A/I$ topologically. This isomorphism induces a natural morphism $\phi\colon X \to \Spa(A/I, (A^+/I\cap A^+)^c)$ by Remark~\ref{grp-act-adic}. \smallskip

We first show that $\phi$ is a homeomorphism. Since both $X$ and $\Spa(A/I, (A^+/I\cap A^+)^c)$ are topologically closed subsets of $\Spa(A, A^+)$, it is sufficient to show that $\phi$ is a bijection. Now Remark~\ref{rmk:support-closed-immersion} and Corollary~\ref{cor:huber-top}
imply that both $X$ and $\Spa(A/I, (A^+/I\cap A^+)^c)$ can be topologically identified with the set 
\[
\{y\in Y \ | \ \mathcal{J}_y \not\simeq \O_{Y, y}\}. 
\]

Now we use Remark~\ref{rmk:huber-forgetful-conservative} to ensure that it suffices to show that 
\[
\phi^\#\colon \O_{\Spa(A/I, (A^+/I\cap A^+)^c)} \to \phi_*\O_X
\]
is an isomorphism of sheaves of topological rings. Since $i'\colon \Spa(A/I, (A^+/I\cap A^+)^c) \to \Spa(A, A^+)$ is topologically a closed immersion, it suffices to show that $\phi^\#$ is an isomorphism after applying $i'_*$, i.e. it suffices to show that the natural morphism
\[
i'_*\O_{\Spa(A/I, (A^+/I\cap A^+)^c)} \to i_*\O_X
\]
is an isomorphism of sheaves of topological rings. Corollary~\ref{cor:huber-top} implies that this is an algebraic isomorphism. We use Remark~\ref{rmk:open-mapping} and Lemma~\ref{lemma:huber-closed-tate-topology}\footnote{And an obvious observation that restriction of a closed immersion over an open subspace of the target is again a closed immersion.} to handle the topological aspect of the isomorphism.
\end{proof}

\begin{cor}\label{cor:closed-immersion-preserved} Let $i\colon X \to Y$ be a closed immersion of locally strongly noetherian adic spaces. Then
\begin{enumerate}
    \item\label{cor:closed-immersion-preserved-fiber-product} for any locally topologically finite type morphism $Z \to Y$, the fiber product $Z\times_Y X \to Z$ is a closed immersion,
    \item\label{cor:closed-immersion-preserved-composition} for any closed immersion $i'\colon Z \to X$, the composition $i \circ i' \colon Z \to Y$ is a closed immersion.
\end{enumerate}
\end{cor}
\begin{proof}
For the purpose of proving (\ref{cor:closed-immersion-preserved-fiber-product}), we may assume that $X$, $Y$ and $Z$ are strongly noetherian Tate affinoids. Then the result follows from Lemma~\ref{lemma:huber-closed-commutes-affinoid} and Corollary~\ref{cor:huber-closed-classification}. \smallskip

Similarly to prove (\ref{cor:closed-immersion-preserved-composition}), we may assume that $Y$ is a strongly noetherian Tate affinoid. Then the same holds for $X$ and $Z$ by Corollary~\ref{cor:huber-closed-classification}. It is clear that $Z \to Y$ is a homeomorphism  onto its closed image, and that $\O_Y \to (i\circ i')_*\O_Z$ is surjective. Thus, we only need to show that the kernel of that map is coherent. It suffices to show that $(i\circ i')_*\O_Z$ is coherent. Now we note that $i$ and $i'$ are finite by Corollary~\ref{cor:huber-closed-classification}, so $i\circ i'$ is also finite. Therefore, $(i\circ i')_*\O_Z$ is coherent by Corollary~\ref{cor:huber-coherent-sheaves}(\ref{cor:huber-coherent-sheaves-2}). 
\end{proof}

\begin{defn}\label{defn:huber-immersion} We say that a morphism $f\colon X \to Y$ of analytic adic spaces is a {\it locally closed immersion} if $f$ can be factored as $j\circ i$ where $i$ is a closed immersion and $j$ is an open immersion.
\end{defn}

\begin{lemma}\label{lemma:huber-composition-of-immersions} Let $f\colon X \to Y$ and $g\colon Y \to Z$ be locally closed immersions of locally strongly noetherian adic spaces. Then so is $g\circ f$.
\end{lemma}
\begin{proof}
We first deal with the case $f$ an open immersion and $g$ a closed immersion. In this case the topology on $Y$ is induced from $Z$, so there is an open adic subspace $U\subset Z$ such that $X=U\cap Z=g^{-1}(U)$. Therefore, we can factor $g \circ f$ as 
\[
X \xr{a} U \xr{b} Z.
\]
We note that $a$ is a closed immersion as the restriction of the closed immersion $g$ over $U\subset Z$, and $b$ is an open immersion by construction. Hence, $g\circ f$ is indeed a locally closed immersion. \smallskip

Now we consider the general case. In this case we can factor $f$ as $j\circ i$ with a closed immersion $i$ and an open immersion $j$. Similarly, we can factor $g=j'\circ i'$ with a closed immersion $j'$ and an open immersion $i'$. The argument above implies that the composition $i'\circ j$ can be rewritten as $j''\circ i''$ for a closed immersion immersion $i''$ and an open immersion $j''$. Therefore,

\[
g\circ f = j'\circ i' \circ j \circ i = j'\circ j'' \circ i'' \circ i = (j'\circ j'') \circ (i''\circ i').
\]
Now $i'' \circ i'$ is a closed immersion by Corollary~\ref{cor:closed-immersion-preserved}(\ref{cor:closed-immersion-preserved-composition}), and clearly $j'\circ j''$ is an open immersion. Therefore, $g\circ f$ is an immersion. 
\end{proof}

\begin{rmk} The order of the open and closed immersion in Definition~\ref{defn:huber-immersion} is needed to ensure that a composition of immersions is an immersion; the same happens over $\mathbf{C}$.
\end{rmk}

\begin{lemma}\label{immersion-and-closed} Let $f\colon X \to Y$ be a locally closed immersion of analytic adic spaces such that the image $f(X)$ is closed in $Y$. Then $f$ is a closed immersion.
\end{lemma}
\begin{proof}
We write $f$ as a composition
\[
X \xr{i} U \xr{j} Y
\]
of a closed immersion $i$ and an open immersion $j$. Since both $i$ and $j$ are topological embeddings, the same holds for $f$. Moreover, its image is closed in $Y$ by hypothesis on $f$. So we are left to show that $\mathcal{I}\coloneqq \ker(\O_Y \to f_*\O_X)$ is coherent, and $f^\#\colon \O_Y \to f_*\O_X$ is surjective. \smallskip

We use the open covering $Y=U\cup (Y\setminus f(X))$. We know that $\mathcal{I}|_{U}$ is coherent by assumption, and it is clear that $\mathcal{I}|_{Y\setminus f(X)}\simeq 0$ is coherent. Therefore, we conclude that $\mathcal{I}$ is coherent on $Y$. \smallskip

Now we show surjectivity of $f^\sharp$. We note that since $f$ is topologically a closed embedding, we conclude that  $(f_*\O_X)_y \cong 0$ for any $y\notin f(X)$. So it suffices to check surjectivity on stalks for $y\in f(X)\subset U$. But then $f^\#_y$ is identified with
\[
\O_{Y,y} \cong \O_{U,y} \twoheadrightarrow \O_{X,y}
\]
by the assumptions on $i$ and $j$. 
\end{proof}



\subsection{Separated Morphisms of Adic Spaces}

\begin{defn} We say that a locally topologically finite type morphism $f\colon X\to Y$ of locally strongly noetherian analytic adic spaces is {\it separated}, if the diagonal morphism $\Delta_{X/Y}\colon X \to X\times_Y X$ has closed image
\end{defn}

\begin{rmk} We assume that $f$ is locally topologically finite type to ensure the existence of the fiber product $X\times_Y X$. 
\end{rmk}

\begin{lemma}\label{lemma:diagonal-immersion} Let $f\colon X\to Y$ be a locally topologically finite morphism of locally strongly noetherian analytic adic spaces. Then $\Delta_{X/Y}\colon X\to X\times_Y X$ is a locally closed immersion.
\end{lemma}
\begin{proof}
We cover $Y$ by strongly noetherian Tate affinoids $(U_i)_{i\in I}$, and then we cover the pre-images $f^{-1}(U_i)$ by strongly noetherian Tate affinoids $(V_{i,j})_{j\in J_i}$. The construction of fiber products in \cite[Proposition 1.2.2 (a)]{H3} implies that $\cup_{i,j} V_{i,j} \times_{U_i} V_{i,j}$ is an open subset in $X\times_Y X$ that contains $\Delta_{X/Y}(X)$. Thus in order to show that $\Delta_{X/Y}$ is an immersion, it is suffices to show 
\[
\a\colon X\to \cup_{i,j} V_{i,j} \times_{U_i} V_{i,j}
\] 
is a closed immersion. \smallskip

Moreover, we note that $\a^{-1}(V_{i,j}\times_{U_i} V_{i,j})=V_{i,j}$ for any $i\in I, j\in J_i$. Since the notion of a closed immersion is easily seen to be local on the target, we conclude that it is enough to show that the diagonal morphism is a closed immersion for affinoid spaces $X=\Spa(B, B^+)$ and $Y=\Spa(A, A^+)$. But then the diagonal morphism $X\to X\times_Y X$ coincides with the morphism
\[
\Delta_{X/Y}\colon \Spa(B, B^+) \to \Spa\left(B\wdh{\otimes}_A B, \left(B\wdh{\otimes}_A B\right)^+\right)
\] 
induced by the natural ``multiplication morphism'' of Tate-Huber pairs
\[
m\colon \left(B\wdh{\otimes}_A B, \left(B\wdh{\otimes}_A B\right)^+\right) \to (B, B^+)
\]
with $\left(B\wdh{\otimes}_A B\right)^+$ being the integral closure of $B^+\wdh{\otimes}_{A^+} B^+$ inside $B\wdh{\otimes}_A B$. Then we see that $\Delta_{X/Y}$ is a closed immersion by Lemma~\ref{example-closed} with $I=\ker m$.
\end{proof}

\begin{cor}\label{sep-closed-immersion} Let $f\colon X \to Y$ be a locally topologically finite type, separated morphism of locally strongly noetherian analytic adic spaces. Then the diagonal morphism $\Delta_{X/Y}\colon X\to X\times_Y X$ is a closed immersion.
\end{cor}
\begin{proof}
This follows from Lemma~\ref{immersion-and-closed}~and~Lemma~\ref{lemma:diagonal-immersion}.
\end{proof}

\begin{cor}\label{lemma:intersection-affinoids} Let $f\colon X \to S$ be a locally topologically finite type, separated morphism of analytic adic spaces. Suppose that $S=\Spa(A, A^+)$ is a strongly noetherian Tate affinoid, and that $U$ and $V$ are two open affinoids in $X$. Then their intersection $U\cap V$ is also an open affinoid in $X$.
\end{cor}
\begin{proof}
Consider the following commutative Cartesian diagram:
\[
\begin{tikzcd}
U\cap V \arrow{d} \arrow{r}{i} & U\times_S V \arrow{d} \\
X \arrow{r}{\Delta_{X/S}} & X\times_S X.
\end{tikzcd}
\]
Since the map $\Delta_{X/S}$ is a closed immersion by Corollary~\ref{sep-closed-immersion}, so is its restriction $i$. Now we note that $U$ and $V$ are strongly noetherian Tate affinoids by Lemma~\ref{noeth-preserve} and Theorem~\ref{global-tft}. Then $U\times_S V$ is also a strongly noetherian Tate affinoid, so we can apply Corollary~\ref{cor:huber-closed-classification} to the map $i$ to conclude that $U\cap V$ is affinoid. 
\end{proof} 


\bibliography{quotients}

\end{document}